\documentclass[12pt]{amsart} 

\usepackage{amsmath}
\usepackage{amsthm}
\usepackage{amssymb}
\usepackage{amsfonts}
\usepackage{amscd}

\newcommand{\be}{{\bold e}}
\newcommand{\balpha}{{\boldsymbol \alpha}}

\newcommand{\bnu}{{\boldsymbol \nu}}
\newcommand{\bbeta}{{\boldsymbol \beta}}
\newcommand{\btheta}{{\boldsymbol \theta}}
\newcommand{\bt}{{\bold t}}
\newcommand{\bp}{{\bold p}}

\newcommand{\Z}{{\bold Z}}

\newcommand{\C}{{\bold C}}
\newcommand{\BP}{{\mathbf P}}

\newcommand{\cO}{{\mathcal O}}

\newcommand{\cC}{{\mathcal C}}

\newcommand{\cN}{{\mathcal N}}

\newcommand{\cP}{{\mathcal P}}
\newcommand{\cR}{{\mathcal R}}
\newcommand{\cS}{{\mathcal S}}
\newcommand{\cJ}{{\mathcal J}}

\newcommand\rank{\mathop{\rm rank}\nolimits}
\newcommand\im{\mathop{\rm Im}\nolimits}
\newcommand\coker{\mathop{\rm coker}\nolimits}
\newcommand{\Aut}{\mathop{\rm Aut}\nolimits}
\newcommand{\Tr}{\mathop{\rm Tr}\nolimits}
\newcommand\Hom{\mathop{\rm Hom}\nolimits}

\newcommand\End{\mathop{\rm End}\nolimits}
\newcommand\Pic{\mathop{\rm Pic}\nolimits}
\newcommand\Spec{\mathop{\rm Spec}\nolimits}
\newcommand\Hilb{\mathop{\rm Hilb}\nolimits}
\newcommand\RH{\mathop{\bf RH}\nolimits}
\newcommand{\length}{\mathop{\rm length}\nolimits}
\newcommand{\res}{\mathop{\sf res}\nolimits}

\newcommand{\sch}{\mathrm{Sch}}
\newcommand{\sets}{\mathrm{Sets}}
\newcommand\lra{\longrightarrow}
\newcommand\ra{\rightarrow}

\newtheorem{Theorem}{Theorem}[section]
\newtheorem{Lemma}{Lemma}[section]
\newtheorem{Remark}{Remark}[section]

\newtheorem{Proposition}{Proposition}[section]

\newtheorem{Definition}{Definition}[section]

\begin{document}

\title[Moduli of unramified irregular singular connections]
{Moduli of unramified irregular singular parabolic connections
on a smooth projective curve.}

\thanks{This research was partly supported by JSPS Grant-in-Aid for Scientific Research (S)19104002, 24224001,  challenging Exploratory Research  23654010 and for Young Scientists (B) 22740014}
\author{Michi-aki Inaba}
\author{Masa-Hiko Saito}
\dedicatory{In memory of Professor Masaki Maruyama}
\address{Department of Mathematics, 
Graduate school of Science, Kyoto University,  
Kyoto,  606-8502, Japan}
\email{inaba@math.kyoto-u.ac.jp}
\address{Department of Mathematics, Graduate School  of Science, 
Kobe University, Kobe, Rokko, 657-8501, Japan}
\email{mhsaito@math.kobe-u.ac.jp}

\begin{abstract}
In this paper we construct a coarse moduli scheme of
stable unramified irregular singular parabolic connections on a 
smooth projective curve and prove that the constructed moduli space
is smooth and has a symplectic structure. 
Moreover we will construct the moduli space of generalized monodromy data 
coming from topological monodromies, formal monodromies, links and Stokes data associated to the generic irregular connections.   We will prove that 
for a generic choice of generalized local exponents, 
the generalized Riemann-Hilbert correspondence from the moduli space of 
the connections to the moduli space of the associated 
generalized monodromy data gives an analytic isomorphism. 
This shows that differential systems 
arising from (generalized) isomonodromic deformations of corresponding 
unramified irregular singular parabolic connections admit 
geometric Painlev\'e property as in the regular singular cases 
proved generally in \cite{Inaba-1}. 
\end{abstract}

\maketitle

\section*{Introduction}

Let $m,l$ be positive integers and
$\nu$ be an element of
$\mathbf{C}dw/w^{lm-l+1}+\cdots+\mathbf{C}dw/w$.
We denote the $\mathbf{C}[[w]]$-module $\mathbf{C}[[w]]^{\oplus r}$
with the connection
\begin{gather*}
 \mathbf{C}[[w]]^{\oplus r}\longrightarrow \mathbf{C}[[w]]^{\oplus r}
 \otimes\genfrac{}{}{}{}{dw}{w^{lm-l+1}} \\
 ae_j \mapsto da e_j+\nu a e_j+w^{-1}dwe_{j-1}
\end{gather*}
by $V(\nu,r)$.
Here $e_1,\ldots,e_r$ is the canonical basis of $\mathbf{C}[[w]]^{\oplus r}$
and $e_0=0$.

We have the following fundamental theorem:
\begin{Theorem}[Hukuhara-Turrittin]\label{thm-H-T}
 Let $V$ be a free $\mathbf{C}[[z]]$-module of rank $r$
 and $\nabla:V\rightarrow V\otimes dz/z^m$
 be a connection.
 Then there are positive integers $l,s,r_1,\ldots,r_s$ such that for a variable
 $w$ with $w^l=z$, there exist
 $\nu_1,\ldots,\nu_s\in
 \mathbf{C}dw/w^{ml-l+1}+\cdots+\mathbf{C}dw/w$
 such that
 \[
  (V,\nabla)\otimes_{\mathbf{C}[[z]]}\mathbf{C}((w))\cong
  \left(V(\nu_1,r_1)\oplus\cdots\oplus V(\nu_s,r_s)\right)
  \otimes_{\mathbf{C}[[w]]}\mathbf{C}((w)).
 \]
\end{Theorem}

For the proof of Theorem \ref{thm-H-T},
see [\cite{Sibuya}, Theorem 6.8.1] for example.
Note that $\nu_1,\ldots,\nu_s$ in Theorem \ref{thm-H-T}
are unique modulo $\mathbf{Z}dw/w$.
So we can take $\nu_1,\ldots,\nu_s$ as
invariants of a connection.
In this paper we consider only the case $l=1$.

Let $C$ be a smooth projective irreducible curve over $\mathbf{C}$
of genus $g$,
$t_1,\ldots,t_n$ be distinct points of $C$
and $m_1,\ldots,m_n$ be positive integers.
Put $D:=\sum_{i=1}^nm_it_i$.
Take $d\in\mathbf{Z}$ and
$\bnu=(\nu^{(i)}_j)^{1\leq i\leq n}_{0\leq j\leq r-1}$
such that
$\nu^{(i)}_j\in\mathbf{C}dz_i/z_i^{m_i}+\cdots+\mathbf{C}dz_i/z_i$
and that
$d+\sum_{i=1}^n\sum_{j=0}^{r-1}\res_{t_i}(\nu^{(i)}_j)=0$,
where $z_i$ is a generator of the maximal ideal of ${\mathcal O}_{C,t_i}$.
Let $N_i$ be positive integers such that $N_i\geq m_i$
for $i=1,\ldots,n$ and set
$N_it_i:=\Spec\left({\mathcal O}_{C,t_i}/(z_i^{N_i})\right)$.
$(E,\nabla,\{l^{(i)}_j\})$ is said to be an irregular singular $\bnu$-parabolic connection
of parabolic depth $(N_i)$
if $E$ is a vector bundle on $C$ of rank $r$ and degree $d$,
$\nabla:E\rightarrow E\otimes\Omega^1_C(D)$ is a connection,
$E|_{N_it_i}=l^{(i)}_0\supset l^{(i)}_1\supset\cdots\supset l^{(i)}_{r-1}
\supset l^{(i)}_r=0$
is a filtration such that $l^{(i)}_j/l^{(i)}_{j+1}\cong{\mathcal O}_{N_it_i}$,
$\nabla|_{N_it_i}(l^{(i)}_j)\subset l^{(i)}_j\otimes\Omega^1_C(D)$
for any $i,j$ and for the induced morphism
$\overline{\nabla^{(i)}_j}:l^{(i)}_j/l^{(i)}_{j+1}\rightarrow
(l^{(i)}_j/l^{(i)}_{j+1})\otimes\Omega^1_C(D)$,
$\left(\overline{\nabla^{(i)}_j}-\nu^{(i)}_j\mathrm{id}\right)
(l^{(i)}_j/l^{(i)}_{j+1})$
is contained in the image of
$(l^{(i)}_j/l^{(i)}_{j+1})\otimes\Omega^1_C\rightarrow
(l^{(i)}_j/l^{(i)}_{j+1})\otimes\Omega^1_C(D)$
for any $i,j$.
We can define $\balpha$-stability for
$\bnu$-parabolic connection.
See Definition \ref{def-stability} for precise definition of $\balpha$-stability.
We first show in Theorem \ref{thm-moduli-exists} that
there is a coarse moduli scheme $M^{\balpha}_{D/C}(r,d,(N_i))_{\bnu}$
of $\balpha$-stable $\bnu$-parabolic connections
of parabolic depth $(N_i)$.
The main theorem of this paper is Theorem \ref{smoothness-thm}
and Theorem \ref{symplectic-form},
which state that the moduli space
$M^{\balpha}_{D/C}(r,d,(m_i))_{\bnu}$
of $\balpha$-stable $\bnu$-parabolic
connections $(E,\nabla,\{l^{(i)}_j\})$ of parabolic depth $(m_i)$
is smooth and has a symplectic structure.

However, there is a serious example (Remark \ref{counter-example})
which states that for special $\bnu$,
there is a member
$(E,\nabla,\{l^{(i)}_j\})\in M^{\balpha}_{D/C}(r,d,(m_i))_{\bnu}$
such that the invariants $\nu_1,\ldots,\nu_s$ for
$(E,\nabla)\otimes\hat{\mathcal O}_{C,t_i}$ given in
Theorem \ref{thm-H-T} are different from the data
$\nu^{(i)}_0,\ldots,\nu^{(i)}_{r-1}$ given by $\bnu$.
So $M^{\balpha}_{D/C}(r,d,(m_i))_{\bnu}$
does not seem a good moduli space at a glance.
On the other hand, assume that $N_i\geq r^2m_i$ for any $i$ and
$0\leq \mathrm{Re}\left(\res_{t_i}(\nu^{(i)}_j)\right)<1$ for any $i,j$.
Then Proposition \ref{prop-deeper} states that
for any member
$(E,\nabla,\{l^{(i)}_j\})\in M^{\balpha}_{D/C}(r,d,(N_i))_{\bnu}$,
the data $\nu_1,\ldots,\nu_s$ for
$(E,\nabla)\otimes\hat{\mathcal O}_{C,t_i}$
given in Theorem \ref{thm-H-T}
are the same as the data
$\nu^{(i)}_0,\ldots,\nu^{(i)}_{r-1}$
given by $\bnu$.
So it seems that
$M^{\balpha}_{D/C}(r,d,(N_i))_{\bnu}$
is a good moduli space.
However, Remark \ref{not-smooth} states that
the moduli space
$M^{\balpha}_{D/C}(r,d,(N_i))_{\bnu}$ is not smooth
for special $\bnu$.
So we cannot define isomonodromic deformations on
the moduli space $M^{\balpha}_{D/C}(r,d,(N_i))$.
After all the authors believe that
the moduli space $M^{\balpha}_{D/C}(r,d,(m_i))$
of $\balpha$-stable parabolic connections of parabolic depth $(m_i)$
is a correct moduli space, although $\bnu$ does not necessarily
reflect the invariants given in Hukuhara-Turrittin theorem.

After we construct the good moduli space  
$M^{\balpha}_{D/C}(r,d,(m_i))_{\bnu}$ of the $\balpha$-stable 
parabolic connections, we will investigate the 
Riemann-Hilbert correspondences for these moduli spaces and 
define the generalized  isomonodromic flows or 
isomonodromic differential systems associated to them. 
For that purpose, it is necessary to construct a good 
moduli space of the generalized monodromy data for the 
parabolic $\bnu$-connection $(E,\nabla,\{l^{(i)}_j\}) \in M^{\balpha}_{D/C}(r,d,(m_i))_{\bnu}$. 
However, for that purpose, we  
should fix the types of decompositions in the Hukuhara-Turrittin 
formal types at all irregular singular points.  However 
for some special $\bnu$, we can not recover these formal types (Remark \ref{counter-example}).  
So in this paper, we will restrict ourselves to the case 
when the local exponents $\bnu$ is generic (cf. Definition \ref{def:generic}). 
In this case, we can also construct the coarse moduli scheme  
$\cR(\bnu)$ of the data consisting of 
Stokes data, links and global topological monodromy 
representation of $\pi_1 (C \setminus \{ t_1, \cdots, t_n \})$.  
Let  us denote by $\bnu_{res}$  the residue part of $\bnu$ and by 
$\bp= \{ \widehat{\gamma_{i}} \}= \be(\bnu_{res})$
its exponential.   
Under the assumption that  $\bnu$ is generic, non-resonant and irreducible, 
 we can see that the moduli space 
$\cR(\bnu)$ is a nonsingular affine scheme.  Moreover for a fixed generic
$\bnu$, we can define the Riemann-Hilbert correspondence 
$\RH_{(D/C), \bnu}:M^{\balpha}_{D/C}(r,d,(m_i))_{\bnu} \lra \cR(\bnu)$, and 
 in Theorem \ref{thm:Riemann-Hilbert}, 
we prove that $\RH_{(D/C), \bnu}$ is an analytic isomorphism under the assumption that 
$\bnu$ is generic, non-resonant and irreducible. In \S 6, we will define 
 the isomonodromic differential systems induced by the family of 
 the Riemann-Hilbert correspondences and show that 
 the corresponding differential systems  has geometric Painlev\'e 
 property when $\bnu$ is generic,  simple, non-resonant and 
 irreducible(cf. Theorem \ref{thm:GP}).  Then as a corollary we can obtain the 
geometric Painlev\'e property of 5 types of classical 
Painlev\'e equations listed below when $\bp= \be(\bnu_{res})$ 
is non-resonant and irreducible.  (Note that if rank $E$ =$2$, $\bnu$ are always 
simple).  
\begin{equation}\label{eq:URP}
P_{VI}(D^{(1)}_4)_{\bp}, P_{V}(D^{(1)}_5)_{\bp}, P_{III}(D^{(1)}_6)_{\bp}, 
P_{IV}(E^{(1)}_6)_{\bp},  P_{II}(E^{(1)}_7)_{\bp}
\end{equation} 
For $P_{VI}(D^{(1)}_4)_{\bp}$, we showed the geometric 
Painlev\'e property for any $\bp$ (\cite{IIS-1}, \cite{IIS-2}, 
\cite{Inaba-1}). More generally, the geometric Painlev\'e 
property for isomonodromic differential systems associated 
to the regular singular parabolic connections for any $\bp$ 
was proved completely  in \cite{Inaba-1}.

The rough plan of this paper is as follows.  
In \S 1, we will prepare some results on the 
formal parabolic connections  and their reductions to the 
finite orders. In \S 2, we will construct the coarse moduli 
scheme $M^{\balpha}_{D/C}(r,d,(N_i))_{\bnu}$ for $N_i \geq m_i$ 
as a quasi-projective scheme 
and show that $M^{\balpha}_{D/C}(r,d,(m_i))_{\bnu}$ is 
smooth for any $\bnu$.  In \S 3,  we will show  the existence of 
the smooth family of the moduli spaces of parabolic 
connections over the space of generalized exponents 
when we  also vary the divisor  $D= \sum_{i=1}^n m_i t_i$ 
in a product of Hilbert schemes of points (cf.   
Theorem \ref{relative-smooth}).  Theorem \ref{relative-smooth} 
seems important from the view point of confluence process of 
singular points.  
In \S 4, we will show  the existence of 
the relative symplectic form $\omega$ 
on the family of moduli spaces of parabolic connections 
parametrized by $\bnu$.  We will use Theorem \ref{relative-smooth}
 to reduce the proof of 
 the closedness $d\omega=0$ to the case of 
 regular singular cases in \cite{Inaba-1}. 
 In \S 5, we will review on the generalized 
 monodromy data and construct the moduli space of 
 them when $\bnu$ is generic. Moreover we define the 
 Riemann-Hilbert correspondence and show that 
 it gives an analytic isomorphism for generic, non-resonant 
 and irreducible $\bnu$.  In \S 6, fixing a non-resonant and irreducible $\bnu_{res}$,  
we will define the family of  Riemann-Hilbert correspondences and 
define the isomonodromic flows on the phase space 
 $\pi_{2,\bnu_{res}}:M^{\balpha}_{D/{\mathcal C}/T^{\circ, s}_{\bnu_{res}} }
\lra T^{\circ, s}_{\bnu_{res}}$
 which is the family of moduli spaces of  $\balpha$-stable parabolic connections over a certain space $T^{\circ,s}_{\bnu_{res}}$ of parameters including generic, simple exponents $\bnu$ with the fixed 
residue part $\bnu_{res}$ (see (\ref{eq:space})).  
The isomonodromic flows define an isomonodromic foliation or 
an isomonodromic differential system on the phase space   and 
its geometric Painlev\'e property follows easily 
from the definition based on Theorem \ref{thm:Riemann-Hilbert}.  
The geometric Painlev\'e property gives a  complete and clear proof of 
the analytic Painlev\'e property for the isomonodromic differential systems 
with non-resonant and irreducible exponents $\bnu_{res}$ or $\bp= \be(\bnu_{\res})$. 
 
As explained in \cite{IISA}, it is important 
to construct the fibers of the  phase space of the isomonodromic 
differential system as smooth algebraic schemes. 
One can use affine algebraic coordinates of the fibers over an
 open set of parameter spaces 
to write down the differential systems explicitly. Then 
the differential systems satisfy the analytic 
Painlev\'e property which easily follows from the geometric 
Painlev\'e property.  

We should mention that Malgrange \cite{Malgrange}, \cite{Malgrange2} 
and Miwa \cite{Miwa} 
gave proofs of the analytic Painlev\'e property for isomonodromic 
differential systems for irregular connections on $\BP^1$.
However, in order to give a complete proof of 
the geometric Painlev\'e property, we believe that 
our algebro-geometric construction of the family of the moduli spaces 
of connections is indispensable. (See also \cite{IIS-1}, \cite{IIS-2} and \cite{Inaba-1} for the regular singular cases).
  
Bremer and Sage studied in \cite{B-S1} the moduli space
of irregular singular connections on $\mathbf{P}^1$.
They consider also the ramified case.
However, they assumed that the bundle $V$ is trivial, which means that 
their moduli space only covers a Zariski open set of our moduli space 
which is not enough to prove the geometric Painlev\'e property even for 
generic unramified cases. See Remark \ref{rem:jumping}.

The second author thanks Marius van der Put for the hospitality and 
many discussions on the generalized monodromy data of irregular singularities 
during his visits in Groningen.  Both authors thank to the referee for pointing out 
mistakes about the multiplicity of singular lines in the first version of this paper and 
other useful suggestions.  

\section{Preliminary}

As a corollary of Theorem \ref{thm-H-T},
we obtain the following Proposition:
\begin{Proposition} \label{prop:filtration}
  Let $V$ be a free $\mathbf{C}[[z]]$-module of rank $r$
 and $\nabla:V\rightarrow V\otimes dz/z^m$
 be a connection.
 Then there is a positive integer $l$ such that for a variable
 $w$ with $w^l=z$, there exist
 $\nu_0,\ldots,\nu_{r-1}\in\mathbf{C}dw/w^{lm-l+1}+\cdots+\mathbf{C}dw/w$
 and a filtration
 $V\otimes\mathbf{C}[[w]]=V_0\supset V_1\supset V_2\supset\cdots\supset V_{r-1}\supset V_r=0$
 by subbundles such that $\nabla(V_j)\subset V_j\otimes dw/w^{lm-l+1}$
 and $V_j/V_{j+1}\cong V(\nu_j,1)$
 for any $j=0,1,\ldots,r-1$.
\end{Proposition}

\begin{proof}
We prove the proposition by induction on $r$.
For $r=1$, we take a basis $e$ of $V$.
Then we have
$\nabla(e)=\nu e dz$
for some $\nu\in\mathbf{C}[[z]]z^{-m}$.
We can write
$\nu=\sum_{j\geq -m}a_jz^j$.
We put
$\nu_0:=\sum_{j\geq 0}a_jz^j$ and
$\mu:=\int\nu_0=\sum_{j\geq 0}(j+1)^{-1}a_jz^{j+1}$.
Then we have $\exp(-\mu)\in\mathbf{C}[[z]]$ and
\[
 \genfrac{}{}{}{}{d}{dz}\exp(-\mu)=-\exp(-\mu)\genfrac{}{}{}{}{d\mu}{dz}
 =-\exp(-\mu)\nu_0.
\]
We put $e':=\exp(-\mu)e$.
Then $e'$ is a basis of $V$ and
\begin{align*}
 \nabla(e')=\nabla(\exp(-\mu)e)&=
 \exp(-\mu)\nabla(e)+\genfrac{}{}{}{}{d\exp(-\mu)}{dz}edz \\
 &=\exp(-\mu)\nu e dz -\exp(-\mu)\nu_0 edz \\
 &=(\nu-\nu_0)\exp(-\mu)edz=(\nu-\nu_0)dz e'.
\end{align*}
Hence we have $V\cong V((\nu-\nu_0)dz,1)$.

Now assume that $r>1$.
By Theorem \ref{thm-H-T},
there is a positive integer $l$ such that for a variable
$w$ with $w^l=z$, there exist
$\mu_1,\ldots,\mu_s\in\mathbf{C}dw/w^m+\cdots\mathbf{C}dw/w$,
positive integers $r_1,\ldots,r_s$
and an isomorphism 
\[
 \varphi:V\otimes_{\mathbf{C}[[z]]}\mathbf{C}((w))\stackrel{\sim}\longrightarrow
 \left(V(\mu_1,r_1)\otimes_{\mathbf{C}[[w]]}\mathbf{C}((w))\right)\oplus
 \cdots\oplus\left(V(\mu_s,r_s)\otimes_{\mathbf{C}[[w]]}\mathbf{C}((w))\right).
\]
We can take an element $e_{r-1}\in \varphi^{-1}(V(\mu_s,r_s))$ such that
$\nabla(e_{r-1})=\mu_se_{r-1}$.
Let $m_{r-1}$ be the smallest integer such that
$w^{m_{r-1}}e_{r-1}\in V\otimes_{\mathbf{C}[[z]]}\mathbf{C}[[w]]$.
Then we have
\[
 \nabla(w^{m_{r-1}}e_{r-1})=m_{r-1}w^{m_{r-1}-1}dwe_{r-1}+\mu_{s}w^{m_{r-1}}e_{r-1}
 =(m_{r-1}w^{-1}dw+\mu_{s})w^{m_{r-1}}e_{r-1}
\]
If we put $V_{r-1}:=\mathbf{C}[[w]]w^{m_{r-1}}e_{r-1}$ and
$\mu'_{r-1}:=m_{r-1}w^{-1}dw+\mu_{s}$, then
$V_{r-1}\cong V(\nu_{r-1},1)$ and
$W_{r-1}:=\left(V\otimes_{\mathbf{C}[[z]]}\mathbf{C}[[w]]\right)/V_{r-1}$
is a torsion free $\mathbf{C}[[w]]$-module.
So $W_{r-1}$ is a free $\mathbf{C}[[w]]$-module of rank $r-1$ and
$\nabla$ induces a connection
\[
 \bar{\nabla}:W_{r-1}\longrightarrow W_{r-1}\otimes \genfrac{}{}{}{}{dw}{w^{ml-l+1}}.
\]
Then by the induction assumption, there is a filtration
$W_{r-1}=\bar{V}_0\supset\bar{V}_1\supset\cdots\supset\bar{V}_{r-2}\supset
\bar{V}_{r-1}=0$
by subbundles such that
$\bar{\nabla}(\bar{V}_j)\subset \bar{V}_j\otimes dw/w^{lm-l+1}$
and
$\bar{V}_j/\bar{V}_{j+1}\cong V(\mu'_j,1)$ for $j=0,\ldots,r-2$
for some $\mu'_j\in\mathbf{C}dw/w^{ml-l+1}+\cdots+\mathbf{C}dw/w$
($0\leq j\leq r-2$).
Let $V_j$ be the pull back of $\bar{V}_j$ by the homomorphism
$V\otimes_{\mathbf{C}[[z]]}\mathbf{C}[[w]]\rightarrow W_{r-1}$
($0\leq j\leq r-1$).
Then $\nabla(V_j)\subset V_j\otimes dw/w^{lm-l+1}$ and
$V_j/V_{j+1}\cong V(\mu'_j,1)$ for $0\leq j\leq r-1$.
\end{proof}

\begin{Remark}\rm
 By the proof of Proposition \ref{prop:filtration}, we can easily see that
 $\{\nu_j\;\mathrm{mod} \,\mathbf{Z}dw/w\}$
 in Proposition \ref{prop:filtration} are nothing but the invariants
 in Hukuhara-Turrittin Theorem (Theorem \ref{thm-H-T}).
 We should remark that we can not give a decomposition
 $(V,\nabla)\otimes_{\mathbf{C}[[z]]}\mathbf{C}[[w]]\cong\bigoplus_{j=0}^{r-1}V(\nu_j,1)$
 even if $\nu_j$ modulo $\mathbf{Z}dw/w$ are mutually distinct.
\end{Remark}

\begin{Remark}\label{counter-example}\rm
Unfortunately, we can not recover $\nu_0,\ldots,\nu_{r-1}$
from $\nabla\otimes\mathbf{C}[[w]]/(w^{ml-l+1})$.
Indeed consider the connection
$\nabla:\mathbf{C}[[z]]^{\oplus 2}\rightarrow
\mathbf{C}[[z]]^{\oplus 2}\otimes dz/z^6$
given by
\[
 \nabla=d+
 \begin{pmatrix}
  z^{-6}dz+z^{-2}dz & z^{-4}dz \\
  0 & z^{-6}dz-z^{-2}dz
 \end{pmatrix}.
\]
Let
\[
 \nabla\otimes\mathbf{C}[[z]]/(z^6):
 (\mathbf{C}[[z]]/(z^6))^{\oplus 2}\rightarrow
 (\mathbf{C}[[z]]/(z^6))^{\oplus 2}\otimes
 \genfrac{}{}{}{}{dz}{z^6}
\]
be the induced $\mathbf{C}[[z]]/(z^6)$-homomorphism.
Then $\nabla\otimes\mathbf{C}[[z]]/(z^6)$ can be given by the matrix
\[
 A=
  \begin{pmatrix}
  z^{-6}dz+z^{-2}dz & z^{-4}dz \\
  0 & z^{-6}dz-z^{-2}dz
 \end{pmatrix}
\]
with respect to the basis
\[
 \genfrac{(}{)}{0pt}{}{1}{0}, \genfrac{(}{)}{0pt}{}{0}{1}
\]
of $(\mathbf{C}[[z]]/(z^6))^{\oplus 2}$.
So the ``eigenvalues'' of $\nabla\otimes\mathbf{C}[z]/(z^6)$
with respect to this basis  are
$z^{-6}dz+z^{-2}dz, z^{-6}dz-z^{-2}dz$.

On the other hand,
take the basis
\[
 \genfrac{(}{)}{0pt}{}{1+z^4}{-z^2}, \genfrac{(}{)}{0pt}{}{-z^2}{1+z^4}
\]
of $(\mathbf{C}[[z]]/(z^6))^{\oplus 2}$.
Then we have
\begin{align*}
 \left(\nabla\otimes\mathbf{C}[[z]]/(z^6)\right)
 \genfrac{(}{)}{0pt}{}{1+z^4}{-z^2}
 &=
 \begin{pmatrix}
  z^{-6}dz+z^{-2}dz & z^{-4}dz \\
  0 & z^{-6}dz-z^{-2}dz
 \end{pmatrix}
 \genfrac{(}{)}{0pt}{}{1+z^4}{-z^2} \\
 &=
 \genfrac{(}{)}{0pt}{}{z^{-6}dz+z^{-2}dz}{-z^{-4}dz}
 =
 z^{-6}dz\genfrac{(}{)}{0pt}{}{1+z^4}{-z^2}
\end{align*}
and
\begin{align*}
 \left(\nabla\otimes\mathbf{C}[[z]]/(z^6)\right)
 \genfrac{(}{)}{0pt}{}{-z^2}{1+z^4}
 &=
 \begin{pmatrix}
  z^{-6}dz+z^{-2}dz & z^{-4}dz \\
  0 & z^{-6}dz-z^{-2}dz
 \end{pmatrix}
 \genfrac{(}{)}{0pt}{}{-z^2}{1+z^4} \\
 &=
 \genfrac{(}{)}{0pt}{}{0}{z^{-6}dz}
 =z^{-4}dz\genfrac{(}{)}{0pt}{}{1+z^4}{-z^2}+
 z^{-6}dz\genfrac{(}{)}{0pt}{}{-z^2}{1+z^4}
\end{align*}
Thus the representation matrix of $\nabla\otimes\mathbf{C}[[z]]/(z^6)$
with respect to the basis
\[
 \genfrac{(}{)}{0pt}{}{1+z^4}{-z^2},
 \genfrac{(}{)}{0pt}{}{-z^2}{1+z^4},
\]
is given by
\[
 \begin{pmatrix}
  z^{-6}dz & z^{-4}dz \\
  0 & z^{-6}dz
 \end{pmatrix}.
\]
So the ``eigenvalues'' of $\nabla\otimes\mathbf{C}[z]/(z^6)$
with respect to this basis are $z^{-6}dz,z^{-6}dz$.
Thus we conclude that the ``eigenvalues'' of
$\nabla\otimes\mathbf{C}[z]/(z^6)$
can not be well-defined.
In other words, the eigenvalues $\nu_1,\ldots,\nu_s$ given
in Hukuhara-Turrittin theorem (Theorem \ref{thm-H-T})
can not be recovered from $\nabla\otimes\mathbf{C}[z]/(z^6)$.
\end{Remark}

On the other hand, we have the following proposition,
which will be possible to improve more generally according to the
referee's valuable comment.
\begin{Proposition}\label{prop-deeper}
Let $V,W$ be free $\mathbf{C}[[z]]/(z^{r^2m})$-modules of rank $r$
with connections
\begin{align*}
 \nabla^V&:V\longrightarrow V\otimes\genfrac{}{}{}{}{dz}{z^m} \\
 \nabla^W&:W\longrightarrow W\otimes\genfrac{}{}{}{}{dz}{z^m}.
\end{align*}
and filtrations
\begin{gather*}
 V=V_0\supset V_1\supset \cdots V_{r-1}\supset V_r=0 \\
 W=W_0\supset W_1\supset\cdots W_{r-1}\supset W_r=0
\end{gather*}
such that $V_i/V_{i+1}\cong\mathbf{C}[[z]]/(z^{r^2m})$,
$W_i/W_{i+1}\cong\mathbf{C}[[z]]/(z^{r^2m})$ and that
$\nabla^V(V_i)\subset V_i\otimes dz/z^m$,
$\nabla^W(W_i)\subset W_i\otimes dz/z^m$
for any $i$.
Let $\nabla^V_i:V_i/V_{i+1}\rightarrow (V_i/V_{i+1})\otimes dz/z^m$
and $\nabla^W_i:W_i/W_{i+1}\rightarrow (W_i/W_{i+1})\otimes dz/z^m$
be the morphisms induced by $\nabla^V$ and $\nabla^W$, respectively.
Choose a basis $e^V_i$ of $V_i/V_{i+1}$ (resp.\ $e^W_i$ of $W_i/W_{i+1}$)
such that $\nabla^V_i(e^V_i)=\nu^V_ie^V_i$ and
$\nabla^W_i(e^W_i)=\nu^W_ie^W_i$ with
\begin{align*}
 \nu^V_i&=\left(a^{(i)}_{-m}z^{-m}+a^{(i)}_{-m+1}z^{-m+1}
 +\cdots+a^{(i)}_{-1}z^{-1}\right)dz \\
 \nu^W_i&=\left(b^{(i)}_{-m}z^{-m}+b^{(i)}_{-m+1}z^{-m+1}
 +\cdots+b^{(i)}_{-1}z^{-1}\right)dz
\end{align*}
Assume that $0\leq\mathrm{Re}(a^{(i)}_{-1})<1$ and
$0\leq\mathrm{Re}(b^{(i)}_{-1})<1$
for any $i$.
If there is an isomorphism $\varphi:V\stackrel{\sim}\longrightarrow W$
of $\mathbf{C}[[z]]/(z^{r^2m})$-modules such that
$\nabla^W\circ\varphi=(\varphi\otimes\mathrm{id})\circ\nabla^V$,
then there is a permutation $\sigma\in S_r$ such that
$\nu^V_i=\nu^W_{\sigma(i)}$ for any $i=0,\ldots,r-1$.
\end{Proposition}

\begin{proof}
We prove the Proposition by induction on $r$.
Assume that $r=1$.
We can write $\varphi(e^V_0)=ce^W_0$ 
with $c\in(\mathbf{C}[[z]]/(z^{m}))^{\times}$.
Then we have
\begin{align*}
 (dc)e^W_0+c\nu^W_0e^W_0 & =\nabla^W(ce^W_0)
 =\nabla^W\varphi(e^V_0)=(\varphi\otimes\mathrm{id})\nabla^V(e^V_0)
 =(\varphi\otimes\mathrm{id})(\nu^V_0e^V_0)
 =\nu^V_0\varphi(e^V_0) \\
& =c\nu^V_0e^W_0.
\end{align*}
So we have
\[
 dc=c(\nu^V_0-\nu^W_0).
\]
If $\nu^V_0\neq\nu^W_0$, we can write
\[
 \nu^V_0-\nu^W_0=\alpha_{-n}z^{-n}dz+\alpha_{-n+1}z^{-n+1}dz
 +\cdots+\alpha_{-1}z^{-1}dz
\]
with $n\geq 1$ and $\alpha_{-n}\neq 0$.
If we put $c=c_0+c_1z+c_2z^2+\cdots+c_{m-1}z^{m-1}$
with each $c_j\in\mathbf{C}$,
then we have $c_0\neq 0$.
So we have
\begin{align*}
 dc&=c(\nu^V_0-\nu^W_0) \\
 &=(c_0+c_1z+\cdots+c_{m-1}z^{m-1})
 (\alpha_{-n}z^{-n}dz+\cdots\alpha_{-1}z^{-1}dz) \\
 &=c_0\alpha_{-n}z^{-n}dz+\sum_{j>-n}\beta_jz^jdz
 \notin \mathbf{C}[[z]]/(z^{m})\otimes dz,
\end{align*}
which is a contradiction.
Thus we have $\nu^V_0=\nu^W_0$.

Next assume that $r>1$.
Consider the composite
\[
 \psi:V_{r-1}\hookrightarrow V\stackrel{\varphi}\longrightarrow
 W \longrightarrow W/W_1.
\]
There exists an element $c\in\mathbf{C}[[z]]/(z^{r^2m})$
such that $\psi(e^V_{r-1})=ce^W_0$ in $W/W_1$.
Then we have
\begin{align*}
 c\nu^V_{r-1}e^W_0 & =
 (\psi\otimes\mathrm{id})(\nu^V_{r-1}e^V_{r-1})
 =(\psi\otimes\mathrm{id})\circ\nabla^V(e^V_{r-1})
  =\nabla^W\circ\psi(e^V_{r-1}) 
  =\nabla^W(ce^W_0) \\
& =(dc)e^W_0+c\nu^W_0e^W_0
\end{align*}
and so we have
\[
 dc=c(\nu^V_{r-1}-\nu^W_0).
\]

If $\psi$ is an isomorphism, then we have
$\nu^V_{r-1}=\nu^W_0$ and the composite
\[
 V\stackrel{\varphi}\longrightarrow W\longrightarrow
 W/W_1\stackrel{\psi^{-1}}\longrightarrow V_{r-1}
\]
gives a splitting of the exact sequence
\[
 0\longrightarrow V_{r-1} \longrightarrow V 
 \longrightarrow V/V_{r-1}\longrightarrow 0.
\]
So we have $V=V_{r-1}\oplus V/V_{r-1}$.
Similarly we have a splitting $W=W/W_1\oplus W_1$
and we have an isomorphism 
$V/V_{r-1}\cong W_1$ which is compatible with the connections.
So we obtain an isomorphism
$(V/V_{r-1})\otimes\mathbf{C}[[z]]/(z^{(r-1)^2m})\stackrel{\sim}
\rightarrow W_1\otimes\mathbf{C}[[z]]/(z^{(r-1)^2m})$.
By induction hypotheses, there is a permutation $\sigma\in S_r$
such that $\sigma(r-1)=0$ and
$\nu^V_i=\nu^W_{\sigma(i)}$
for any $i$.

So assume that $\psi$ is not an isomorphism.
Then we can write $c=c_kz^k+c_{k+1}z^{k+1}+\cdots+c_{r^2m}z^{r^2m}$
with $c_k\neq 0$ and $k>0$.
If $k\leq (r^2-1)m$, we have
\[
 dc=kc_kz^{k-1}dz+(k+1)c_{k+1}z^kdz+\cdots+(r^2m-1)c_{r^2m-1}z^{r^2m-2}dz
 \neq 0.
\]
So we have $\nu^V_{r-1}-\nu^W_0\neq 0$.
Put $n:=\max\{j|a^{(r-1)}_{-j}-b^{(0)}_{-j}\neq 0\}$.
Then we have
\begin{align*}
 dc=c(\nu^V_{r-1}-\nu^W_0)&=
 \left(\sum_{j=k}^{r^2m-1}c_jz^j\right)\sum_{j=-n}^{-1}(a^{(r-1)}_j-b^{(0)}_j)z^jdz \\
 &=c_k(a^{(r-1)}_{-n}-b^{(0)}_{-n})z^{k-n}dz+\sum_{j>k-n}\gamma_jz^j.
\end{align*}
Thus we have $k-1=k-n$ and $kc_k=c_k(a^{(r-1)}_{-n}-b^{(0)}_{-n})$.
So $n=1$ and $a^{(r-1)}_{-1}-b^{(0)}_{-1}=k\geq 1$, which contradicts the assumption
that $0\leq\mathrm{Re}(a^{(r-1)}_{-1})<1$, 
$0\leq \mathrm{Re}(b^{(0)}_{-1})<1$.
Hence we have $k\geq (r^2-1)m+1$.
Then $\im\psi\in z^{(r^2-1)m+1}(W/W_1)$.
So $\varphi$ induces a morphism
\[
 V_{r-1}\otimes\mathbf{C}[[z]]/(z^{(r^2-1)m})
 \longrightarrow
 W_1\otimes\mathbf{C}[[z]]/(z^{(r^2-1)m}),
\]
which also induces a morphism
\[
 \psi_1:V_{r-1}\otimes\mathbf{C}[[z]]/(z^{(r^2-1)m})
 \longrightarrow (W_1/W_2)\otimes\mathbf{C}[[z]]/(z^{(r^2-1)m}).
\]
We define $c^{(1)}\in\mathbf{C}[[z]]/(z^{(r^2-1)m})$
by $\psi_1(e^V_{r-1})=c^{(1)}e^W_1$.
If $\psi_1$ is isomorphic, then
\[
 (W_1/W_2)\otimes\mathbf{C}[[z]]/(z^{(r^2-1)m})
 \stackrel{\psi^{-1}_1}\longrightarrow
 V_{r-1}\otimes\mathbf{C}[[z]]/(z^{(r^2-1)m})
 \stackrel{\varphi}\longrightarrow
 W_1\otimes\mathbf{C}[[z]]/(z^{(r^2-1)m})
\]
gives a splitting of the exact sequence
\footnotesize
$$
 0\longrightarrow W_2\otimes\mathbf{C}[[z]]/(z^{(r^2-1)m})
 \longrightarrow W_1\otimes\mathbf{C}[[z]]/(z^{(r^2-1)m})
 \longrightarrow(W_1/W_2)\otimes\mathbf{C}[[z]]/(z^{(r^2-1)m})
 \longrightarrow 0.
$$\normalsize
So we have
\[
 W_1\otimes\mathbf{C}[[z]]/(z^{(r^2-1)m})
=\left((W_1/W_2)\otimes\mathbf{C}[[z]]/(z^{(r^2-1)m})\right)
\oplus \left(W_2\otimes\mathbf{C}[[z]]/(z^{(r^2-1)m})\right)
\]
and
$(\varphi\otimes\mathrm{id})
\left(V_{r-1}\otimes\mathbf{C}[[z]]/(z^{(r^2-1)m})\right)
=(W_1/W_2)\otimes\mathbf{C}[[z]]/(z^{(r^2-1)m})$.
Then we have $\nu^V_{r-1}=\nu^W_1$
and an isomorphism
\[
 (V/V_{r-1})\otimes\mathbf{C}[[z]]/(z^{(r-1)^2m})
 \stackrel{\sim}\longrightarrow
 \left(W\otimes\mathbf{C}[[z]]/(z^{(r-1)^2m})\right)\left/
 \left((W_1/W_2)\otimes\mathbf{C}[[z]]/(z^{(r-1)^2m})\right)\right..
\]
By induction hypothesis, there exists a permutation
$\sigma\in S_r$ such that
$\sigma(r-1)=1$ and $\nu^V_i=\nu^W_{\sigma(i)}$
for $i\neq r-1$.
If $\psi_1$ is not an isomorphism, then we can see by
a similar argument to the above that
$\varphi$ induces  a homomorphism
\[
 \psi_2:V_{r-1}\otimes\mathbf{C}[[z]]/(z^{(r^2-2)m})
 \longrightarrow
 (W_2/W_3)\otimes\mathbf{C}[[z]]/(z^{(r^2-2)m}).
\]
We repeat this argument and we finally obtain an isomorphism
\[
 \psi_j:V_{r-1}\otimes\mathbf{C}[[z]]/(z^{(r^2-j)m})
 \stackrel{\sim}\longrightarrow
 (W_j/W_{j+1})\otimes\mathbf{C}[[z]]/(z^{(r^2-j)m})
\]
for some $j$ with $0\leq j\leq r-1$.
So there is a slitting
\[
 (W_j/W_{j+1})\otimes\mathbf{C}[[z]]/(z^{((r^2-j)m})
 \stackrel{\psi_j^{-1}}\longrightarrow
 V_{r-1}\otimes\mathbf{C}[[z]]/(z^{(r^2-j)m})
 \stackrel{\varphi}\longrightarrow
 W_j\otimes\mathbf{C}[[z]]/(z^{(r^2-j)m})
\]
of the exact sequence
\footnotesize
$$
 0\longrightarrow W_{j+1}\otimes\mathbf{C}[[z]]/(z^{(r^2-j)m})
 \longrightarrow W_j\otimes\mathbf{C}[[z]]/(z^{(r^2-j)m})
 \longrightarrow (W_j/W_{j+1})\otimes\mathbf{C}[[z]]/(z^{((r^2-j)m})
 \longrightarrow 0.
$$\normalsize 
Therefore we have
$$
W_j\otimes\mathbf{C}[[z]]/(z^{(r^2-j)m})=
\left(W_j/W_{j+1}\otimes\mathbf{C}[[z]]/(z^{(r^2-j)m})\right)
\oplus \left(W_{j+1}\otimes\mathbf{C}[[z]]/(z^{(r^2-j)m})\right)
$$
and
$(\varphi\otimes\mathrm{id})(V_{r-1}\otimes\mathbf{C}[[z]]/(z^{(r^2-j)m}))
\subset (W_j/W_{j+1})\otimes\mathbf{C}[[z]]/(z^{(r^2-j)m})$.
Since $\varphi\otimes\mathrm{id}$ induces an isomorphism
$V_{r-1}\otimes\mathbf{C}[[z]]/(z^{(r^2-j)m})\stackrel{\sim}\longrightarrow
(W_j/W_{j+1})\otimes\mathbf{C}[[z]]/(z^{(r^2-j)m})$,
it also induces an isomorphism
\[
 (V/V_{r-1})\otimes\mathbf{C}[[z]]/(z^{(r-1)^2m})
 \stackrel{\sim}\longrightarrow
 \left(W\otimes\mathbf{C}[[z]]/(z^{(r-1)^2m})\right)\left/
 \left((W_j/W_{j+1})\otimes\mathbf{C}[[z]]/(z^{(r-1)^2m})\right)\right..
\]
So we have $\nu^V_{r-1}=\nu^W_j$ and by induction
hypothesis there exists a permutation $\sigma\in S_r$
such that $\sigma(r-1)=j$ and
$\nu^V_k=\nu^W_{\sigma(k)}$
for any $k\neq r-1$.
\end{proof}

\begin{Remark}\rm
Assume that $l=1$ and $0\leq\mathrm{Re}(\res(\nu_j))<1$ for any $j$
in Proposition \ref{prop:filtration}.
Then the eigenvalues $\nu_i^{V\otimes\mathbf{C}[z]/(z^{r^2m})}$
appeared in Proposition \ref{prop-deeper}
are nothing but the eigenvalues given in Hukuhara-Turrittin theorem
(Theorem \ref{thm-H-T}).
\end{Remark}

\section{Moduli space of unramified irregular singular parabolic connections}

Let $C$ be a smooth projective irreducible curve
over $\mathbf{C}$ of genus $g$ and
\[
 D=\sum_{i=1}^n m_i t_i \quad (\text{$m_i>0$, $t_i\neq t_j$ for $i\neq j$})
\]
be an effective divisor on $C$.
Take a generator $z_i$ of the maximal ideal of
${\mathcal O}_{C,t_i}$.
Let $E$ be a vector bundle of rank $r$ on $C$ and
$\nabla:E\rightarrow E\otimes\Omega^1_C(D)$
be a connection.
Take a positive integer $N_i$ with $N_i\geq m_i$ and put
$N_it_i:=\Spec\left({\mathcal O}_{C,t_i}/(z_i^{N_i})\right)$.
Then $\nabla$ induces a morphism
\[
 \nabla|_{N_it_i}:E\otimes{\mathcal O}_{C,t_i}/(z_i^{N_i})
 \longrightarrow E\otimes\Omega^1_C(D)
 \otimes{\mathcal O}_{C,t_i}/(z_i^{N_i}).
\]
Put
\begin{equation}\label{eq:exponent}
 N^{(n)}_r(d,D):=
 \left\{ \bnu=(\nu^{(i)}_j)_{1\leq i\leq n}^{0\leq j\leq r-1}
 \left|
 \begin{array}{l}
 \nu^{(i)}_j\in\sum_{k=-m_i}^{-1}\mathbf{C}z_i^kdz_i,
 \text{and} \\
 d+\sum_{1\leq i\leq n}\sum_{0\leq j\leq r-1}\res_{t_i}(\nu^{(i)}_j)=0
 \end{array}
 \right\}\right..
\end{equation}

\begin{Definition} \label{def-connection}\rm
 Take $\bnu \in N^{(n)}_r(d,D)$.
 We say $(E,\nabla,\{l^{(i)}_j\})$ an unramified irregular singular $\bnu$-parabolic connection
 of parabolic depth $(N_i)_{i=1}^n$ on $C$ if
 \begin{itemize}
 \item[(1)] $E$ is a rank $r$ vector bundle of degree $d$ on $C$,
 \item[(2)] $\nabla:E\rightarrow E\otimes\Omega_C(D)$ is a connection and
 \item[(3)] $E|_{N_it_i}=l^{(i)}_0\supset l^{(i)}_1\supset\cdots\supset
 l^{(i)}_{r-1}\supset l^{(i)}_r=0$ is a filtration by
 free ${\mathcal O}_{N_it_i}$-modules such that
 $l^{(i)}_j/l^{(i)}_{j+1}\cong{\mathcal O}_{N_it_i}$ for any $i,j$,
 $\nabla|_{N_it_i}(l^{(i)}_j)\subset l^{(i)}_j\otimes\Omega^1_C(D)$
 for any $i,j$ and for the induced morphism
 $\overline{\nabla}^{(i)}_j:l^{(i)}_j/l^{(i)}_{j+1}\rightarrow
 l^{(i)}_j/l^{(i)}_{j+1}\otimes\Omega^1_C(D)$,
 $\im(\overline{\nabla}^{(i)}_j-\nu^{(i)}_j\mathrm{id}_{l^{(i)}_j/l^{(i)}_{j+1}})$
 is contained in the image of
 $(l^{(i)}_j/l^{(i)}_{j+1})\otimes\Omega^1_C\rightarrow
 (l^{(i)}_j/l^{(i)}_{j+1})\otimes\Omega^1_C(D)$.
 \end{itemize}
\end{Definition}

We fix a sequence of rational numbers
$\balpha=(\alpha^{(i)}_j)^{1\leq i\leq n}_{1\leq j\leq r}$ such that
$0<\alpha^{(i)}_1<\alpha^{(i)}_2<\cdots<\alpha^{(i)}_r<1$
for any $i$ and $\alpha^{(i)}_j\neq \alpha^{(i')}_{j'}$ for
$(i,j)\neq (i',j')$.

\begin{Definition}\label{def-stability}\rm
 A $\bnu$-parabolic connection $(E,\nabla,\{l^{(i)}_j\})$ is said to be 
 $\balpha$-stable (resp.\ $\balpha$-semistable) if
 for any subbundle $0\neq F\subsetneq E$ with
 $\nabla(F)\subset F\otimes\Omega^1_C(D)$, the inequality
 \begin{gather*}
 \genfrac{}{}{}{}{\deg F+\sum_{i=1}^n\sum_{j=1}^r
 \alpha^{(i)}_j\length\left(\left(F|_{N_it_i}\cap l^{(i)}_{j-1}\right)/
 \left(F|_{N_it_i}\cap l^{(i)}_j\right)\right)}{\rank F} \\
 \genfrac{}{}{0pt}{}{<}{(\text{resp.\ $\leq$})}
 \genfrac{}{}{}{}{\deg E+\sum_{i=1}^n\sum_{j=1}^r
 \alpha^{(i)}_j\length(l^{(i)}_{j-1}/l^{(i)}_j)}
 {\rank E}
 \end{gather*}
 holds.
\end{Definition}

\begin{Remark}\rm
O.~Biquard and P.~Boalch consider in [\cite{BB}, section 8]
a stability condition for a meromorphic connection with the assumption
that the restriction of the connection to each singular point
is equivalent to diagonal one.
For a parabolic weight $\balpha=(\alpha^{(i)}_j)$ with
$0<\alpha^{(i)}_j<1/N_i$, the $\balpha$-stability in our definition
for a parabolic connection $(E,\nabla,\{l^{(i)}_j\})$ is equivalent to
the $(\alpha^{(i)}_jN_i)$-stability in \cite{BB}
for $(E,\nabla)$ under the Main assumption in \cite{BB}.
\end{Remark}

\begin{Remark}\rm
Take a parabolic connection $(E,\nabla,\{l^{(i)}_j\})$
with parabolic depth $(m_i)$.
Fix $l^{(i')}_{j'}$ and put
$E':=\ker(E\rightarrow E|_{m_{i'}t_{i'}}/l^{(i')}_{j'})$.
Then $\nabla$ induces a connection
$\nabla':E'\rightarrow E'\otimes\Omega^1_C(D)$.
We define a parabolic structure $\{(l')^{(i)}_j\}$ on $E'$ by
$(l')^{(i)}_j:=l^{(i)}_j$ for $i\neq i'$,
$(l')^{(i')}_j:=\ker(E'|_{m_{i'}t_{i'}}\rightarrow E|_{m_{i'}t_{i'}}/l^{(i')}_{j+j'})$
for $0\leq j\leq r-j'$ and
$(l')^{(i')}_j:=\mathrm{im}(l^{(i')}_{j-r+j'}\otimes{\mathcal O}_C(-m_{i'}t_{i'})
\hookrightarrow E|_{m_{i'}t_{i'}}\otimes{\mathcal O}_C(-m_{i'}t_{i'})
\rightarrow E'|_{m_{i'}t_{i'}})$
for $r-j'\leq j\leq r$.
Then we obtain a new parabolic connection $(E',\nabla',\{(l')^{(i)}_j\})$.
We call this the elementary transform of $(E,\nabla,\{l^{(i)}_j\})$
along $l^{(i')}_{j'}$.
We put $(\alpha')^{(i)}_j:=\alpha^{(i)}_j$ for $i\neq i'$,
$(\alpha')^{(i')}_j:=\alpha^{(i')}_{j+j'}$ for $1\leq j\leq r-j'$ and
$(\alpha')^{(i')}_j:=\alpha^{(i')}_{j-r+j'}+1$ for $r-j'+1\leq j\leq r$.
Then $(E,\nabla,\{l^{(i)}_j\})$ is $\balpha$-stable if and only if
$(E',\nabla',\{(l')^{(i)}_j\})$ satisfies the following stability condition:
for any subbundle $F'\subset E'$ with
$\nabla'(F')\subset F'\otimes\Omega^1_C(D)$,
\begin{gather*}
 \genfrac{}{}{}{}{\deg F'+\sum_{i=1}^n\sum_{j=1}^r
 (\alpha')^{(i)}_j\length((F'|_{m_it_i}\cap(l')^{(i)}_{j-1})/(F'|_{m_it_i}\cap(l')^{(i)}_j))}{\rank F'} \\
 < \genfrac{}{}{}{}{\deg E'+\sum_{i=1}^n\sum_{j=1}^r(\alpha')^{(i)}_j\length((l')^{(i)}_{j-1}/(l')^{(i)}_j)}
 {\rank E'}
\end{gather*}
holds.
So we can consider a stability of a parabolic connection with respect to a
weight $\balpha=(\alpha^{(i)}_j)$ without the condition
$0<\alpha^{(i)}_j<1$.
\end{Remark}


Let $S$ be an algebraic scheme over $\mathbf{C}$ and
${\mathcal C}$ be a projective flat scheme over $S$, such that
each geometric fiber ${\mathcal C}_s$ of ${\mathcal C}$ over $S$
is a smooth irreducible curve of genus $g$.
Let $\tilde{t}_1,\ldots,\tilde{t}_n\subset {\mathcal C}$ be closed
subschemes such that the composite
$\tilde{t}_i\hookrightarrow{\mathcal C}\rightarrow S$
is an isomorphism for any $i$ and that
$\tilde{t}_i\cap\tilde{t}_j=\emptyset$ for any $i\neq j$.
We put $D:=\sum_{i=1}^n m_i\tilde{t}_i$.
Then $D$ is an effective Cartier divisor on ${\mathcal C}$
flat over $S$.
Let  $\cN^{(n)}_r(d,D)$ be the scheme over $S$
such that for any $T\rightarrow S$,
\begin{equation}\label{eq:scheme-exp}
 \cN^{(n)}_r(d,D)(T)=
 \left\{ \bnu=(\nu^{(i)}_j)\left|
 \begin{array}{l}
 \nu^{(i)}_j\in
 H^0\left(T,\Omega^1_{\mathcal C}(m_i\tilde{t}_i)_T/(\Omega^1_{\mathcal C})_T\right) \\
 d+\sum_{i,j}\res_{\tilde{t_i}}(\nu^{(i)}_j)=0
 \end{array}
 \right\}\right.
\end{equation}

\begin{Theorem}\label{thm-moduli-exists}
 There exists a relative coarse moduli scheme 
 $M^{\balpha}_{D/{\mathcal C}/S}(r,d,(N_i))\stackrel{\pi}\longrightarrow
 \cN^{(n)}_r(d,D)$
 of $\balpha$-stable unramified irregular singular $\bnu$-parabolic connections
 ($\bnu$ moves around in $\cN^{(n)}_r(d,D)$)
 on ${\mathcal C}$ over $S$
 of parabolic depth $(N_i)_{i=1}^n$.
 Moreover $M^{\balpha}_{D/{\mathcal C}/S}(r,d,(N_i))$
 is quasi-projective over $\cN^{(n)}_r(d,D)$.
\end{Theorem}

\begin{proof}
Fix a weight $\balpha$ which determines the stability of
irregular singular parabolic connections.
We take positive integers $\beta_1,\beta_2,\gamma$
and rational numbers
$0<\tilde{\alpha}^{(i)}_1<\tilde{\alpha}^{(i)}_2<\cdots<\tilde{\alpha}^{(i)}_r<1$
satisfying $(\beta_1+\beta_2)\alpha^{(i)}_j=\beta_1\tilde{\alpha}^{(i)}_j$
for any $i,j$.
We assume that $\gamma\gg 0$.
We can take an increasing sequence
$0<\alpha'_1<\alpha'_2<\cdots<\alpha'_{nr}<1$
such that
$\{\alpha'_p|p=1,\ldots,nr\}=
\left\{\tilde{\alpha}^{(i)}_j| 1\leq i\leq n, 1\leq j\leq r\right\}$.

Take any member
$(E,\nabla,\{l^{(i)}_j\})\in{\mathcal M}^{\balpha}_{D/\cC/S}(r,d,(N_i))(T)$,
where ${\mathcal M}^{\balpha}_{D/\cC/S}(r,d,(N_i))$
is the moduli functor of $\balpha$-stable unramified irregular singular parabolic connections
of parabolic depth $(N_i)$.
We define subsheaves $F_p(E)\subset E$ inductively as follows:
First we put $F_1(E):=E$.
Inductively we define
$F_{p+1}(E):=\ker\left(F_p(E)\rightarrow (E|_{N_i(\tilde{t}_i)_T})/l^{(i)}_j\right)$,
where $(i,j)$ is determined by $\alpha'_p=\alpha^{(i)}_j$.
We also put $d_p:=\length((E/F_{p+1}(E))\otimes k(x))$
for $p=1,\ldots,rn$ and $x\in T$.
Then
$(E,\nabla,\{l^{(i)}_j\})\mapsto
(E,E,\mathrm{id}_E,\nabla,F_*(E))$
determines a morphism
\[
 \iota:{\mathcal M}^{\balpha}_{D/\cC/S}(r,d,(N_i))\longrightarrow
 \overline{{\mathcal M}^{D',\balpha',\bbeta,\gamma}
 _{\cC\times_S\cN^{(n)}_r(d,D)/\cN^{(n)}_r(d,D)}}(r,d,\{d_i\}_{1\leq i\leq rn}),
\]
where
$\overline{{\mathcal M}^{D',\balpha',\bbeta,\gamma}
_{\cC\times_S\cN^{(n)}_r(d,D)/\cN^{(n)}_r(d,D)}}(r,d,\{d_i\}_{1\leq i\leq rn})$
is the moduli functor of $(\balpha',\bbeta,\gamma)$-stable
parabolic $\Lambda^1_{D'}$-triples whose coarse moduli scheme
$\overline{M^{D',\balpha',\bbeta,\gamma}
_{\cC\times_S\cN^{(n)}_r(d,D)/\cN^{(n)}_r(d,D)}}(r,d,\{d_i\}_{1\leq i\leq rn})$
exists by
[\cite{IIS-1}, Theorem 5.1].
Here we put $D':=\sum_{i=1}^nN_i\tilde{t}_i$.
We can check that $\iota$ is representable by an immersion.
So we can prove in the same way as [\cite{IIS-1}, Theorem 2.1]
that a certain locally closed subscheme
$M^{\balpha}_{D/\cC/S}(r,d,(N_i))$ of
$\overline{M^{D',\balpha',\bbeta,\gamma}
_{\cC\times_S\cN^{(n)}_r(d,D)/\cN^{(n)}_r(d,D)}}(r,d,\{d_i\}_{1\leq i\leq rn})$
is just the coarse moduli scheme of
${\mathcal M}^{\balpha}_{D/\cC/S}(r,d,(N_i))$.
By construction, we can see that
$M^{\balpha}_{D/\cC/S}(r,d,(N_i))$
represents the \'etale sheafification of
${\mathcal M}^{\balpha}_{D/\cC/S}(r,d,(N_i))$.
\end{proof}

There is also a coarse moduli scheme
$\tilde{M}^{\balpha}_{D/{\mathcal C}/S}(r,d,(N_i))$ of
$\bnu$-parabolic connections
$(E,\nabla,\{l^{(i)}_j\})$ of parabolic depth $(N_i)$ such that
$(E,\nabla,\{l^{(i)}_j\otimes\mathbf{C}[z_i]/(z_i^{m_i})\})$
is $\balpha$-stable.
Indeed we can construct
$\tilde{M}^{\balpha}_{D/{\mathcal C}/S}(r,d,(N_i))$
as a quasi-projective scheme over
$M^{\balpha}_{D/{\mathcal C}/S}(r,d,(m_i))$.

\begin{Theorem}\label{smoothness-thm}
 $M^{\balpha}_{D/{\mathcal C}/S}(r,d,(m_i))$ is smooth over
 $\cN^{(n)}_r(d,D)$ and
 \[
  \dim\left( M^{\balpha}_{D/{\mathcal C}/S}(r,d,(m_i))_{\bnu}\right)
  =2r^2(g-1)+\sum_{i=1}^nm_ir(r-1) +2
 \]
 for any $\bnu\in\cN^{(n)}_r(d,D)$
 if $M^{\balpha}_{D/\cC/S}(r,d,(m_i))_{\bnu}$ is not empty.
\end{Theorem}

We will prove Theorem \ref{smoothness-thm}
in several steps.

We can canonically define a morphism
\[
 \det:M^{\balpha}_{D/{\mathcal C}/S}(r,d,(m_i))
 \longrightarrow
 M_{D/{\mathcal C}/S}(1,d,(m_i))\times_{\cN^{(n)}_1(d,D)}\cN^{(n)}_r(d,D)
\]
by
\[
 \det(E,\nabla,\{l^{(i)}_j\}):=(\det(E),\det(\nabla),\pi(E,\nabla,\{l^{(i)}_j\})).
\]
Here 
$M_{D/\cC/S}(1,d,(m_i))$ is the moduli space of
pairs $(L,\nabla^L)$ of a line bundle $L$ on
$\cC_s$ and
a connection
$\nabla^L:L\rightarrow L\otimes\Omega^1_{\cC_s}(D_s)$.
Note that we put
\[
 \det(\nabla):=(\nabla\wedge\mathrm{id}\wedge\cdots\wedge\mathrm{id})
 +(\mathrm{id}\wedge\nabla\wedge\cdots\wedge\mathrm{id})
 +\cdots+(\mathrm{id}\wedge\cdots\wedge\mathrm{id}\wedge\nabla)
\]
and the morphism $\Tr:\cN^{(n)}_r(d,D)\rightarrow\cN^{(n)}_1(d,D)$
is given by $\Tr((\nu^{(i)}_j))=\left(\sum_{j=0}^{r-1}\nu^{(i)}_j\right)_{i=1}^n$.

\begin{Proposition}\label{prop-smoothness-det}
 The morphism
 \[
  \det:M^{\balpha}_{D/{\mathcal C}/S}(r,d,(m_i))
 \longrightarrow
 M_{D/{\mathcal C}/S}(1,d,(m_i))\times_{\cN^{(n)}_1(d,D)}\cN^{(n)}_r(d,D)
 \]
 defined above is smooth.
\end{Proposition}

\begin{proof}
We can see by an easy argument that
it is sufficient to show that
the morphism of moduli functors
\[
 \det: {\mathcal M}^{\balpha}_{D/{\mathcal C}/S}(r,d,(m_i))
 \longrightarrow
 M_{D/{\mathcal C}/S}(1,d,(m_i))\times_{\cN^{(n)}_1(d,D)}
 \cN^{(n)}_r(d,D)
\]
is formally smooth.
Let $A$ be an artinian local ring with maximal ideal $m$
and residue field $k=A/m$.
Take an ideal $I$ of $A$ such that $mI=0$.
Let 
\[
\begin{CD}
 \Spec(A/I) @>f>> {\mathcal M}^{\balpha}_{D/{\mathcal C}/S}(r,d,(m_i)) \\
 @VVV @VV\det V \\
 \Spec(A) @>g>> 
 M_{D/{\mathcal C}/S}(1,d,(m_i))\times_{\cN^{(n)}_1(d,D)}\cN^{(n)}_r(d,D)
\end{CD}
\]
be a commutative diagram.
$g$ corresponds to a line bundle $L$ on ${\mathcal C}_A$ with a connection
$\nabla^L:L\rightarrow L\otimes\Omega^1_{{\mathcal C}_A/A}(D_A)$
and $\bnu=(\nu^{(i)}_j)\in\cN^{(n)}_r(d,D)(A)$
such that
$\nabla^L|_{m_i(\tilde{t}_i)_A}(a)=\left(\sum_{j=0}^{r-1}\nu^{(i)}_j\right)a$
for any $a\in L|_{m_i(\tilde{t}_i)_A}$ and $i=1,\ldots,n$.
$f$ corresponds to an element
$(E,\nabla,\{l^{(i)}_j\})\in
{\mathcal M}^{\balpha}_{D/{\mathcal C}/S}(r,d,(m_i))(A/I)$.
Put
$(\overline{E},\overline{\nabla},\{\overline{l}^{(i)}_j\}):=
(E,\nabla,\{l^{(i)}_j\})\otimes A/m$.
We set
\begin{align*}
 &{\mathcal F}_0^0:=\left\{a \in{\mathcal End}(\overline{E})\left|
 \text{$\Tr(a)=0$ and
 $a|_{m_i(\tilde{t}_i)_k}(\overline{l}^{(i)}_j)\subset\overline{l}^{(i)}_j$
 for any $i,j$}
 \right\}\right. \\
 &{\mathcal F}_0^1:=\left\{ b\in{\mathcal End}(\overline{E})
 \otimes\Omega^1_{{\mathcal C}/S}(D)
 \left| \text{$\Tr(b)=0$ and
 $b|_{m_i(\tilde{t}_i)_k}(\overline{l}^{(i)}_j)\subset
 \overline{l}^{(i)}_{j+1}\otimes\Omega^1_{\cC/S}(D)$ for any $i,j$}
 \right\}\right.  \\
 &\nabla_{{\mathcal F}_0^{\bullet}}:{\mathcal F}_0^0\ni a\mapsto
 \overline{\nabla}a-a\overline{\nabla}\in{\mathcal F}_0^1.
\end{align*}
Let $\cC_A=\bigcup_{\alpha}U_{\alpha}$ be an affine open
covering such that
$E|_{U_{\alpha}\otimes A/I}\cong
{\mathcal O}_{U_{\alpha}\otimes A/I}^{\oplus r}$,
$\sharp\{(\tilde{t}_i)_A|(\tilde{t}_i)_A\in U_{\alpha}\}\leq 1$ for any $\alpha$
and $\sharp\{\alpha|(\tilde{t}_i)_A\in U_{\alpha}\}=1$
for any $(\tilde{t}_i)_A$.
Take a free ${\mathcal O}_{U_{\alpha}}$-module
$E_{\alpha}$ with isomorphisms
$\varphi_{\alpha}:\det(E_{\alpha})\stackrel{\sim}\rightarrow
L|_{U_{\alpha}}$ and
$\phi_{\alpha}:E_{\alpha}\otimes A/I\stackrel{\sim}\rightarrow
E|_{U_{\alpha}\otimes A/I}$
such that
\[
 \varphi_{\alpha}\otimes A/I=\det(\phi_{\alpha}):
 \det(E_{\alpha})\otimes A/I\stackrel{\sim}\longrightarrow
 \det(E)|_{U_{\alpha}\otimes A/I}
 =(L\otimes A/I)|_{U_{\alpha}\otimes A/I}.
\]
If $(\tilde{t}_i)_A\in U_{\alpha}$,
we may assume that parabolic structure
$\{l^{(i)}_j\}$ is given by
\[
 l^{(i)}_{r-j}=\langle e_1|_{m_i(\tilde{t}_i)_{A/I}},
 \ldots,e_j|_{m_i(\tilde{t}_i)_{A/I}} \rangle,
\]
where $e_1,\ldots,e_r$ is the standard basis of $E_{\alpha}$.
We define a parabolic structure $\{(l_{\alpha})^{(i)}_j\}$
on $E_{\alpha}$ by
\[
 (l_{\alpha})^{(i)}_{r-j}:=\langle e_1|_{m_i(\tilde{t}_i)_A},
 \ldots,e_j|_{m_i(\tilde{t}_i)_A} \rangle.
\]
The connection
$\phi_{\alpha}^{-1}\circ(\nabla|_{U_{\alpha}})\circ\phi_{\alpha}:
E_{\alpha}\otimes A/I\rightarrow
E_{\alpha}\otimes\Omega^1_{\cC/S}(D)\otimes A/I$
is given by a connection matrix
$B_{\alpha}\in
H^0((E_{\alpha})^{\vee}\otimes E_{\alpha}
\otimes\Omega^1_{\cC/S}(D)\otimes A/I)$.
Then we have
\[
 B_{\alpha}|_{(\tilde{t}_i)_{A/I}}=
 \begin{pmatrix}
  \nu^{(i)}_{r-1}\otimes A/I & * & \cdots & * \\
  0 & \nu^{(i)}_{r-2}\otimes A/I & \cdots & * \\
  \vdots & \vdots & \ddots & \vdots \\
  0 & 0 & \cdots & \nu^{(i)}_0\otimes A/I
 \end{pmatrix}.
\]
We can take a lift
$\tilde{B}_{\alpha}\in
H^0(E_{\alpha}^{\vee}\otimes E_{\alpha}\otimes\Omega^1_{\cC/S}(D))$
of $B_{\alpha}$ such that
\[
 \tilde{B}_{\alpha}|_{(\tilde{t}_i)_A}=
 \begin{pmatrix}
  \nu^{(i)}_{r-1} & * & \cdots & * \\
  0 & \nu^{(i)}_{r-2} & \cdots & * \\
  \vdots & \vdots & \ddots & \vdots \\
  0 & 0 & \cdots & \nu^{(i)}_0
 \end{pmatrix}.
\]
and that
$\Tr(\tilde{B}_{\alpha})(e_1\wedge\cdots\wedge e_r)
=(\varphi_{\alpha}\otimes\mathrm{id})^{-1}
(\nabla_L|_{U_{\alpha}}(\varphi_{\alpha}(e_1\wedge\cdots\wedge e_r)))$.
Consider the connection
$\nabla_{\alpha}:E_{\alpha}\rightarrow
E_{\alpha}\otimes\Omega^1_{\cC/S}(D)$
defined by
\[
 \nabla_{\alpha}
 \begin{pmatrix}
  f_1 \\
  \vdots \\
  f_r
 \end{pmatrix}
 =
 \begin{pmatrix}
  df_1 \\
  \vdots \\
  df_r
 \end{pmatrix}
 +
 \tilde{B}_{\alpha}
 \begin{pmatrix}
  f_1 \\
  \vdots \\
  f_r
 \end{pmatrix}.
\]
Then we obtain a local parabolic connection
$(E_{\alpha},\nabla_{\alpha},\{(l_{\alpha})^{(i)}_j\})$.
If $(\tilde{t}_i)_A\notin U_{\alpha}$ for any $i$,
we can easily define a local parabolic connection
$(E_{\alpha},\nabla_{\alpha},\{(l_{\alpha})^{(i)}_j\})$
(in this case the parabolic structure $\{(l_{\alpha})^{(i)}_j\}$
is nothing).

We put $U_{\alpha\beta}:=U_{\alpha}\cap U_{\beta}$ and
$U_{\alpha\beta\gamma}:=U_{\alpha}\cap U_{\beta}\cap U_{\gamma}$.
Take an isomorphism
\[
 \theta_{\beta\alpha}:E_{\alpha}|_{U_{\alpha\beta}}
 \stackrel{\sim}\longrightarrow E_{\beta}|_{U_{\alpha\beta}}
\]
such that
$\theta_{\beta\alpha}\otimes A/I=\phi_{\beta}^{-1}\circ\phi_{\alpha}$
and that
$\varphi_{\beta}\circ\det(\theta_{\beta\alpha})=\varphi_{\alpha}$.
We put
\[
 u_{\alpha\beta\gamma}:=\phi_{\alpha}\circ
 \left(\theta_{\gamma\alpha}^{-1}|_{U_{\alpha\beta\gamma}}
 \circ\theta_{\gamma\beta}|_{U_{\alpha\beta\gamma}}\circ
 \theta_{\beta\alpha}|_{U_{\alpha\beta\gamma}}
 -\mathrm{id}_{E_{\alpha}|_{U_{\alpha\beta\gamma}}}
 \right)\circ\phi_{\alpha}^{-1}
\]
and
\[
 v_{\alpha\beta}:=\phi_{\alpha}\circ\left(
 \nabla_{\alpha}|_{U_{\alpha\beta}}
 -\theta_{\beta\alpha}^{-1}\circ\nabla_{\beta}|_{U_{\alpha\beta}}
 \circ\theta_{\beta\alpha}\right)\circ\phi_{\alpha}^{-1}.
\]
Then we have
$\{u_{\alpha\beta\gamma}\}\in C^2(\{U_{\alpha}\},{\mathcal F}_0^0\otimes I)$
and
$\{v_{\alpha\beta}\}\in C^1(\{U_{\alpha}\},{\mathcal F}_0^1\otimes I)$.
We can easily see that
\[
 d\{u_{\alpha\beta\gamma}\}=0 \quad\text{and}\quad
 \nabla_{{\mathcal F}_0^{\bullet}}\{u_{\alpha\beta\gamma}\}=
 -d\{v_{\alpha\beta}\}.
\]
So we can define an element
\[
 \omega(E,\nabla,\{l^{(i)}_j\}):=
 [(\{u_{\alpha\beta\gamma}\},\{v_{\alpha\beta}\})]
 \in\mathbf{H}^2({\mathcal F}_0^{\bullet})\otimes_kI.
\]
Then we can check that $\omega(E,\nabla,\{l^{(i)}_j\})=0$
if and only if
$(E,\nabla,\{l^{(i)}_j\})$ can be lifted to an element
$(\tilde{E},\tilde{\nabla},\{\tilde{l}^{(i)}_j\})$ of
${\mathcal M}^{\balpha}_{D/{\mathcal C}/S}(r,d,(m_i))(A)$
such that
\[
 \det(\tilde{E},\tilde{\nabla},\{\tilde{l}^{(i)}_j\})=g.
\]
{F}rom the spectral sequence
$H^q({\mathcal F}_0^p)\Rightarrow
\mathbf{H}^{p+q}({\mathcal F}_0^{\bullet})$,
there is an isomorphism
\[
 \mathbf{H}^2({\mathcal F}_0^{\bullet})\cong
 \coker\left(H^1({\mathcal F}_0^0)
 \xrightarrow{H^1(\nabla_{{\mathcal F}_0^{\bullet}})}
 H^1({\mathcal F}_0^1)\right).
\]
Since
$({\mathcal F}_0^0)^{\vee}\otimes\Omega^1_{{\mathcal C}_k/k}
\cong{\mathcal F}_0^1$ and
$({\mathcal F}_0^1)^{\vee}\otimes\Omega^1_{{\mathcal C}_k/k}
\cong{\mathcal F}_0^0$,
we have
\begin{align*}
 \mathbf{H}^2({\mathcal F}_0^{\bullet})
 &\cong
 \coker\left(H^1({\mathcal F}_0^0)
 \xrightarrow{H^1(\nabla_{{\mathcal F}_0^{\bullet}})}
 H^1({\mathcal F}_0^1)\right) \\
 &\cong \ker\left(H^1({\mathcal F}_0^1)^{\vee}
 \xrightarrow{H^1(\nabla_{{\mathcal F}_0^{\bullet}})^{\vee}}
 H^1({\mathcal F}_0^0)^{\vee}\right)^{\vee} \\
 &\cong \ker\left(
 H^0(({\mathcal F}_0^1)^{\vee}\otimes\Omega^1_{{\mathcal C}_k/k})
 \xrightarrow{-H^0(\nabla_{({\mathcal F}_0^{\bullet})^{\vee}})}
 H^0(({\mathcal F}_0^0)^{\vee}\otimes\Omega^1_{{\mathcal C}_k/k})
 \right)^{\vee} \\
 &\cong\ker\left(
 H^0({\mathcal F}_0^0)\xrightarrow{-H^0(\nabla_{{\mathcal F}_0^{\bullet}})}
 H^0({\mathcal F}_0^1)
 \right)^{\vee}.
\end{align*}
Take any element
$a\in \ker\left(H^0({\mathcal F}_0^0)
\xrightarrow{-H^0(\nabla_{{\mathcal F}_0^{\bullet}})}
H^0({\mathcal F}_0^1)
\right)$.
Then we have
$a\in\End(\overline{E},\overline{\nabla},\{\overline{l}^{(i)}_j\})$.
Since $(\overline{E},\overline{\nabla},\{\overline{l}^{(i)}_j\})$ is
$\balpha$-stable, we have $a=c\cdot\mathrm{id}_{\overline{E}}$
for some $c\in\mathbf{C}$.
So we have $a=0$, because $\Tr(a)=0$.
Thus  we have
$\ker\left(H^0({\mathcal F}_0^0)
\xrightarrow{-H^0(\nabla_{{\mathcal F}_0^{\bullet}})}
H^0({\mathcal F}_0^1)\right)=0$ and so we have
$\mathbf{H}^2({\mathcal F}_0^{\bullet})=0$.
In particular, we have $\omega(E,\nabla,\{l^{(i)}_j\})=0$.
Thus $(E,\nabla,\{l^{(i)}_j\})$ can be lifted to a member
$(\tilde{E},\tilde{\nabla},\{\tilde{l}^{(i)}_j\})\in
{\mathcal M}^{\balpha}_{D/\cC/S}(r,d,(m_i))(A)$
such that
$(\tilde{E},\tilde{\nabla},\{\tilde{l}^{(i)}_j\})\otimes A/I
\cong(E,\nabla,\{l^{(i)}_j\})$
and
$\det(\tilde{E},\tilde{\nabla},\{\tilde{l}^{(i)}_j\})=g$.
Hence $\det$ is a smooth morphism.
\end{proof}

We can see that the moduli space
$M_{D/\cC/S}(1,d,(m_i))$
is an affine space bundle
over 
\par \noindent
$\Pic^d_{\cC/S}\times\cN^{(n)}_1(d,D)$
with fibers $H^0(\Omega^1_{\cC_s}) (s\in S)$.
So $M_{D/\cC/S}(1,d,(m_i))$ is smooth over
$\cN^{(n)}_1(d,D)$.
Combined with Proposition \ref{prop-smoothness-det},
we can see that
$M^{\balpha}_{D/\cC/S}(r,d,(m_i))$ is smooth over
$\cN^{(n)}_r(d,D)$.

\begin{Proposition}\label{prop-moduli-dimension}
 For any $\bnu\in\cN^{(n)}_r(d,D)$,
 the fiber
 $\pi^{-1}(\bnu)=M^{\balpha}_{D/\cC/S}(r,d,(m_i))_{\bnu}$
 is equidimensional of dimension
 $2r^2(g-1)+2+r(r-1)\sum_{i=1}^n m_i$
if it is not empty.
\end{Proposition}

\begin{proof}
Since $M^{\balpha}_{D/\cC/S}(r,d,(m_i))_{\bnu}$
is smooth over $\mathbf{C}$ for any
$\bnu\in\cN^{(n)}_r(d,D)(\mathbf{C})$,
it is sufficient to show that
the dimension of the tangent space
$\Theta_{M^{\balpha}_{D/\cC/S}(r,d,(m_i))_{\bnu}}(x)$
of $M^{\balpha}_{D/\cC/S}(r,d,(m_i))_{\bnu}$ at any point
$x=(E,\nabla,\{l^{(i)}_j\})\in
M^{\balpha}_{D/\cC/S}(r,d,(m_i))_{\bnu}$
is equal to
\[
 2r^2(g-1)+2+r(r-1)\sum_{i=1}^n m_i.
\]

We define a complex ${\mathcal F}^{\bullet}$ on $\cC_x$ by
\begin{align*}
 &{\mathcal F}^0:=\left\{ a\in{\mathcal End}(E)\left|
 \text{$a|_{m_i(\tilde{t}_i)_x}(l^{(i)}_j)\subset l^{(i)}_j$ for any $i,j$}
 \right\}\right., \\
 &{\mathcal F}^1:=\left\{
 b\in{\mathcal End}(E)\otimes\Omega^1_{\cC/S}(D)\left|
 \text{$b|_{m_i(\tilde{t}_i)_x}(l^{(i)}_j)\subset
 l^{(i)}_{j+1}\otimes\Omega^1_{\cC/S}(D)$ for any $i,j$}
 \right\}\right., \\
 &\nabla_{{\mathcal F}^{\bullet}}:
 {\mathcal F}^0\ni a\mapsto \nabla\circ a-a\circ\nabla
 \in{\mathcal F}^1
\end{align*}
Take a tangent vector
$v\in\Theta_{M^{\balpha}_{D/\cC/S}(r,d,(m_i))_{\bnu}}(x)$.
Then $v$ corresponds to a member 
$$(E^v,\nabla^v,\{(l^v)^{(i)}_j\})\in
M^{\balpha}_{D/\cC/S}(r,d,(m_i))_{\bnu}(\mathbf{C}[\epsilon])
$$
such that
$(E^v,\nabla^v,\{(l^v)^{(i)}_j\})\otimes\mathbf{C}[\epsilon]/(\epsilon)
\cong(E,\nabla,\{l^{(i)}_j\})$,
where $\epsilon^2=0$.
Take an affine open covering
$\cC_x=\bigcup_{\alpha}U_{\alpha}$
such that
$E|_{U_{\alpha}}\cong{\mathcal O}_{U_{\alpha}}^{\oplus r}$,
$\sharp\{i|(\tilde{t}_i)_x\in U_{\alpha}\}\leq 1$ for any $\alpha$
and $\sharp\{\alpha|(\tilde{t}_i)_x\in U_{\alpha}\}=1$
for any $i$.
We can take an isomorphism
\[
 \varphi_{\alpha}:E^v|_{U_{\alpha}\times\Spec\mathbf{C}[\epsilon]}
 \stackrel{\sim}\longrightarrow
 (E\otimes_{\mathbf{C}}\mathbf{C}[\epsilon])
 |_{U_{\alpha}\times\Spec\mathbf{C}[\epsilon]}
\]
such that
$\varphi_{\alpha}\otimes\mathbf{C}[\epsilon]/(\epsilon):
E^v\otimes\mathbf{C}[\epsilon]/(\epsilon)|_{U_{\alpha}}
\stackrel{\sim}\rightarrow
(E\otimes\mathbf{C}[\epsilon]/(\epsilon))|_{U_{\alpha}}
=E|_{U_{\alpha}}$
is the given isomorphism. 
We put 
\begin{eqnarray*}
u_{\alpha\beta}& := & \varphi_{\alpha}\circ\varphi_{\beta}^{-1}
 -\mathrm{id}_{(E\otimes\mathbf{C}[\epsilon])
 |_{U_{\alpha\beta}\times\Spec\mathbf{C}[\epsilon]}}, \\   
 v_{\alpha} & := & (\varphi_{\alpha}\otimes\mathrm{id})\circ
 \nabla^v|_{U_{\alpha}\times\Spec\mathbf{C}[\epsilon]}
 \circ\varphi_{\alpha}^{-1}
 -\nabla\otimes\mathbf{C}[\epsilon]
 |_{U_{\alpha}\times\Spec\mathbf{C}[\epsilon]}.
 \end{eqnarray*}
Then we have
$\{u_{\alpha\beta}\}\in C^1(\{U_{\alpha}\},(\epsilon)\otimes{\mathcal F}^0)$,
$\{v_{\alpha}\}\in C^0(\{U_{\alpha}\},(\epsilon)\otimes{\mathcal F}^1)$
and
\[
 d\{u_{\alpha\beta}\}=
 \{u_{\beta\gamma}-u_{\alpha\gamma}+u_{\alpha\beta}\}=0,\;
 \nabla_{{\mathcal F}^{\bullet}}\{u_{\alpha\beta}\}=
 \{v_{\beta}-v_{\alpha}\}=d\{v_{\alpha}\}.
\]
So $[(\{u_{\alpha\beta}\},\{v_{\alpha}\}]$ determines
an element $\sigma_x(v)\in \mathbf{H}^1({\mathcal F}^{\bullet})$.
We can easily check that
the correspondence
$v\mapsto \sigma_x(v)$ gives an isomorphism
\[
 \Theta_{M^{\balpha}_{D/\cC/S}(r,d,(m_i))_{\bnu}}(x)
 \stackrel{\sim}\longrightarrow
 \mathbf{H}^1({\mathcal F}^{\bullet}).
\]
{F}rom the spectral sequence
$H^q({\mathcal F}^p)\Rightarrow \mathbf{H}^{p+q}({\mathcal F}^{\bullet})$,
we obtain an exact sequence
\[
 0 \longrightarrow \mathbf{C} \longrightarrow H^0({\mathcal F}^0)
 \longrightarrow H^0({\mathcal F}^1) \longrightarrow
 \mathbf{H}^1({\mathcal F}^{\bullet}) \longrightarrow
 H^1({\mathcal F}^0) \longrightarrow H^1({\mathcal F}^1)
 \longrightarrow \mathbf{C} \longrightarrow 0.
\]
So we have
\begin{align*}
 \dim\mathbf{H}^1({\mathcal F}^{\bullet})&=
 \dim H^0({\mathcal F}^1)
 +\dim H^1({\mathcal F}^0)-\dim H^0({\mathcal F}^0)
 -\dim H^1({\mathcal F}^1)+2\dim_{\mathbf{C}}\mathbf{C} \\
 &=\dim H^0(({\mathcal F}^0)^{\vee}\otimes\Omega^1_{\cC_x})
 +\dim H^1({\mathcal F}^0)-\dim H^0({\mathcal F}^0)
 -\dim H^1(({\mathcal F}^0)^{\vee}\otimes\Omega^1_{\cC_x})
 +2 \\
 &=\dim H^1({\mathcal F}^0)^{\vee}+\dim H^1({\mathcal F}^0)
 -\dim H^0({\mathcal F}^0)-\dim H^0({\mathcal F}^0)^{\vee}
 +2 \\
 &=2-2\chi({\mathcal F}^0).
\end{align*}
Here we used the isomorphisms
${\mathcal F}^1\cong({\mathcal F}^0)^{\vee}\otimes\Omega^1_{\cC_x}$,
${\mathcal F}^0\cong({\mathcal F}^1)^{\vee}\otimes\Omega^1_{\cC_x}$
and Serre duality.
We define a subsheaf ${\mathcal E}_1\subset{\mathcal End}(E)$
by the exact sequence
\[
 0\longrightarrow {\mathcal E}_1 \longrightarrow {\mathcal End}(E)
 \longrightarrow \bigoplus_{i=1}^n
 {\mathcal Hom}_{{\mathcal O}_{m_i(\tilde{t}_i)_x}}(l^{(i)}_1,l^{(i)}_0/l^{(i)}_1)
 \longrightarrow 0.
\]
Inductively we define a subsheaf ${\mathcal E}_k\subset{\mathcal End}(E)$
by the exact sequence
\[
 0\longrightarrow {\mathcal E}_k \longrightarrow {\mathcal E}_{k-1}
 \longrightarrow \bigoplus_{i=1}^n
 {\mathcal Hom}_{{\mathcal O}_{m_i(\tilde{t}_i)_x}}(l^{(i)}_k,l^{(i)}_{k-1}/l^{(i)}_k)
 \longrightarrow 0.
\]
Then we have ${\mathcal E}_{r-1}={\mathcal F}^0$
and we have
\begin{align*}
 \chi({\mathcal F}^0)=\chi({\mathcal E}_{r-1})&=
 \chi({\mathcal End}(E))-\sum_{i=1}^n\sum_{j=1}^{r-1}
 \length\left({\mathcal Hom}_{{\mathcal O}_{m_i(\tilde{t}_i)_x}}
 (l^{(i)}_j,l^{(i)}_{j-1}/l^{(i)}_j)\right) \\
 &=r^2(1-g)-\sum_{i=1}^n\sum_{j=1}^{r-1}m_i(r-j) \\
 &=r^2(1-g)-r(r-1)\sum_{i=1}^n m_i/2.
\end{align*}
Thus we have
\[
 \dim\mathbf{H}^1({\mathcal F}^{\bullet})
 =2-2\chi({\mathcal F}^0)
 =2+2r^2(g-1)+r(r-1)\sum_{i=1}^nm_i
\]
and the statement of Proposition follows.
\end{proof}

{\it Proof of Theorem \ref{smoothness-thm}.}
By Proposition \ref{prop-smoothness-det},
we can see that
$M^{\balpha}_{D/\cC/S}(r,d,(m_i))$ is smooth
over $\cN^{(n)}_r(d,D)$
and by Proposition \ref{prop-moduli-dimension},
every fiber
$M^{\balpha}_{D/\cC/S}(r,d,(m_i))_{\bnu}$
over $\bnu\in\cN^{(n)}_r(d,D)$
is smooth of equidimension
$2r^2(g-1)+2+r(r-1)\sum_{i=1}^n m_i$.
So we obtain Theorem \ref{smoothness-thm}.
\hfill $\square$

\begin{Proposition}\label{prop-generic-fiber}
 Take $\bnu=(\nu^{(i)}_j)\in\cN^{(n)}_r(d,D)$
 and write $\nu^{(i)}_j=\sum_{k=-m_i}^{-1}a^{(i,j)}_kz_i^kdz_i$,
 where $(C,t_1,\ldots,t_n):=(\cC,\tilde{t}_1,\ldots,\tilde{t}_n)_{\bnu}$
 and $z_i$ is a generator of the maximal ideal of ${\mathcal O}_{C,t_i}$.
 Assume that $m_i>1$ and $a^{(i,j)}_{-m_i}\neq a^{(i,j')}_{-m_i}$ for any
 $i$ and any $j\neq j'$.
 Then the canonical morphism
 \begin{gather*}
  p:\tilde{M}^{\balpha}_{D/{\mathcal C}/S}(r,d,(N_i))_{\bnu}
  \longrightarrow
  M^{\balpha}_{D/{\mathcal C}/S}(r,d,(m_i))_{\bnu} \\
  (E,\nabla,\{l^{(i)}_j\})\mapsto
  \left(E,\nabla,\left\{l^{(i)}_j\otimes{\mathcal O}_{C,t_i}/(z_i^{m_i})\right\}\right)
 \end{gather*}
 is an isomorphism.
\end{Proposition}

\begin{proof}
Take any member
$(E,\nabla,\{l^{(i)}_j\})\in
M^{\balpha}_{D/{\mathcal C}/S}(r,d,(m_i))_{\bnu}(\mathbf{C})$.
We can see the following claim by Hukuhara-Turrittin theorem
(see [\cite{Sibuya}, Theorem 6.1.1):

\noindent
{\bf Claim.}
We have
$(E,\nabla)\otimes\mathbf{C}[[z_i]]\cong
V(\nu^{(i)}_{r-1},1)\oplus\cdots\oplus V(\nu^{(i)}_0,1)$. \\
By the above claim, we have
$l^{(i)}_j=V(\nu^{(i)}_{r-1},1)|_{m_it_i}\oplus\cdots\oplus V(\nu^{(i)}_j,1)|_{m_it_i}$.
If we set
$\tilde{l}^{(i)}_j:=V(\nu^{(i)}_{r-1},1)|_{N_it_i}\oplus\cdots\oplus V(\nu^{(i)}_j,1)|_{N_it_i}$,
then $(E,\nabla,\{\tilde{l}^{(i)}_j\})\in M^{\balpha}_{D/{\mathcal C}/S}(r,d,(N_i))_{\bnu}$
and
$p(E,\nabla,\{\tilde{l}^{(i)}_j\})=(E,\nabla,\{l^{(i)}_j\})$.
Thus $p$ is surjective. 

Take any member
$(E,\nabla,\{\overline{l}^{(i)}_j\})\in
M^{\balpha}_{D/{\mathcal C}/S}(r,d,(m_i))_{\bnu}(U)$
and two members
$(E,\nabla,\{l^{(i)}_j\})$, $(E,\nabla,\{(l')^{(i)}_j\})\in
p^{-1}(E,\nabla,\{\overline{l}^{(i)}_j\})$,
where $U$ is a scheme over $\mathbf{C}$.
Take any point $x\in U$ and a local section
$e'_{r-1}\in ((l')^{(i)}_{r-1})_x$
such that
$({\mathcal O}_{N_it_i}\otimes{\mathcal O}_{U,x})e'_{r-1}
=((l')^{(i)}_{r-1})_x$,
and
$\nabla(e'_{r-1})=\nu^{(i)}_{r-1}e'_{r-1}$.
Let $c_1$ be the image of $e'_{r-1}$ by the homomorphism
\[
 \pi_1:E|_{N_it_i\times U}\longrightarrow E|_{N_it_i\times U}/l^{(i)}_1\cong
 {\mathcal O}_{N_it_i\times U}.
\]
Then we have
\[
 c_1\nu^{(i)}_{r-1}=\pi_1(\nu^{(i)}_{r-1}e'_{r-1})
 =\pi_1\nabla (e'_{r-1})=\nabla\pi_1(e'_{r-1})
 =\nabla(c_1)=dc_1+c_1\nu^{(i)}_1.
\]
So we have
\[
 c_1(\nu^{(i)}_{r-1}-\nu^{(i)}_1)=dc_1.
\]
Since
\[
 \nu^{(i)}_{r-1}-\nu^{(i)}_1
 =(a^{(i,r-1)}_{-m_i}-a^{(i,1)}_{-m_i})z_i^{-m_i}dz_i+
 (a^{(i,r-1)}_{-m_i+1}-a^{(i,1)}_{-m_i+1})z_i^{-m_i+1}dz_i+\cdots
\]
and $a^{(i,r-1)}_{-m_i}-a^{(i,1)}_{-m_i}\in\mathbf{C}\setminus\{0\}$,
we have $c_1=0$.
Similarly, the projection of $e'_{r-1}$ to $E|_{N_it_i\times U}/l^{(i)}_j$
is zero for $j=1,\ldots,r-1$.
So we have $e'_{r-1}\in (l^{(i)}_{r-1})_x$
and so $(l')^{(i)}_{r-1}\subset l^{(i)}_{r-1}$.
Similarly we have $l^{(i)}_{r-1}\subset (l')^{(i)}_{r-1}$
and $l^{(i)}_{r-1}=(l')^{(i)}_{r-1}$.
By induction on $r$, we have
$l^{(i)}_j=(l')^{(i)}_j$ for $j=1,\ldots,r-1$.
So we have $(E,\nabla,\{l^{(i)}_j\})=(E,\nabla,\{(l')^{(i)}_j\})$.
Thus $p$ is a monomorphism.

Finally we will show that $p$ is smooth.
Let $A$ be an artinian local ring with maximal ideal $m$
and $I$ be an ideal of $A$ such that $mI=0$.
Assume that a commutative diagram
\[
 \begin{CD}
  \Spec(A/I) @>f>> \tilde{\mathcal M}^{\balpha}_{D/\cC/S}(r,d,(N_i))_{\bnu} \\
  @VVV @VV p V \\
  \Spec(A) @>g>> {\mathcal M}^{\balpha}_{D/\cC/S}(r,d,(m_i)) 
 \end{CD}
\]
is given.
Then $g$ corresponds to a member
$(E,\nabla,\{l^{(i)}_j\})\in {\mathcal M}^{\balpha}_{D/\cC/S}(r,d,(m_i))_{\bnu}(A)$
and $f$ corresponds to a member
$(E\otimes A/I,\nabla\otimes A/I,\{\bar{l}^{(i)}_j\})\in
\tilde{\mathcal M}^{\balpha}_{D/\cC/S}(r,d,(N_i))_{\bnu}(A/I)$.
Note that
$l^{(i)}_j=\bigoplus_{k=j}^{r-1}\ker(\nabla|_{m_it_i}-\nu^{(i)}_k)$
and that
$\overline{l}^{(i)}_j=\bigoplus_{k=j}^{r-1}\ker((\nabla\otimes A/I)|_{N_it_i}-\nu^{(i)}_k)$.
We can easily check that a canonical homomorphism
$\ker(\nabla|_{N_it_i}-\nu^{(i)}_j)\rightarrow
\ker(\nabla|_{m_it_i}-\nu^{(i)}_j)$ is surjective.
So the canonical homomorphism
$\varphi:\bigoplus_{j=0}^{r-1}\ker(\nabla|_{N_it_i}-\nu^{(i)}_j)
\rightarrow E|_{N_it_i}$
is surjective by Nakayama's lemma.
We can easily check that $\varphi$ is also injective.
If we put $\tilde{l}^{(i)}_j:=\bigoplus_{k=j}^{r-1}\ker(\nabla|_{N_it_i}-\nu^{(i)}_k)$,
then
$(E,\nabla,\{\tilde{l}^{(i)}_j\})\in \tilde{\mathcal M}^{\balpha}_{D/\cC/S}(r,d,(N_i))(A)$,
$p(E,\nabla,\{\tilde{l}^{(i)}_j\})=(E,\nabla,\{l^{(i)}_j\})$ and
$(E,\nabla,\{\tilde{l}^{(i)}_j\})\otimes A/I=
(E\otimes A/I,\nabla\otimes A/I,\{\bar{l}^{(i)}_j\})$.
Thus $p$ is a smooth morphism.

By the above proof, $p$ becomes bijective and \'etale.
Hence $p$ is an isomorphism.
\end{proof}

\begin{Remark}\label{not-smooth}\rm
In general the moduli space $M^{\balpha}_{D/\cC/S}(r,d,(N_i))$
is not smooth over $\cN^{(n)}_r(d,D)$
if $N_i>m_i$.
For example assume that $m_i>1$ for any $i$ and $g\geq 1$.
Then a general fiber $M^{\balpha}_{D/\cC/S}(r,d,(N_i))_{\bnu}$
over $\bnu\in\cN^{(n)}_r(d,D)$
is smooth of dimension
$2r^2(g-1)+2+r(r-1)\sum_{i=1}^nm_i$
by Proposition \ref{prop-generic-fiber} and
Theorem \ref{smoothness-thm}.
Take $x\in S$ and
put $C:=\cC_x$, $t_i:=(\tilde{t}_i)_x$
and $E:={\mathcal O}_C(-t_1)\oplus{\mathcal O}_C$.
Take a non-zero section $\omega\in H^0(C,\Omega^1_C((m_1+1)t_1))$
and consider the connection
\begin{gather*}
 \nabla:E\longrightarrow E\otimes\Omega^1_C\left(\sum_{i=1}^nm_it_i\right) \\
 \nabla\begin{pmatrix} f_1 \\ f_2 \end{pmatrix}=
 \begin{pmatrix} df_1 \\ df_2 \end{pmatrix}
 +
 \begin{pmatrix}
  0 & \omega \\
  0 & 0
 \end{pmatrix}
 \begin{pmatrix} f_1 \\ f_2 \end{pmatrix}
 \quad 
 (f_1\in{\mathcal O}_C(-t_1),\; f_2\in{\mathcal O}_C)
\end{gather*}
Then there is a canonical extension
$0\rightarrow{\mathcal O}_C(-t_1)\rightarrow E
\rightarrow{\mathcal O}_C\rightarrow 0$
which is compatible with the connections.
We take the parabolic structure
$l^{(i)}_1:={\mathcal O}_C(-t_1)|_{N_it_i}$.
If we take $\alpha^{(i)}_j\ll 1$ for any $i,j$,
$(E,\nabla,\{l^{(i)}_j\})$ becomes $\balpha$-stable.
We define a complex ${\mathcal F}^{\bullet}$ as follows:
\begin{align*}
 {\mathcal F}^0&:=\left\{ a\in{\mathcal End}(E) \left|
 \text{$a|_{N_it_i}(l^{(i)}_j)\subset l^{(i)}_j$ for any $i,j$}
 \right\}\right., \\
 {\mathcal F}^1&:=\left\{
 b\in{\mathcal End}(E)\otimes\Omega^1_C\left(\sum_{i=1}^nm_it_i\right)
 \left|
 \begin{array}{l}
 \text{$b|_{N_it_i}(l^{(i)}_j)\subset l^{(i)}_j\otimes\Omega^1_C(\sum_{i=1}^nm_it_i)$
 for any $i,j$} \\
 \text{$\overline{b^{(i)}_j}(l^{(i)}_j/l^{(i)}_{j+1})$ is contained in the image of} \\
 \text{$(l^{(i)}_j/l^{(i)}_{j+1})\otimes\Omega^1_C\rightarrow
 (l^{(i)}_j/l^{(i)}_{j+1})\otimes\Omega^1_C(m_it_i)$ for any $i,j$}
 \end{array}
 \right\}\right., \\
 \nabla_{{\mathcal F}^{\bullet}}&:{\mathcal F}^0
 \ni a \mapsto \nabla a-a\nabla \in{\mathcal F}^1,
\end{align*}
where
$\overline{b^{(i)}_j}:l^{(i)}_j/l^{(i)}_{j+1}\rightarrow
(l^{(i)}_j/l^{(i)}_{j+1})\otimes\Omega^1_C(m_it_i)$
is the homomorphism induced by $b|_{N_it_i}$.
We can see that the relative tangent space
$\Theta_{M^{\balpha}_{D/\cC/S}(2,-1,(m_i))/\cN^{(n)}_r(d,D)}\otimes k(x)$
at the point $x=(E,\nabla,\{l^{(i)}_j\})$ is isomorphic to
$\mathbf{H}^1({\mathcal F}^{\bullet})$.
{F}rom the spectral sequence
$H^q({\mathcal F}^p)\Rightarrow \mathbf{H}^{p+q}({\mathcal F}^{\bullet})$,
we obtain an exact sequence
\[
 0\longrightarrow \mathbf{C}\longrightarrow H^0({\mathcal F}^0)\longrightarrow
 H^0({\mathcal F}^1)\longrightarrow\mathbf{H}^1({\mathcal F}^{\bullet})
 \longrightarrow H^1({\mathcal F}^0)\longrightarrow H^1({\mathcal F}^1)
 \longrightarrow \mathbf{H}^2({\mathcal F}^{\bullet})\longrightarrow 0.
\]
So we have
\begin{align*}
 \dim\mathbf{H}^1({\mathcal F}^{\bullet})&=
 \dim H^0({\mathcal F}^1)+\dim H^1({\mathcal F}^0)-\dim H^0({\mathcal F}^0)
 -\dim H^1({\mathcal F}^1)+1+\dim\mathbf{H}^2({\mathcal F}^{\bullet}) \\
 &=\chi({\mathcal F}^1)-\chi({\mathcal F}^0)+1
 +\dim\mathbf{H}^2({\mathcal F}^{\bullet}) \\
 &=\left(2^2(1-g)+2^2(2g-2+\sum_{i=1}^nm_i)-\sum_{i=1}^n(N_i+2m_i)\right) \\
 &\quad -\left(2^2(1-g)-\sum_{i=1}^nN_i\right)+1+\dim\mathbf{H}^2({\mathcal F}^{\bullet}) \\
 &=8(g-1)+2+2\sum_{i=1}^nm_i+(\dim\mathbf{H}^2({\mathcal F}^{\bullet})-1).
\end{align*}
If we put \small
\begin{align*}
 &({\mathcal F}')^0:= \\
 &\left\{
 a\in{\mathcal End}(E)\otimes{\mathcal O}_C(\sum_{i=1}^n(N_i-m_i)t_i)\left|
 \begin{array}{l}
 \text{$a|_{N_it_i}(l^{(i)}_j)\subset l^{(i)}_j\otimes{\mathcal O}_C((N_i-m_i)t_i)$
 for any $i,j$} \\
 \text{and $\overline{a^{(i)}_j}(l^{(i)}_j/l^{(i)}_{j+1})$ is contained in the image of} \\
 \text{$(l^{(i)}_j/l^{(i)}_{j+1})\rightarrow(l^{(i)}_j/l^{(i)}_{j+1})
 \otimes{\mathcal O}_C((N_i-m_i)t_i)$ for any $i,j$}
 \end{array}
 \right\}\right., \\
 &({\mathcal F}')^1:=\left.\left\{
 b\in{\mathcal End}(E)\otimes\Omega^1_C(\sum_{i=1}^nN_it_i)
 \right|
  \text{$b|_{N_it_i}(l^{(i)}_j)\subset l^{(i)}_{j+1}\otimes\Omega^1_C(N_it_i)$
  for any $i,j$}
 \right\}, \\
 &\nabla_{({\mathcal F'})^{\bullet}}:({\mathcal F}')^0
 \ni a\mapsto \nabla a-a\nabla \in ({\mathcal F}')^1,
\end{align*}\normalsize
then we have
$({\mathcal F}^1)^{\vee}\otimes\Omega^1_C\cong
({\mathcal F}')^0$ and
$({\mathcal F}^0)^{\vee}\otimes\Omega^1_C\cong
({\mathcal F}')^1$.
We have
\begin{align*}
 \mathbf{H}^2({\mathcal F}^{\bullet})&\cong
 \coker\left(H^1({\mathcal F}^0)
 \xrightarrow{H^1(\nabla_{{\mathcal F}^{\bullet}})}
 H^1({\mathcal F}^1)\right) \\
 &\cong\ker\left(
 H^1({\mathcal F}^1)^{\vee}\xrightarrow{H^1(\nabla_{{\mathcal F}^{\bullet}})^{\vee}}
 H^1({\mathcal F}^0)^{\vee}\right)^{\vee} \\
 &\cong \ker \left(
 H^0(({\mathcal F}^1)^{\vee}\otimes\Omega^1_C)
 \xrightarrow{H^1(\nabla_{{\mathcal F}^{\bullet}})^{\vee}}
 H^0(({\mathcal F}^0)^{\vee}\otimes\Omega^1_C) \right)^{\vee} \\
 &\cong\ker\left(
 H^0(({\mathcal F}')^0)\xrightarrow{-H^0(\nabla_{({\mathcal F'})^{\bullet}})}
 H^0(({\mathcal F}')^1)\right)^{\vee}
\end{align*}
Note that
$\mathbf{C}\cdot\mathrm{id}_E\subset\ker
\left(H^0(({\mathcal F}')^0)\xrightarrow{-H^0(\nabla_{({\mathcal F'})^{\bullet}})}
H^0(({\mathcal F}')^1)\right)$.
Let $f:E\rightarrow E(t_1)$ be the composite
\[
 f:E\longrightarrow{\mathcal O}_C\hookrightarrow E(t_1).
\]
Since $f$ is compatible with the connections, we have
$f\in\ker
\left(H^0(({\mathcal F}')^0)\xrightarrow{-H^0(\nabla_{({\mathcal F'})^{\bullet}})}
H^0(({\mathcal F}')^1)\right)$.
Note that $0\neq f\notin\mathbf{C}\cdot\mathrm{id}_E$.
So we have $\dim\mathbf{H}^2({\mathcal F}^{\bullet})\geq 2$,
which means that
$\dim\mathbf{H}^1({\mathcal F}^{\bullet})\geq
8(g-1)+3+2\sum_{i=1}^nm_i$.
So $M^{\balpha}_{D/\cC/S}(2,-1,(N_i))$ is not smooth over
$\cN^{(n)}_r(d,D)$ at $x$.
\end{Remark}

\section{Smoothness of the family of the moduli spaces over configuration space}

Take any point $x\in S$.
If we put $t_i:=(\tilde{t}_i)_x$, we have
$D_x=\sum_{i=1}^n m_i t_i$.
Consider the Hilbert scheme
$H_i:=\Hilb^{m_i}_{\cC_x}$.
Put $H:=\Hilb^{m_1}_{\cC_x}\times\cdots\times\Hilb^{m_n}_{\cC_x}$
and let $D_i\subset\cC_x\times H$ be the universal divisors
for $i=1,\ldots,n$.
Note that $H$ is smooth over $\mathbf{C}$.
Let $H'\subset H$ be the open subscheme
such that
$H'=\{h\in H|\text{$(D_i)_h\cap(D_j)_h=\emptyset$ for $i\neq j$}\}$.
Consider the affine space bundle
\[
 \cN:=\prod_{i=1}^n
 \mathbf{V}_*\left((\pi_i)_*\left(\Omega^1_{\cC_x\times H'}
 ((D_i)_{H'})|_{(D_i)_{H_i}}\right)\right)
\]
over $H'$,
where $\pi_i:(D_i)_{H'}\rightarrow H'$ is the projection.
Take the universal family
$(\tilde{\nu}^{(i)}_j)$,
where
$\tilde{\nu}^{(i)}_j\in
H^0((D_i)_{\cN},\Omega^1_{\cC_x}((D_i)_{\cN})|_{(D_i)_{\cN}})$.

Assume that $\bnu=(\nu^{(i)}_j)\in\cN$ is given.
Let $h\in H'$ be the corresponding point and write
$(D_i)_h=\sum_km'_kt'_k$ with $t'_k\neq t'_j$ for $k\neq j$
and $\nu^{(i)}_j=\sum_k\nu'_k$ with
$\nu'_k\in H^0(m'_kt'_k,\Omega^1_{\cC_x}((D_i)_h)|_{m'_kt'_k})$.
Then we define $f^{(i)}_j(\bnu)=\sum_k\res_{t'_k}(\nu'_k)$.

Though it is obvious that the function $f^{(i)}_j$ defined above
is an algebraic function on $\cN$,
we give a proof by way of precaution.
We can take a disk $\Delta_k\subset\cC_x$ containing $t'_k$
such that $\overline{\Delta_k}\cap\overline{\Delta_{k'}}=\emptyset$
for $k\neq k'$.
Taking a sufficiently small analytic open neighborhood $U$ of $h$ in $H'$,
we can write $(D_i)_U=\sum_k D'_k$ with $D'_k$ an effective Cartier
divisor on $\cC_x\times U$ flat over $U$,
$(D'_k)_h=m'_kt'_k$ and $(D'_k)_g\subset\Delta_k$
for any $g\in U$.
Then we can write $(\tilde{\nu}^{(i)}_j)_{\cN_U}=\sum_k\tilde{\nu}'_k$
with $\tilde{\nu}'_k\in H^0\left((D'_k)_{\cN_U},
\Omega^1_{\cC_x}(D'_k)\otimes{\mathcal O}_{(D'_k)_{\cN_U}}\right)$.
By shrinking $U$ if necessary, we can take an open subset
$W'_k\subset\cC_x\times\cN_U$ such that
$\overline{\Delta_k}\times\cN_U\subset W'_k$
and a section
$\tilde{\omega}'_k\in
H^0(W'_k,\Omega^1_{\cC_x\times\cN_U/\cN_U}(D'_k)|_{W'_k})$
such that $\tilde{\omega}'_k|_{D'_k\times_U\cN_U}=\tilde{\nu}'_k$.
Then we have
\[
 (f^{(i)}_j)_U=\genfrac{}{}{}{}{1}{2\pi\sqrt{-1}}
 \sum_k\int_{\partial\Delta_k}\tilde{\omega}'_k.
\]
So $(f^{(i)}_j)_U$ becomes a holomorphic function on
$\cN_U$.
We can glue $(f^{(i)}_j)_U$ and obtain a holomorphic
function
$f^{(i)}_j$ on $\cN$.
Note that $f^{(i)}_j|_{\cN_{H^{\circ}}}$
is an algebraic function on $\cN_{H^{\circ}}$ by its definition,
where $H^{\circ}$ is the Zariski open subset of $H'$ defined by
\[
 H^{\circ}:=\{h\in H'|\text{$(D_i)_h$ has no multiple component}\}.
\]
So $f^{(i)}_j$ is a rational function on $\cN$,
which is holomorphic on $\cN$ and hence
$f^{(i)}_j$ becomes an algebraic function on $\cN$.

We define
\[
 \cN^{(n)}_r(d,(D_i)):=\left\{ \bnu\in \cN\left|
 d+\sum_{i=1}^n\sum_{j=0}^{r-1}f^{(i)}_j(\bnu)=0 \right\}\right..
\]
Then we can easily see that $\cN^{(n)}_r(d,(D_i))$
is smooth over $H'$.
We put $\tilde{D}:=\sum_{i=1}^n(D_i)_{\cN^{(n)}_r(d,(D_i))}$
and define a moduli functor
${\mathcal M}_{\cC_x}^{\balpha}(r,d,(D_i)):(\sch/\cN^{(n)}_r(d,(D_i)))
\rightarrow(\sets)$
by
\begin{align*}
 &{\mathcal M}^{\balpha}_{\cC_x}(r,d,(D_i))(T):= \\
 &\left\{ (E,\nabla,\{l^{(i)}_j\}) \left|
 \begin{array}{l}
  \text{$E$ is a vector bundle on $\cC_x\times T$ of rank $r$,} \\
  \text{$\nabla:E\rightarrow E\otimes\Omega^1_{\cC_x\times T/T}(\tilde{D}_T)$
  is a relative connection,} \\
  \text{$E|_{(D_i)_T}=l^{(i)}_0\supset l^{(i)}_1\supset\cdots\supset l^{(i)}_{r-1}
  \supset l^{(i)}_r=0$ is a filtration} \\
  \text{such that for any $i,j$, $l^{(i)}_j/l^{(i)}_{j+1}$ is a line bundle on $(D_i)_T$}, \\
  \text{$(\nabla|_{(D_i)_T}-(\tilde{\nu}^{(i)}_j)_T\mathrm{id}_{E|_{(D_i)_T}})(l^{(i)}_j)
  \subset l^{(i)}_{j+1}\otimes\Omega^1_{\cC_x\times T/T}(\tilde{D}_T)$} \\
  \text{$(E,\nabla,\{l^{(i)}_j\})\otimes k(y)$ satisfies the $\balpha$-stability
  $(\dag)$ below} \\
  \text{for any geometric point $y$ of $T$}
 \end{array}
 \right\}\right/\sim,
\end{align*}
where
$T$ is a locally noetherian scheme over $\cN^{(n)}_r(d,(D_i))$
and $(E,\nabla,\{l^{(i)}_j\})\sim(E',\nabla',\{(l')^{(i)}_j\})$
if there is a line bundle ${\mathcal L}$ on $T$ such that
$(E,\nabla,\{l^{(i)}_j\})\cong(E',\nabla',\{(l')^{(i)}_j\})\otimes{\mathcal L}$.
$(E,\nabla,\{l^{(i)}_j\})\otimes k(y)$ is $\balpha$-stable if
\begin{align*}
 (\dag) \hspace{20pt} &
 \text{for any subbundle $0\neq F\subsetneq E\otimes k(y)$
 with
 $(\nabla\otimes k(y))(F)\subset F\otimes\Omega^1_{\cC_y}(\tilde{D}_y)$,} \\
 &\genfrac{}{}{}{}{\deg F+\sum_{i=1}^n\sum_{j=1}^r\alpha^{(i)}_j
 \length\left( \left(F|_{(D_i)_y}\cap (l^{(i)}_{j-1}\otimes k(y))\right)
 /\left(F|_{(D_i)_y}\cap (l^{(i)}_j\otimes k(y)) \right) \right)}{\rank F} \\
 &<
 \genfrac{}{}{}{}{\deg (E\otimes k(y))+\sum_{i=1}^n\sum_{j=1}^r\alpha^{(i)}_j
 \length((l^{(i)}_{j-1}\otimes k(y))/(l^{(i)}_j\otimes k(y)))}{\rank E}.
\end{align*}

\begin{Theorem}\label{relative-smooth}
There exists a relative coarse moduli scheme
\[
 \pi:M^{\balpha}_{\cC_x}(r,d,(D_i))\longrightarrow
 \cN^{(n)}_r(d,(D_i))
\]
of ${\mathcal M}^{\balpha}_{\cC_x}(r,d,(D_i))$.
Moreover $\pi$ is a smooth morphism.
\end{Theorem}

\begin{proof}
We can see by the same argument as Theorem \ref{thm-moduli-exists}
that there exists a relative coarse moduli scheme
\[
 \pi:M^{\balpha}_{\cC_x}(r,d,(D_i))\longrightarrow
 \cN^{(n)}_r(d,(D_i))
\]
of ${\mathcal M}^{\balpha}_{\cC_x}(r,d,(D_i))$.
More precisely, $M^{\balpha}_{\cC_x}(r,d,(D_i))$ represents
the \'etale sheafification of
\par\noindent
${\mathcal M}^{\balpha}_{\cC_x}(r,d,(D_i))$.
We can define a morphism
\begin{align*}
 \det:& M^{\balpha}_{\cC_x}(r,d,(D_i))\longrightarrow
 M_{\cC_x}(1,d,(D_i))\times_{\cN^{(n)}_1(d,(D_i))}\cN^{(n)}_r(d,(D_i)) \\
 &(E,\nabla,\{l^{(i)}_j\})\mapsto
 \left((\det(E),\det(\nabla)),\pi(E,\nabla,\{l^{(i)}_j\})\right).
\end{align*}
Here $M_{\cC_x}(1,d,(D_i))$ is the moduli space of
line bundles with a connection.
We can construct $M_{\cC_x}(1,d,(D_i))$ as an affine space
bundle over $\Pic^d_{\cC_x} \times \cN^{(n)}_1(d,(D_i)) $
whose fiber is isomorphic to
$H^0(\Omega^1_{\cC_x})$.
So $M_{\cC_x}(1,d,(D_i))$ is smooth over $\cN^{(n)}_1(d,(D_i))$.
Let $A$ be an artinian local ring with maximal ideal $m$
and residue field $k=A/m$.
Assume that an ideal $I$ of $A$ such that $mI=0$ and
a commutative diagram
\[
 \begin{CD}
  \Spec(A/I) @>f>> {\mathcal M}^{\balpha}_{\cC_x}(r,d,(D_i)) \\
  @VVV @VV\det V \\
  \Spec(A) @>g>>
  M_{\cC_x}(1,d,(D_i))\times_{\cN^{(n)}_1(d,(D_i))}\cN^{(n)}_r(d,(D_i))
 \end{CD}
\]
are given.
$f$ corresponds to an $A/I$-valued point
$(E,\nabla,\{l^{(i)}_j\})\in{\mathcal M}^{\balpha}_{\cC_x}(r,d,(D_i))(A/I)$.
Put $(\overline{E},\overline{\nabla},\{\overline{l}^{(i)}_j\}):=
(E,\nabla,\{l^{(i)}_j\})\otimes A/m$.
Set \small
\begin{align*}
 {\mathcal F}^0&:=\left\{ a\in{\mathcal End}(\overline{E})
 \left| \text{$\Tr(a)=0$ and $a|_{(D_i)_k}(\overline{l}^{(i)}_j)\subset \overline{l}^{(i)}_j$
 for any $i,j$} \right\}\right., \\
 {\mathcal F}^1&:=\left\{
 b\in{\mathcal End}(\overline{E})\otimes\Omega^1_{(\cC_x)_k}(\tilde{D}_k)
 \left| \text{$\Tr(b)=0$ and $b|_{(D_i)_k}(\overline{l}^{(i)}_j)\subset
 \overline{l}^{(i)}_{j+1}\otimes\Omega^1_{(\cC_x)_k}(\tilde{D}_k)$ for any $i,j$}
 \right\}\right., \\
 \nabla_{{\mathcal F}^{\bullet}}&:{\mathcal F}^0
 \ni a\mapsto\overline{\nabla}a-a\overline{\nabla}\in{\mathcal F}^1.
\end{align*}\normalsize
Then we can see by the same argument as that of
Proposition \ref{prop-smoothness-det} that
there is an obstruction class
$\omega(E,\nabla,\{l^{(i)}_j\})\in
\mathbf{H}^2({\mathcal F}^{\bullet})\otimes I$
such that $\omega(E,\nabla,\{l^{(i)}_j\})=0$
if and only if
$(E,\nabla,\{l^{(i)}_j\})$ can be
lifted to an $A$-valued point
$(\tilde{E},\tilde{\nabla},\{\tilde{l}^{(i)}_j\})\in
{\mathcal M}^{\balpha}_{\cC_x}(r,d,(D_i))(A)$
such that
$(\tilde{E},\tilde{\nabla},\{\tilde{l}^{(i)}_j\})\otimes A/I
\cong(E,\nabla,\{l^{(i)}_j\})$ and
$\det(\tilde{E},\tilde{\nabla},\{\tilde{l}^{(i)}_j\})=g$.
Since $(\overline{E},\overline{\nabla},\{\overline{l}^{(i)}_j\})$
is $\balpha$-stable, we can see by the proof of
Proposition \ref{prop-smoothness-det} that
$\mathbf{H}^2({\mathcal F}^{\bullet})=0$.
So $\det$ is a smooth morphism.
Since $M_{\cC_x}(1,d,(D_i))$ is smooth over $\cN^{(n)}_1(d,(D_i))$,
we can see that
$M^{\balpha}_{\cC_x}(r,d,(D_i))$ is smooth over $\cN^{(n)}_r(d,(D_i))$.
\end{proof}

\section{Relative Symplectic form on the moduli space}

\begin{Theorem}\label{symplectic-form}
There exists a relative symplectic form
\[
 \omega\in H^0(M^{\balpha}_{D/\cC/S}(r,d,(m_i)),
 \Omega^2_{M^{\balpha}_{D/\cC/S}(r,d,(m_i))/\cN^{(n)}_r(d,D)}).
\]
\end{Theorem}

We prove Theorem \ref{symplectic-form} in several steps.

\begin{Proposition}\label{nondegenerate-form}
There exists a skew symmetric nondegenerate pairing
\[
 \omega:\Theta_{M^{\balpha}_{D/\cC/S}(r,d,(m_i))/\cN^{(n)}_r(d,D)}
 \times\Theta_{M^{\balpha}_{D/\cC/S}(r,d,(m_i))/\cN^{(n)}_r(d,D)}
 \longrightarrow
 {\mathcal O}_{M^{\balpha}_{D/\cC/S}(r,d,(m_i))}.
\]
\end{Proposition}

\begin{proof}
There are an affine scheme $U$
and an \'{e}tale surjective morphism
$p:U\rightarrow M^{\balpha}_{D/\cC/S}(r,d,(m_i))$
which factors through
${\mathcal M}^{\balpha}_{D/\cC/S}(r,d,(m_i))$,
namely there is a universal family
$(\tilde{E},\tilde{\nabla},\{\tilde{l}^{(i)}_j\})$
on $\cC\times_S U$.
We define a complex ${\mathcal F}^{\bullet}$
on $\cC\times_SU$ by
\begin{align*}
 {\mathcal F}^0&:=\left\{ a\in{\mathcal End}(\tilde{E}) \left|
 a|_{m_i(\tilde{t}_i)_U}(\tilde{l}^{(i)}_j)\subset\tilde{l}^{(i)}_j\;
 \text{for any $i,j$}
 \right\}\right., \\
 {\mathcal F}^1&:=\left\{
 b\in{\mathcal End}(\tilde{E})\otimes\Omega^1_{\cC/S}(D)
 \left|
 b|_{m_i(\tilde{t}_i)_U}(\tilde{l}^{(i)}_j)\subset
 \tilde{l}^{(i)}_{j+1}\otimes\Omega^1_{\cC/S}(D)
 \;\text{for any $i,j$}
 \right\}\right., \\
 \nabla_{{\mathcal F}^{\bullet}}&:{\mathcal F}^0
 \ni a \mapsto \tilde{\nabla}\circ a-a\circ\tilde{\nabla}
 \in {\mathcal F}^1.
\end{align*}
Let $\pi_U:\cC\times_SU\rightarrow U$ be the projection.
Then we have
\[
 \Theta_{U/\cN^{(n)}_r(d,D)}\cong
 p^*(\Theta_{M^{\balpha}_{D/\cC/S}(r,d,(m_i))/\cN^{(n)}_r(d,D)})
 \cong \mathbf{R}^1(\pi_U)_*({\mathcal F}^{\bullet}).
\]
Take an affine open covering
$\cC\times_SU=\bigcup_{\alpha}U_{\alpha}$
and a member
$v\in H^0(U,\mathbf{R}^1(\pi_U)_*({\mathcal F}^{\bullet}))
=\mathbf{H}^1(\cC\times_SU,{\mathcal F}^{\bullet}_U)$.
$v$ is given by
$[(\{u_{\alpha\beta}\},\{v_{\alpha}\})]$,
where
$\{u_{\alpha\beta}\}\in C^1(\{U_{\alpha}\},{\mathcal F}^0_U)$,
$\{v_{\alpha}\}\in C^0(\{U_{\alpha}\},{\mathcal F}^1_U)$
and
\[
 d\{u_{\alpha\beta}\}
 =\{u_{\beta\gamma}-u_{\alpha\gamma}+u_{\alpha\beta}\}=0,\;
 \nabla_{{\mathcal F}^{\bullet}}(\{u_{\alpha\beta}\})
 =\{v_{\beta}-v_{\alpha}\}=d\{v_{\alpha}\}.
\]
We define a pairing
\[
 \omega_U:\mathbf{H}^1(\cC\times_SU,{\mathcal F}^{\bullet})
 \times \mathbf{H}^1(\cC\times_SU,{\mathcal F}^{\bullet})
 \longrightarrow
 \mathbf{H}^2(\cC\times_SU,\Omega^{\bullet}_{\cC\times_SU/U})
 \cong H^0(U,{\mathcal O}_U)
\]
by
\[
 \omega_U([(\{u_{\alpha\beta}\},\{v_{\alpha}\})],
 [(\{u'_{\alpha\beta}\},\{v'_{\alpha}\})]):=
 [\{\Tr(u_{\alpha\beta}\circ u'_{\beta\gamma})\},
 -\{\Tr(u_{\alpha\beta}\circ v'_{\beta})-\Tr(v_{\alpha}\circ u'_{\alpha\beta})
 \}].
\]
By construction, $\omega_U$ is functorial in $U$.
So $\omega_U$ descends to a pairing
\[
 \omega:\Theta_{M^{\balpha}_{D/\cC/S}(r,d,(m_i))/\cN^{(n)}_r(d,D)}
 \times\Theta_{M^{\balpha}_{D/\cC/S}(r,d,(m_i))/\cN^{(n)}_r(d,D)}
 \longrightarrow {\mathcal O}_{M^{\balpha}_{D/\cC/S}(r,d,(m_i))}.
\]

Take any $\mathbf{C}$-valued point
$x=(E,\nabla,\{l^{(i)}_j\})\in
M^{\balpha}_{D/\cC/S}(r,d,(m_i))(\mathbf{C})$
and put $\bnu:=\pi(x)$.
Then a tangent vector
$v\in\Theta_{M^{\balpha}_{D/\cC/S}(r,d,(m_i))/\cN^{(n)}_r(d,D)}(x)
=\Theta_{M^{\balpha}_{D/\cC/S}(r,d,(m_i))_{\bnu}}(x)$
corresponds to a $\mathbf{C}[t]/(t^2)$-valued point
$(E^v,\nabla^v,\{(l^v)^{(i)}_j\})\in
{\mathcal M}^{\balpha}_{D/\cC/S}(r,d,(m_i))_{\bnu}(\mathbf{C}[t]/(t^2))$
such that
$(E^v,\nabla^v,\{(l^v)^{(i)}_j\})\otimes\mathbf{C}[t]/(t)
\cong(E,\nabla,\{l^{(i)}_j\})$.
We can check that
$\omega(v,v)$ is nothing but the obstruction class
for the lifting of $(E^v,\nabla^v,\{(l^v)^{(i)}_j\})$
to a member of
$M^{\balpha}_{D/\cC/S}(r,d,(m_i))_{\bnu}(\mathbf{C}[t]/(t^3))$.
Since $M^{\balpha}_{D/\cC/S}(r,d,(m_i))_{\bnu}$ is smooth,
we have $\omega(v,v)=0$.
So $\omega$ is a skew symmetric bilinear pairing.
Let
$\xi:\Theta_{M^{\balpha}_{D/\cC/S}(r,d,(m_i))/\cN^{(n)}_r(d,D)}
\rightarrow
\Theta_{M^{\balpha}_{D/\cC/S}(r,d,(m_i))/\cN^{(n)}_r(d,D)}^{\vee}$
be the homomorphism induced by $\omega$.
For any $\mathbf{C}$-valued point
$x\in M^{\balpha}_{D/\cC/S}(r,d,(m_i))(\mathbf{C})$,
\small
\[
 \xi(x):{\bf H}^1({\mathcal F}^{\bullet}(x))=
 \Theta_{M^{\balpha}_{D/\cC/S}(r,d,(m_i))/\cN^{(n)}_r(d,D)}(x)
 \longrightarrow
 \Theta_{M^{\balpha}_{D/\cC/S}(r,d,(m_i))/\cN^{(n)}_r(d,D)}^{\vee}(x)
 ={\mathbf H}^1({\mathcal F}^{\bullet}(x))^{\vee}
\]\normalsize
induces a commutative diagram
\small\[
 \begin{CD}
  H^0({\mathcal F}^0(x)) @>>> H^0({\mathcal F}^1(x)) @>>>
  \mathbf{H}^1({\mathcal F}^{\bullet}(x)) @>>>
  H^1({\mathcal F}^0(x)) @>>> H^1({\mathcal F}^1(x)) \\
  @V b_1 VV @V b_2 VV @V\xi VV @V b_3 VV @V b_4 VV \\
  H^1({\mathcal F}^1(x))^{\vee} @ >>> H^1({\mathcal F}^0(x))^{\vee}
  @>>> \mathbf{H}^1({\mathcal F}^1(x))^{\vee} @>>>
  H^0({\mathcal F}^1(x))^{\vee} @>>> H^1({\mathcal F}^0(x))^{\vee},
 \end{CD}
\]\normalsize
where $b_1,b_2,b_3,b_4$ are isomorphisms induced by
${\mathcal F}^0(x)\cong{\mathcal F}^1(x)^{\vee}\otimes\Omega^1_{\cC_x}$,
${\mathcal F}^1(x)\cong{\mathcal F}^0(x)^{\vee}\otimes\Omega^1_{\cC_x}$
and Serre duality.
Thus $\xi$ becomes an isomorphism by the five lemma.
\end{proof}

\begin{Proposition}\label{prop-d-closed}
 For the $2$-form $\omega$ constructed in
 Proposition \ref{nondegenerate-form},
 we have $d\omega=0$.
\end{Proposition}

\begin{proof}
Take any point $x\in S$.
We will show that
$d\omega|_{M^{\balpha}_{D/\cC/S}(r,d,(m_i))_x}=0$.
We use the notation in Theorem \ref{relative-smooth}.
Note that the relative moduli space
$M^{\balpha}_{\cC_x}(r,d,(D_i))$ is smooth over $\cN^{(n)}_r(d,(D_i))$.
There is an affine scheme $U$ and an \'etale surjective morphism
$p:U\rightarrow M^{\balpha}_{\cC_x}(r,d,(D_i))$
which factors through ${\mathcal M}^{\balpha}_{\cC_x}(r,d,(D_i))$,
namely there exists a universal family
$(\tilde{E},\tilde{\nabla},\{\tilde{l}^{(i)}_j\})$
on $\cC_x\times U$.
Set
\begin{align*}
 \tilde{\mathcal F}^0&:=\left\{ a\in{\mathcal End}(\tilde{E})
 \left| \text{$a|_{(D_i)_U}(\tilde{l}^{(i)}_j)\subset\tilde{l}^{(i)}_j$ for any $i,j$}
 \right\}\right., \\
 \tilde{\mathcal F}^1&:=\left\{
 b\in{\mathcal End}(\tilde{E})\otimes
 \Omega^1_{\cC_x\times U/U}(\tilde{D}_U)
 \left| \text{$b|_{(D_i)_U}(\tilde{l}^{(i)}_j)\subset\tilde{l}^{(i)}_{j+1}$ for any $i,j$}
 \right\}\right., \\
 \nabla_{\tilde{\mathcal F}^{\bullet}}&:\tilde{\mathcal F}^0
 \ni a \mapsto \tilde{\nabla}a-a\tilde{\nabla}
 \in\tilde{\mathcal F}^1.
\end{align*}
Then we have a canonical isomorphism
$H^0(U,p^*(\Theta_{M^{\balpha}_{\cC_x}(r,d,(D_i))/\cN^{(n)}_r(d,(D_i))}))
\cong\mathbf{H}^1(\tilde{\mathcal F}^{\bullet})$.
We can define a skew symmetric pairing
\begin{gather*}
 \tilde{\omega}_U:\mathbf{H}^1(\tilde{\mathcal F}^{\bullet})
 \times\mathbf{H}^1(\tilde{\mathcal F}^{\bullet})
 \longrightarrow
 \mathbf{H}^2(U,\Omega^1_{\cC_x\times U/U})\cong{\mathcal O}_U \\
 ([(\{u_{\alpha\beta}\},\{v_{\alpha}\})],[(\{u'_{\alpha\beta}\},\{v'_{\alpha}\})])
 \mapsto
 [(\{ \Tr (u_{\alpha\beta}\circ u'_{\beta\gamma}) \},
 -\{\Tr (u_{\alpha\beta}\circ v'_{\beta})-\Tr (v_{\alpha}\circ u'_{\alpha\beta})\}].
\end{gather*}
Since $\tilde{\omega}_U$ is functorial in $U$,
it descends to a $2$-form
$$
\tilde{\omega}\in H^0(M^{\balpha}_{\cC_x}(r,d,(D_i)),
\Theta_{M^{\balpha}_{\cC_x}(r,d,(D_i))/\cN^{(n)}_r(d,(D_i))}).
$$
By construction
the restriction $\tilde{\omega}|_{M^{\balpha}_{\cC_x}(r,d,(D_i))_{\overline{h}}}$
is nothing but the restriction
$\omega|_{M^{\balpha}_{D/\cC/S}(r,d,(m_i))_x}$ of
the $2$-form $\omega$ defined in Proposition \ref{nondegenerate-form}.
On the other hand, for generic $\bnu\in\cN^{(n)}_r(d,(D_i))$,
the fiber $M^{\balpha}_{\cC_x}(r,d,(D_i))_{\bnu}$ is
nothing but the moduli space of regular singular parabolic
connections considered in \cite{Inaba-1}.
Note that for generic $\bnu$,
every $\bnu$-parabolic connection is irreducible
and automatically stable.
Moreover the restriction
$\tilde{\omega}|_{M^{\balpha}_{\cC_x}(r,d,(D_i))_{\bnu}}$
is nothing but the restriction of the relative $2$-form
considered in [\cite{Inaba-1}, Proposition 7.2].
By [\cite{Inaba-1}, Proposition 7.3], we have
$d\tilde{\omega}|_{M^{\balpha}_{\cC_x}(r,d,(D_i))_{\bnu}}=0$.
Since $M^{\balpha}_{\cC_x}(r,d,(D_i))$ is smooth over
$\cN^{(n)}_r(d,(D_i))$, we have
$d\tilde{\omega}=0$.
So we have
$d\omega|_{M^{\balpha}_{D/\cC/S}(r,d,(m_i))_x}=
d\tilde{\omega}|_{M^{\balpha}_{\cC_x}(r,d,(D_i))_{\overline{h}}}=0$.
Hence we have $d\omega=0$.
\end{proof}

\section{Moduli spaces of generalized monodromy data and Riemann-Hilbert correspondence}

\subsection{Fixing the formal type}

Fix a nonsingular projective curve $C$ and a divisor 
$D = \sum_{i=1}^n m_i t_i$  on $C$ such that $m_i >0$,  $t_i \not= t_j$ 
for $  i \not= j$. 
At each point $t_i$, we take a generator $z_i $ of the maximal 
ideal ${\frak m}_{t_i}$   of $\cO_{C, t_i}$ then we have the formal completion 
$\widehat{\cO_{C, t_i}} = \lim_{k} \cO_{C, t_i}/{\frak m}_{t_i}^{k} 
\simeq \C[[z_i]]$. 

For given integers $r>0, d$, let us fix generalized exponents
$ \bnu= (\nu^{(i)}_j)_{1\leq i \leq n}^{0\leq j \leq r-1}  \in N_{r}^{(n)}(d, D)$ 
(cf. (\ref{eq:exponent})).   In Theorem \ref{thm-moduli-exists} and 
\ref{smoothness-thm}, we have constructed a smooth quasi-projective moduli scheme 
$
M^{\balpha}_{D/C}(r, d, (m_i))_{\bnu}
$
of $\balpha$-stable 
$\bnu$-parabolic connections  on $C$  of 
parabolic depth $(m_i)_{i=1}^r$,  
with $\rank$  $r$, $\deg$ $d $.  

For each fixed $\bnu$-parabolic connection $(E, \nabla, \{l^{(i)}_j \} ) \in 
M^{\balpha}_{D/C}(r, d, (m_i))_{\bnu}$, we can  define a formal connection by 
$$
\widehat{E}_{t_i} = E \otimes_{\cO_{C, t_i}}\C[[z_i]],  \quad 
\widehat{\nabla}_{t_i} : \widehat{E}_{t_i} \lra \widehat{E}_{t_i} \otimes 
\C[[z_i]] \frac{dz_i}{(z_i)^{m_i}}.
$$
In this section, we  assume 
that $(\widehat{E}_{t_i}, \widehat{\nabla}_{t_i})$ 
is unramified for each $i, 1 \leq i \leq n$, that is, in Hukuhara-Turrittin 
decomposition in Theorem \ref{thm-H-T}, 
$l=1$.   

By Proposition \ref{prop:filtration}, 
there exists a filtration by $\C[[z_i]]$-submodules 
$$
\widehat{E}_{t_i} = \widehat{l}^{(i)}_0 \supset \widehat{l}^{(i)}_1
\supset \widehat{l}^{(i)}_2 \supset \cdots \supset \widehat{l}^{(i)}_{r-1} 
\supset \widehat{l}^{(i)}_{r} = 0
$$
such that $\widehat{\nabla}_{t_i} (\widehat{l}^{(i)}_j) \subset \widehat{l}^{(i)}_j
\otimes dz_{i}/z_{i}^{m_i}$ and $\widehat{l}^{(i)}_j/\widehat{l}^{(i)}_{j+1} 
\simeq V(\tilde{\nu}^{(i)}_j, 1)$ where $\tilde{\bnu} = 
(\tilde{\nu}^{(i)}_j )_{1\leq i \leq n}^{0 \leq j \leq r-1} \in  N^{(n)}_r(d, D)$.  

The isomorphism class of $(\widehat{E}_{t_i}, \widehat{\nabla}_{t_i}, 
\{ \widehat{l}^{(i)}_j \}) $ at each $t_i$ as $\C[[z_i]]$-connection
is called the {\em formal type of 
the connection $(E, \nabla)$} at $t_i$.
For each $i$,   
the  data $\tilde{\bnu}^{(i)}:=(\tilde{\nu}^{(i)}_j )^{0\leq j \leq r-1} $ is 
called {\em formal generalized exponents} of  
$(\widehat{E}_{t_i}, \widehat{\nabla}_{t_i}, \{ \widehat{l}^{(i)}_j \}) $.  
Note that the original parabolic structure $\{ l^{(i)}_j \}$ is a filtration 
of $E \otimes_{\cO_{C, t_i}} \C[z_i]/(z_i^{m_i})$.  Moreover as we see in 
Remark \ref{counter-example}, $\bnu$ may not be equal to the formal 
generalized exponents $\tilde{\bnu}$.  

The main purpose of this section is  to define the Riemann-Hilbert correspondence 
from the moduli space $ M^{\balpha}_{D/C}(r, d, (m_i))_{\bnu}$ of $\bnu$-parabolic 
connections to the moduli space of {\em generalized monodromy data} consisting  of 
{\em monodromy representation} of fundamental group 
$\pi_{1}(C \setminus \{t_1, \cdots, t_n \}, \ast)$, {\em links (or connection matrices)}, 
{\em formal monodromies} and {\em Stokes data}.  
Moreover we may expect that the Riemann-Hilbert 
correspondence is a proper bimeromorphic surjective analytic morphism for any $\bnu$ 
as we proved in the case of at most regular singularities, that is, the case when  
$m_i=1$ for all $i, 1 \leq i \leq n$  (cf. \cite{Inaba-1}, \cite{IIS-1}, 
\cite{IIS-2}).  
 
As explained in \cite{JMU1}, \cite{vdP-S} and \cite{PS}, 
in order to construct the moduli space of generalized monodromy data and 
define the Riemann-Hilbert correspondence, 
we need to fix a formal type of the parabolic connection 
$(E, \nabla, \{l^{(i)}_j \} )$ at each 
irregular or regular singular point $t_i$.  
However the counter-example in Remark \ref{counter-example} shows 
that for a special $\bnu$, one can not determine the formal 
type of a connection $(E, \nabla, \{l^{(i)}_j \} ) \in 
M^{\balpha}_{D/C}(r, d, (m_i))_{\bnu}$, that is, the reductions up to the order $m_i$ 
is not enough to determine the formal type for a special $\bnu$.
Since we have Proposition \ref{prop-deeper}, 
we may take  deeper reductions of order $N_i=r^2 m_i > m_i$
 to recover the formal type. 
However, in Remark \ref{not-smooth}, we see that the corresponding
 moduli space $M^{\balpha}_{D/C}(r, d, (N_i))_{\bnu}$ is not 
 smooth.

At this moment, we do not know how to handle these difficulties.  
By this reason, we impose the following genericity conditions on $
\bnu= (\nu^{(i)}_j)_{1\leq i \leq n}^{0\leq j \leq r-1}  \in N_{r}^{(n)}(d, D)$.   

Let us write $\nu^{(i)}_j(z_i)$ explicitly as 
\begin{equation}\label{eq:local-exponent}
\nu^{(i)}_j (z_i) =(a^{(i)}_{j, -m_i} z_i^{-m_i} + \cdots + a^{(i)}_{j, -1} z_i^{-1 } )dz_i 
=\sum_{k=-m_i}^{-1} (a^{(i)}_{j, k} z_i^{k}) dz_i  
\quad \mbox{for $0 
\leq j \leq r-1$}. 
\end{equation}

\begin{Definition}\label{def:generic}{\rm 
Let $\bnu =  
\{ \nu^{(i)}_j(z_i) \}_{1 \leq i \leq n}^{0 \leq j \leq r-1} \in N_{r}^{(n)}(d, D)$ 
be written as in (\ref{eq:local-exponent}). 

\begin{enumerate}
\item 
$\bnu$   is {\em generic} if for every $(i, j_1), (i, j_2)$, 
$j_1 \not= j_2$, the top terms are different, that is, 
$a^{(i)}_{j_1, -m_i} \not= a^{(i)}_{j_2, -m_i}$. 

\item  $\bnu$ is {\em resonant} if for some $i, 1\leq i \leq n$ 
with $m_i=1$ 
there exists $j_1, j_2$,   $j_1 \not= j_2$ such that 
$$
a^{(i)}_{j_1, -1} - a^{(i)}_{j_2, -1} \in \Z .
$$
Moreover $\bnu$ is called {\em non-resonant} if it is not resonant.  

\item $\bnu$ is {\em reducible}, if for some 
$h, 1 \leq h  < r$, there exist some 
choices of $j^{(i)}_1, \cdots, j^{(i)}_h$, $0 \leq j^{(i)}_1 < j^{(i)}_2 < \cdots < j^{(i)}_h \leq r-1$
 for each $i, 1 \leq i \leq n$ such that 
\begin{equation}\label{eq:reducible}
\sum_{i=1}^n \sum_{k=1}^{h}a^{(i)}_{j^{(i)}_{k}, -1} \in \Z. 
\end{equation}
If $\bnu$ is not reducible, we call $\bnu$ {\em irreducible}. 
\end{enumerate}
 }
\end{Definition}
Note that the genericity and resonance
of $\bnu$ does not depend on the choice of 
the local coordinates $z_i$. 
An easy argument shows that if $\bnu$ is irreducible, 
every $\bnu$-parabolic connection 
$(E, \nabla, \{ l^{(i)}_j \} )$ is irreducible, 
hence $\balpha$-stable for any choice of 
the weights $\balpha$.

{F}rom now on, we assume that $\bnu$ is generic. 
{F}rom Hukuhara-Turrittin theorem ([\cite{Sibuya}, Theorem 6.1.1)]), 
it is easy to see the following Lemma.    

\begin{Lemma}\label{lemma:generic} Let $(E, \nabla, \{ l^{(i)}_j \} ) $ 
be an $\balpha$-stable $\bnu$-parabolic connection in $M^{\balpha}_{D/C}(r, d, (m_i))_{\bnu}$.  
Assume that $\bnu= \{ \nu^{(i)}_j(z_i) \} $ is generic.  
Then we have a direct sum decomposition of the formal connection
\begin{equation}\label{eq:direct}
(\widehat{E}_{t_i}, \widehat{\nabla}_{t_i} ) 
\simeq V(\nu^{(i)}_{0}, 1) \oplus V(\nu^{(i)}_1, 1) \oplus \cdots \oplus V(\nu^{(i)}_{r-1},1 ). 
\end{equation} Here $V(\nu^{(i)}_{j}, 1) \simeq \C[[z_i]] e^{(i)}_j $ is a rank $1$ $\C[[z_i]]$-module 
with  a connection given by $e^{(i)}_j \mapsto 
\nu^{(i)}_j(z_i) e^{(i)}_j $.  
In particular, the formal type of $(\widehat{E}_{t_i}, \widehat{\nabla}_{t_i} ) $
is uniquely determined by generalized exponents $ \{ \nu^{(i)}_j\}_{0 \leq j \leq r-1}$. 
Moreover the decomposition (\ref{eq:direct}) is compatible with the 
parabolic structure $\{ l^{(i)}_j \}_{0 \leq i \leq r-1}   $.  
\end{Lemma}

This lemma implies that there exists a free basis $e^{(i)}_0, \cdots, e^{(i)}_{r-1}$ of $\widehat{E}_{t_i}$
as a $\C[[z_i]]$-module, such that 
$$
\widehat{\nabla}_{t_i} e^{(i)}_j = \nu^{(i)}_j(z_i) e^{(i)}_j. 
$$
Moreover for $\widehat{E}_{t_i} \otimes \C[[z_i]]/(z_i^{m_i})$, the 
induced basis $\{ \overline{e}^{(i)}_j \}$ gives a parabolic structure
$$
l^{(i)}_k = < \overline{e}^{(i)}_k,\overline{e}^{(i)}_{k+1}, \cdots, \overline{e}^{(i)}_{r-1}>.
$$
Let us take  a generic 
$\bnu = \{ \bnu^{(i)} \}_{1 \leq i \leq n}  \in N_{r}^{(n)}(d, D)$.   
For each $t_i$, define {\em the space of formal solutions} at $t_i$ by 
\begin{equation}\label{eq:formal-dec}
V_{t_i} = \{ \sigma \in \widehat{E}_{t_i}\otimes_{\C[[z_i]]} Univ_{t_i} \ 
 | \widehat{\nabla}_{t_i} \sigma  = 0 \}
\end{equation}
where  $Univ_{t_i}$ denote the differential ring extension of 
$\C[[z_i]]$ which is similarly defined as in  [1.2, \cite{vdP-S}].    
Under the isomorphism of (\ref{eq:direct}), 
the space $V_{t_i}$  is a $\C$-vector space of dimension $r$ and has a natural 
decomposition 
\begin{equation}\label{eq:formal-decomp}
V_{t_i} = V^{(i)}_0 \oplus \cdots \oplus V^{(i)}_{r-1} 
\end{equation}
where $V^{(i)}_j = \C (f^{(i)}_{j}(z_i) e^{(i)}_j )$ is a one dimensional 
vector subspace and 
$f^{(i)}_j(z_i) = \exp (- \int \nu^{(i)}_j(z_i) ) \in Univ_{t_i}$.  Note that we have
{$d f^{(i)}_{j}(z_i) =  - f^{(i)}_j(z_i)\nu^{(i)}_{j}(z_i)  $.

\subsection{Generalized monodromy data}
\label{ss:mon}
 As in the former subsection, 
we  fix a nonsingular projective curve $C$ and a divisor 
$D = \sum_{i=1}^n m_i t_i$  on $C$ such that $m_i >0$,  $t_i \not= t_j$ 
for $  i \not= j$.  Moreover,  at each point $t_i$, we fix a generator $z_i $ of the maximal 
ideal ${\frak m}_{t_i}$   of $\cO_{C, t_i}$ so that we have the formal completion 
$\widehat{\cO_{C, t_i}} = \lim_{k} \cO_{C, t_i}/{\frak m}_{t_i}^{k} 
\simeq \C[[z_i]]$.   Let us fix  a generic  element $ \bnu \in N_{r}^{(n)}(d, D)$ written as 
in (\ref{eq:local-exponent}) .  
Then Lemma \ref{lemma:generic} implies that 
the formal types of every $\bnu$-parabolic connection 
$(E, \nabla, \{ l^{(i)}_j \}) \in M^{\balpha}_{D/C}(r, d, (m_i))_{\bnu}$
at $t_i$ can be fixed as in (\ref{eq:direct}). 
Fixing these data, we will associate a generalized monodromy data 
to each $(E, \nabla, \{ l^{(i)}_j \}) \in M^{\balpha}_{D/C}(r, d, (m_i))_{\bnu}$  as follows. 
We will basically follow the formulation in \cite{JMU1} and 
\cite{vdP-S} of genus $0$ case, which is easily generalized to higher genus case.
(For the known facts on generalized monodromy data,  see 
\cite{BV}, \cite{JMU1}, \cite{sabbah-book}, \cite{Sibuya} and \cite{PS}.)

\begin{enumerate}

\item{\bf Local coordinates:}  For each $i, 1 \leq i \leq n$, 
we consider the fixed  generator $z_i$ of the maximal ideal of $\cO_{C, t_i}$ as a local analytic 
coordinate around $t_i$ of  $C$. 

\item{\bf Local neighborhoods:}  
An analytic local neighborhood $\Delta_{i} \subset C $ of $t_i$ which is 
identified with $\{ z_i \ | \ |z_i| < \epsilon_i \}$ for a small 
positive number $\epsilon_i$.  

\item{\bf Singular directions and sectors:} 

Let us identify $d, 0 \leq d < 2 \pi$  as a ray 
starting from the origin $z_i=0$ with an argument $d$.  
Fixing a generic  $\bnu = \{ \nu^{(i)}_{j}(z_i) \} \in N_{r}^{(n)}(d, D)$,
we can define the singular directions 
$\{d^{(i)}_k \}^{1 \leq i \leq n}_{1 \leq k \leq s_i}$ such that 
$0 \leq  d^{(i)}_1 < d^{(i)}_2 < \cdots < d^{(i)}_{s_i} < 2 \pi$.
A direction $d$ at $t_i$ is called {\em singular} if for some 
$j_1 \not=j_2$ the function  
$\exp \left( \int \left(\nu^{(i)}_{j_1}- \nu^{(i)}_{j_2} \right) \right) $ 
has ``maximal descent" along the ray $z_i= r_i e^{\sqrt{-1} d}$ for $r_i \ra 0$.  
More explicitly, if 
$$
\nu^{(i)}_{j_1}- \nu^{(i)}_{j_2} = ((a^{(i)}_{j_1, -m_i} - a^{(i)}_{j_2, -m_i}) z_i^{-m_i} +  \cdots) dz_i, 
$$
$(a^{(i)}_{j_1, -m_i} - a^{(i)}_{j_2, -m_i})\not=0$, 
then $d$   is a singular direction if 
$$
- ((a^{(i)}_{j_1, -m_i} - a^{(i)}_{j_2, -m_i})e^{-\sqrt{-1}(m_i-1)d} r_i^{-(m_i-1)}/(m_i-1) )  
$$
is a negative real number.  (For more detail, see 1.3 of \cite{vdP-S})  .
For each $i, 1\leq i \leq n$, let  $0 \leq  d^{(i)}_1 < d^{(i)}_2 < \cdots < d^{(i)}_{s_i} < 2 \pi$ 
be all the singular directions at $t_i$.   In order to fix the order of  Stokes data at $t_i$, 
we take a  point $t^{\ast}_{i} \in \partial \Delta_i$ 
such that    $ d^{(i)}_{s_i} - 2 \pi <  \arg t^{\ast}_{i} < d^{(i)}_1$. (Later we will not impose this last condition for $t^{\ast}_i$ when we will vary the associated data continuously.)     
 We denote by $\gamma_i = \partial \Delta_i $ a closed counterclockwise loop starting 
 from $t^{\ast}_i$.   Moreover we set $ d^{(i)}_{0} = d^{(i)}_{s_i} - 2 \pi < 0 $.  
 For each $1 \leq k \leq s_i$, we define a  {\em sector} $S_k^{(i)}$ by 
\begin{equation}\label{eq:sector} 
S_k^{(i)} = \{ z_i \in \Delta_i \quad  | \quad  0 < |z_i| < \epsilon_i, \ 
d_{k-1}^{(i)} < \arg z_i < d^{(i)}_k \}. 
\end{equation} 
  (See Figure \ref{figure:1}).
For a singular direction $d$ at  $t_i$, 
let ${\mathcal J}(d, i )$ be the set of all pairs $(j_1, j_2) $ such that 
a singular direction of $\nu^{(i)}_{j_1} - \nu^{(i)}_{j_2}$ is $d$.  
The number $\sharp {\mathcal J}(d, i) $ is called the multiplicity of $d$ at $t_i$.
It is easy to see  
\begin{equation}\label{eq:diection}
 \sum_{1 \leq  k \leq s_i}  \sharp {\mathcal J}(d^{(i)}_k,  i) = (m_i-1) r(r-1). 
\end{equation}
Note that  if  the multiplicity $\sharp {\mathcal J}(d^{(i)}_k, i)$ is one for all $1 \leq k \leq s_i$, the number  of singular direction  is equal to $(m_i-1) r(r-1)$.

\item{\bf Paths and Loops:}

We fix a point $b$ on $C \setminus \{t_1, \cdots, t_n \}$ and a continuous path 
$l_i$  from $b$ to $t^{\ast}_i$. Let us set $\gamma^{l}_{i} := l_i \gamma_i l_i^{-1} $ for $1 \leq i \leq n$ and 
usual symplectic generators $\alpha_k, \beta_k$, $1 \leq k \leq g$  
of $\pi_1(C, b)$ so that the fundamental group 
$\pi_1(C \setminus \{t_1, \cdots, t_n \}, b )$ is generated by 
$\{ \gamma^{l}_i, \alpha_k, \beta_k \}$. 
Moreover we assume that 
our choice of paths $l_i$ and loops $\gamma_i$ $\alpha_k, \beta_l$ satisfies 
the conditions $
\prod_{k=1}^g [\alpha_k, \beta_k ] \prod_{i=1}^{n} \gamma^{l}_i  = 1,   
$
where $[\alpha_k, \beta_k ]= \alpha_k \beta_k \alpha_k^{-1} \beta_k^{-1}$.
Then we have the following presentation of the fundamental group.  
(See Figure \ref{figure:2}). 
\begin{equation}\label{eq:fund}
\pi_1(C \setminus \{t_1, \cdots, t_n \}, b ) = 
\langle \gamma^{l}_i, \alpha_k, \beta_k \ | \
 \prod_{k=1}^g [\alpha_k, \beta_k ] \prod_{i=1}^{n} \gamma^{l}_i  = 1
\rangle
\end{equation}

\item{\bf Spaces of formal solutions and analytic solutions:} 
Since we assume that $\bnu$ is generic, we can fix a decomposition  of 
the formal connection 
$
(\widehat{E}_{t_i}, \widehat{\nabla}_{t_i}) \simeq V(\nu^{(i)}_{0}, 1) 
\oplus V(\nu^{(i)}_1, 1) \oplus \cdots \oplus V(\nu^{(i)}_{r-1},1 )
$ as in (\ref{eq:direct}) 
and {\em the space of formal solutions} $ V_{t_i} $   as in (\ref{eq:formal-decomp}).  
Moreover we fix the space of analytic solutions  $V_b $ of $(E, \nabla)$
near $b$ which is a $\C$-vector space of dimension $r$.  
\end{enumerate}

\begin{figure}[h]
\begin{center}
\unitlength 0.1in
\begin{picture}(37.50,32.90)(1.50,-34.90)
%
\special{pn 13}%
\special{ar 2156 2089 1195 1196  1.0771791 6.2831853}%
\special{ar 2156 2089 1195 1196  0.0000000 1.0701705}%
%
\special{pn 13}%
\special{sh 1}%
\special{ar 2126 2109 10 10 0  6.28318530717959E+0000}%
\special{sh 1}%
\special{ar 2126 2109 10 10 0  6.28318530717959E+0000}%
%
\special{pn 13}%
\special{pa 2136 2119}%
\special{pa 3733 1852}%
\special{fp}%
%
\special{pn 13}%
\special{pa 2111 2109}%
\special{pa 3254 960}%
\special{fp}%
%
\special{pn 13}%
\special{pa 2121 2104}%
\special{pa 2355 500}%
\special{fp}%
%
\special{pn 13}%
\special{pa 2116 2114}%
\special{pa 3399 3103}%
\special{fp}%
\put(33.1000,-9.3000){\makebox(0,0)[lb]{$d^{(i)}_2$}}%
\put(34.4000,-22.1000){\makebox(0,0)[lb]{$ (d=0)$}}%
%
\special{pn 13}%
\special{ar 2156 2089 1195 1196  1.0771791 6.2831853}%
\special{ar 2156 2089 1195 1196  0.0000000 1.0701705}%
%
\special{pn 13}%
\special{sh 1}%
\special{ar 2126 2109 10 10 0  6.28318530717959E+0000}%
\special{sh 1}%
\special{ar 2126 2109 10 10 0  6.28318530717959E+0000}%
%
\special{pn 13}%
\special{pa 2136 2119}%
\special{pa 3733 1852}%
\special{fp}%
%
\special{pn 13}%
\special{pa 2111 2109}%
\special{pa 3254 960}%
\special{fp}%
%
\special{pn 13}%
\special{pa 2121 2104}%
\special{pa 2355 500}%
\special{fp}%
%
\special{pn 13}%
\special{pa 2116 2114}%
\special{pa 3399 3103}%
\special{fp}%
\put(39.0000,-19.4000){\makebox(0,0)[lb]{Singular direction $d^{(i)}_1$}}%
\put(35.4100,-32.5000){\makebox(0,0)[lb]{$d^{(i)}_{s_i}$}}%
\put(18.7600,-21.5900){\makebox(0,0)[lb]{$t_i$}}%
%
\special{pn 20}%
\special{sh 1}%
\special{ar 3341 2119 10 10 0  6.28318530717959E+0000}%
\special{sh 1}%
\special{ar 3345 2114 10 10 0  6.28318530717959E+0000}%
\put(1.5000,-33.8000){\makebox(0,0)[lb]{$\Delta_i=\{ z_i, |z_i| < \epsilon_i\}$}}%
%
\special{pn 13}%
\special{pa 2126 2134}%
\special{pa 3355 2134}%
\special{da 0.070}%
\put(22.8000,-15.2000){\makebox(0,0)[lb]{$S^{(i)}_3$}}%
\put(25.6000,-18.7000){\makebox(0,0)[lb]{Sector $S^{(i)}_2$}}%
%
\put(29.1500,-18.5800){\makebox(0,0)[lb]{}}%
\put(25.9000,-25.0000){\makebox(0,0)[lb]{Sector $S^{(i)}_1$}}%
\put(14.2000,-7.5000){\makebox(0,0)[lb]{$\gamma_i= \partial \Delta_i$}}%
%
\special{pn 13}%
\special{pa 1590 1020}%
\special{pa 1530 1070}%
\special{fp}%
\special{sh 1}%
\special{pa 1530 1070}%
\special{pa 1594 1043}%
\special{pa 1571 1036}%
\special{pa 1568 1012}%
\special{pa 1530 1070}%
\special{fp}%
\put(23.0000,-3.7000){\makebox(0,0)[lb]{$d^{(i)}_3$}}%
%
\put(29.5000,-21.4000){\makebox(0,0)[lb]{}}%
\put(34.5000,-23.9000){\makebox(0,0)[lb]{$t^{\ast}_i$}}%
%
\special{pn 13}%
\special{pa 2140 2160}%
\special{pa 2290 3400}%
\special{fp}%
%
\special{pn 13}%
\special{pa 2120 2120}%
\special{pa 3330 2300}%
\special{dt 0.045}%
\special{pa 3330 2300}%
\special{pa 3329 2300}%
\special{dt 0.045}%
\put(21.7000,-36.6000){\makebox(0,0)[lb]{$d^{(i)}_{s_i-1}$}}%
\put(23.9000,-28.2000){\makebox(0,0)[lb]{$S^{(i)}_{s_i}$}}%
%
\special{pn 4}%
\special{sh 1}%
\special{ar 3330 2310 10 10 0  6.28318530717959E+0000}%
\special{sh 1}%
\special{ar 3330 2310 10 10 0  6.28318530717959E+0000}%
\end{picture}%
\end{center}
\vspace{0.5cm}
\caption{Local neighborhood of $t_i$.}
\label{figure:1}
\end{figure}
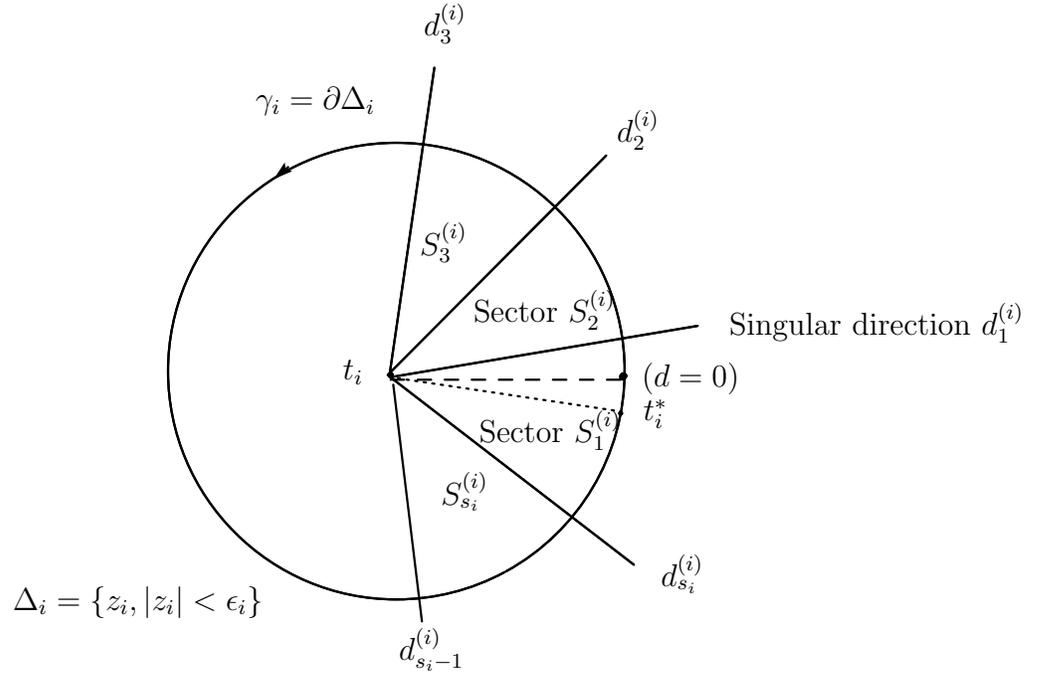

\begin{figure}[h]
\begin{center}
\unitlength 0.1in
\begin{picture}(54.20,29.24)(1.72,-31.90)
%
\special{pn 13}%
\special{pa 467 1205}%
\special{pa 486 1179}%
\special{pa 505 1153}%
\special{pa 524 1127}%
\special{pa 543 1101}%
\special{pa 563 1075}%
\special{pa 582 1049}%
\special{pa 602 1024}%
\special{pa 621 998}%
\special{pa 641 973}%
\special{pa 662 949}%
\special{pa 682 924}%
\special{pa 703 900}%
\special{pa 724 876}%
\special{pa 746 852}%
\special{pa 767 829}%
\special{pa 790 806}%
\special{pa 812 784}%
\special{pa 836 762}%
\special{pa 859 741}%
\special{pa 883 720}%
\special{pa 908 700}%
\special{pa 933 680}%
\special{pa 958 661}%
\special{pa 984 642}%
\special{pa 1011 624}%
\special{pa 1037 606}%
\special{pa 1064 589}%
\special{pa 1092 572}%
\special{pa 1120 555}%
\special{pa 1148 540}%
\special{pa 1176 524}%
\special{pa 1205 509}%
\special{pa 1234 495}%
\special{pa 1264 481}%
\special{pa 1293 467}%
\special{pa 1323 454}%
\special{pa 1353 442}%
\special{pa 1384 430}%
\special{pa 1414 418}%
\special{pa 1445 407}%
\special{pa 1476 396}%
\special{pa 1508 386}%
\special{pa 1539 376}%
\special{pa 1570 367}%
\special{pa 1602 358}%
\special{pa 1634 349}%
\special{pa 1666 341}%
\special{pa 1698 333}%
\special{pa 1730 326}%
\special{pa 1762 319}%
\special{pa 1795 313}%
\special{pa 1827 307}%
\special{pa 1859 302}%
\special{pa 1892 297}%
\special{pa 1924 292}%
\special{pa 1956 288}%
\special{pa 1989 284}%
\special{pa 2021 281}%
\special{pa 2053 278}%
\special{pa 2086 275}%
\special{pa 2118 273}%
\special{pa 2150 271}%
\special{pa 2182 269}%
\special{pa 2214 268}%
\special{pa 2247 267}%
\special{pa 2279 267}%
\special{pa 2311 266}%
\special{pa 2343 266}%
\special{pa 2375 266}%
\special{pa 2407 266}%
\special{pa 2439 267}%
\special{pa 2471 268}%
\special{pa 2502 268}%
\special{pa 2534 269}%
\special{pa 2566 270}%
\special{pa 2598 272}%
\special{pa 2630 273}%
\special{pa 2662 274}%
\special{pa 2694 276}%
\special{pa 2726 277}%
\special{pa 2758 278}%
\special{pa 2789 280}%
\special{pa 2821 281}%
\special{pa 2853 282}%
\special{pa 2885 284}%
\special{pa 2917 285}%
\special{pa 2949 286}%
\special{pa 2981 287}%
\special{pa 3013 288}%
\special{pa 3045 288}%
\special{pa 3077 289}%
\special{pa 3109 290}%
\special{pa 3141 290}%
\special{pa 3173 291}%
\special{pa 3205 291}%
\special{pa 3237 291}%
\special{pa 3269 292}%
\special{pa 3301 292}%
\special{pa 3333 292}%
\special{pa 3365 292}%
\special{pa 3397 292}%
\special{pa 3429 292}%
\special{pa 3461 292}%
\special{pa 3493 292}%
\special{pa 3525 292}%
\special{pa 3557 291}%
\special{pa 3589 291}%
\special{pa 3621 291}%
\special{pa 3653 291}%
\special{pa 3685 291}%
\special{pa 3717 290}%
\special{pa 3749 290}%
\special{pa 3781 290}%
\special{pa 3813 290}%
\special{pa 3844 290}%
\special{pa 3876 291}%
\special{pa 3908 291}%
\special{pa 3940 292}%
\special{pa 3972 293}%
\special{pa 4004 294}%
\special{pa 4036 296}%
\special{pa 4067 298}%
\special{pa 4099 300}%
\special{pa 4131 302}%
\special{pa 4163 305}%
\special{pa 4195 309}%
\special{pa 4227 313}%
\special{pa 4259 317}%
\special{pa 4290 322}%
\special{pa 4322 327}%
\special{pa 4354 333}%
\special{pa 4386 340}%
\special{pa 4418 347}%
\special{pa 4450 355}%
\special{pa 4482 364}%
\special{pa 4514 373}%
\special{pa 4546 383}%
\special{pa 4578 393}%
\special{pa 4610 405}%
\special{pa 4642 417}%
\special{pa 4673 429}%
\special{pa 4705 442}%
\special{pa 4736 456}%
\special{pa 4767 471}%
\special{pa 4798 486}%
\special{pa 4829 502}%
\special{pa 4860 518}%
\special{pa 4890 535}%
\special{pa 4920 553}%
\special{pa 4950 571}%
\special{pa 4979 590}%
\special{pa 5008 609}%
\special{pa 5037 629}%
\special{pa 5065 649}%
\special{pa 5092 670}%
\special{pa 5119 692}%
\special{pa 5146 714}%
\special{pa 5172 736}%
\special{pa 5198 759}%
\special{pa 5223 783}%
\special{pa 5247 807}%
\special{pa 5271 831}%
\special{pa 5294 856}%
\special{pa 5317 882}%
\special{pa 5339 908}%
\special{pa 5360 934}%
\special{pa 5380 961}%
\special{pa 5400 988}%
\special{pa 5418 1015}%
\special{pa 5436 1043}%
\special{pa 5453 1072}%
\special{pa 5470 1101}%
\special{pa 5485 1130}%
\special{pa 5499 1160}%
\special{pa 5513 1189}%
\special{pa 5525 1220}%
\special{pa 5537 1250}%
\special{pa 5547 1281}%
\special{pa 5557 1313}%
\special{pa 5565 1344}%
\special{pa 5572 1376}%
\special{pa 5578 1409}%
\special{pa 5583 1441}%
\special{pa 5587 1474}%
\special{pa 5590 1507}%
\special{pa 5591 1541}%
\special{pa 5592 1574}%
\special{pa 5591 1608}%
\special{pa 5589 1642}%
\special{pa 5586 1677}%
\special{pa 5582 1711}%
\special{pa 5577 1745}%
\special{pa 5571 1780}%
\special{pa 5564 1815}%
\special{pa 5556 1849}%
\special{pa 5547 1884}%
\special{pa 5537 1919}%
\special{pa 5526 1954}%
\special{pa 5514 1988}%
\special{pa 5502 2023}%
\special{fp}%
\special{pa 5502 2023}%
\special{pa 5488 2058}%
\special{pa 5474 2092}%
\special{pa 5459 2127}%
\special{pa 5443 2161}%
\special{pa 5426 2195}%
\special{pa 5409 2229}%
\special{pa 5391 2262}%
\special{pa 5372 2296}%
\special{pa 5353 2329}%
\special{pa 5332 2362}%
\special{pa 5312 2394}%
\special{pa 5290 2427}%
\special{pa 5268 2458}%
\special{pa 5246 2490}%
\special{pa 5223 2521}%
\special{pa 5199 2552}%
\special{pa 5175 2582}%
\special{pa 5150 2612}%
\special{pa 5125 2641}%
\special{pa 5100 2670}%
\special{pa 5074 2698}%
\special{pa 5048 2726}%
\special{pa 5021 2753}%
\special{pa 4994 2780}%
\special{pa 4966 2806}%
\special{pa 4939 2831}%
\special{pa 4911 2856}%
\special{pa 4882 2880}%
\special{pa 4854 2903}%
\special{pa 4825 2925}%
\special{pa 4796 2947}%
\special{pa 4767 2968}%
\special{pa 4738 2988}%
\special{pa 4708 3008}%
\special{pa 4678 3026}%
\special{pa 4649 3044}%
\special{pa 4619 3061}%
\special{pa 4589 3076}%
\special{pa 4559 3091}%
\special{pa 4529 3105}%
\special{pa 4500 3118}%
\special{pa 4470 3130}%
\special{pa 4440 3141}%
\special{pa 4410 3151}%
\special{pa 4380 3160}%
\special{pa 4351 3168}%
\special{pa 4322 3174}%
\special{pa 4292 3180}%
\special{pa 4263 3184}%
\special{pa 4234 3187}%
\special{pa 4206 3189}%
\special{pa 4177 3190}%
\special{pa 4149 3189}%
\special{pa 4121 3188}%
\special{pa 4094 3185}%
\special{pa 4066 3180}%
\special{pa 4039 3175}%
\special{pa 4013 3168}%
\special{pa 3986 3160}%
\special{pa 3960 3150}%
\special{pa 3934 3140}%
\special{pa 3908 3129}%
\special{pa 3883 3116}%
\special{pa 3858 3103}%
\special{pa 3832 3089}%
\special{pa 3807 3074}%
\special{pa 3783 3058}%
\special{pa 3758 3041}%
\special{pa 3733 3023}%
\special{pa 3708 3005}%
\special{pa 3684 2986}%
\special{pa 3659 2966}%
\special{pa 3635 2946}%
\special{pa 3610 2925}%
\special{pa 3586 2904}%
\special{pa 3561 2882}%
\special{pa 3536 2860}%
\special{pa 3512 2838}%
\special{pa 3487 2815}%
\special{pa 3462 2792}%
\special{pa 3437 2768}%
\special{pa 3412 2745}%
\special{pa 3387 2721}%
\special{pa 3361 2697}%
\special{pa 3336 2673}%
\special{pa 3310 2649}%
\special{pa 3284 2625}%
\special{pa 3257 2601}%
\special{pa 3231 2577}%
\special{pa 3204 2553}%
\special{pa 3177 2530}%
\special{pa 3149 2506}%
\special{pa 3121 2483}%
\special{pa 3093 2460}%
\special{pa 3065 2437}%
\special{pa 3036 2415}%
\special{pa 3007 2394}%
\special{pa 2977 2372}%
\special{pa 2947 2352}%
\special{pa 2916 2332}%
\special{pa 2886 2313}%
\special{pa 2855 2295}%
\special{pa 2825 2279}%
\special{pa 2794 2264}%
\special{pa 2764 2251}%
\special{pa 2734 2240}%
\special{pa 2704 2231}%
\special{pa 2675 2225}%
\special{pa 2646 2220}%
\special{pa 2618 2219}%
\special{pa 2590 2220}%
\special{pa 2564 2224}%
\special{pa 2538 2232}%
\special{pa 2513 2242}%
\special{pa 2488 2254}%
\special{pa 2464 2269}%
\special{pa 2440 2287}%
\special{pa 2417 2306}%
\special{pa 2395 2327}%
\special{pa 2372 2350}%
\special{pa 2350 2374}%
\special{pa 2328 2399}%
\special{pa 2306 2426}%
\special{pa 2283 2453}%
\special{pa 2261 2481}%
\special{pa 2239 2510}%
\special{pa 2217 2539}%
\special{pa 2194 2568}%
\special{pa 2171 2597}%
\special{pa 2148 2625}%
\special{pa 2124 2654}%
\special{pa 2100 2682}%
\special{pa 2075 2709}%
\special{pa 2050 2735}%
\special{pa 2024 2761}%
\special{pa 1998 2786}%
\special{pa 1971 2811}%
\special{pa 1944 2834}%
\special{pa 1916 2857}%
\special{pa 1889 2879}%
\special{pa 1861 2900}%
\special{pa 1832 2921}%
\special{pa 1804 2940}%
\special{pa 1775 2959}%
\special{pa 1745 2976}%
\special{pa 1716 2993}%
\special{pa 1686 3009}%
\special{pa 1657 3023}%
\special{pa 1627 3037}%
\special{pa 1597 3050}%
\special{pa 1566 3061}%
\special{pa 1536 3072}%
\special{pa 1506 3081}%
\special{pa 1475 3089}%
\special{pa 1445 3096}%
\special{pa 1415 3102}%
\special{pa 1384 3106}%
\special{pa 1354 3109}%
\special{pa 1324 3111}%
\special{pa 1293 3112}%
\special{pa 1263 3111}%
\special{pa 1233 3109}%
\special{pa 1203 3106}%
\special{pa 1174 3101}%
\special{pa 1144 3095}%
\special{pa 1115 3087}%
\special{pa 1086 3079}%
\special{pa 1057 3069}%
\special{pa 1028 3057}%
\special{pa 1000 3045}%
\special{pa 972 3031}%
\special{pa 944 3017}%
\special{pa 916 3001}%
\special{pa 889 2984}%
\special{pa 861 2967}%
\special{pa 834 2948}%
\special{pa 808 2929}%
\special{pa 781 2908}%
\special{pa 755 2887}%
\special{pa 730 2865}%
\special{pa 704 2843}%
\special{pa 679 2819}%
\special{pa 654 2795}%
\special{pa 630 2771}%
\special{pa 606 2746}%
\special{pa 582 2720}%
\special{pa 558 2694}%
\special{pa 535 2668}%
\special{pa 513 2641}%
\special{pa 491 2614}%
\special{pa 469 2586}%
\special{pa 447 2558}%
\special{pa 427 2530}%
\special{pa 406 2501}%
\special{pa 387 2473}%
\special{pa 368 2444}%
\special{pa 349 2414}%
\special{pa 332 2385}%
\special{pa 315 2355}%
\special{pa 299 2325}%
\special{fp}%
\special{pa 299 2325}%
\special{pa 283 2295}%
\special{pa 269 2264}%
\special{pa 255 2234}%
\special{pa 242 2203}%
\special{pa 230 2173}%
\special{pa 220 2142}%
\special{pa 210 2111}%
\special{pa 201 2080}%
\special{pa 193 2049}%
\special{pa 187 2018}%
\special{pa 181 1987}%
\special{pa 177 1956}%
\special{pa 174 1925}%
\special{pa 172 1894}%
\special{pa 172 1863}%
\special{pa 173 1833}%
\special{pa 175 1802}%
\special{pa 179 1771}%
\special{pa 184 1741}%
\special{pa 190 1711}%
\special{pa 198 1681}%
\special{pa 207 1651}%
\special{pa 218 1621}%
\special{pa 229 1591}%
\special{pa 242 1562}%
\special{pa 256 1533}%
\special{pa 270 1504}%
\special{pa 286 1475}%
\special{pa 302 1447}%
\special{pa 319 1418}%
\special{pa 336 1390}%
\special{pa 354 1363}%
\special{pa 372 1335}%
\special{pa 391 1308}%
\special{pa 410 1282}%
\special{pa 429 1255}%
\special{pa 449 1229}%
\special{pa 468 1203}%
\special{pa 488 1178}%
\special{pa 493 1171}%
\special{sp}%
%
\special{pn 13}%
\special{pa 913 2255}%
\special{pa 934 2231}%
\special{pa 956 2208}%
\special{pa 978 2185}%
\special{pa 1001 2162}%
\special{pa 1026 2142}%
\special{pa 1052 2122}%
\special{pa 1079 2104}%
\special{pa 1108 2088}%
\special{pa 1138 2074}%
\special{pa 1170 2062}%
\special{pa 1202 2052}%
\special{pa 1234 2044}%
\special{pa 1268 2039}%
\special{pa 1301 2035}%
\special{pa 1334 2034}%
\special{pa 1367 2036}%
\special{pa 1399 2040}%
\special{pa 1430 2046}%
\special{pa 1460 2055}%
\special{pa 1489 2067}%
\special{pa 1517 2081}%
\special{pa 1544 2097}%
\special{pa 1570 2115}%
\special{pa 1596 2135}%
\special{pa 1621 2156}%
\special{pa 1645 2178}%
\special{pa 1669 2201}%
\special{pa 1693 2224}%
\special{pa 1716 2248}%
\special{pa 1723 2255}%
\special{sp}%
%
\special{pn 13}%
\special{pa 3747 2315}%
\special{pa 3761 2343}%
\special{pa 3775 2372}%
\special{pa 3789 2400}%
\special{pa 3804 2428}%
\special{pa 3820 2456}%
\special{pa 3838 2484}%
\special{pa 3857 2511}%
\special{pa 3878 2538}%
\special{pa 3900 2564}%
\special{pa 3924 2589}%
\special{pa 3948 2612}%
\special{pa 3974 2634}%
\special{pa 4001 2654}%
\special{pa 4028 2672}%
\special{pa 4056 2688}%
\special{pa 4085 2702}%
\special{pa 4114 2712}%
\special{pa 4143 2719}%
\special{pa 4173 2723}%
\special{pa 4203 2724}%
\special{pa 4232 2722}%
\special{pa 4262 2716}%
\special{pa 4292 2708}%
\special{pa 4321 2698}%
\special{pa 4351 2685}%
\special{pa 4380 2670}%
\special{pa 4409 2653}%
\special{pa 4437 2634}%
\special{pa 4465 2614}%
\special{pa 4493 2592}%
\special{pa 4520 2569}%
\special{pa 4547 2546}%
\special{pa 4573 2521}%
\special{pa 4598 2496}%
\special{pa 4622 2470}%
\special{pa 4646 2444}%
\special{pa 4669 2418}%
\special{pa 4691 2392}%
\special{pa 4712 2367}%
\special{pa 4732 2342}%
\special{pa 4737 2335}%
\special{sp}%
%
\special{pn 13}%
\special{pa 847 2121}%
\special{pa 861 2149}%
\special{pa 875 2178}%
\special{pa 889 2206}%
\special{pa 905 2234}%
\special{pa 921 2262}%
\special{pa 939 2289}%
\special{pa 958 2316}%
\special{pa 979 2343}%
\special{pa 1001 2369}%
\special{pa 1025 2394}%
\special{pa 1050 2417}%
\special{pa 1075 2439}%
\special{pa 1102 2459}%
\special{pa 1130 2477}%
\special{pa 1158 2493}%
\special{pa 1186 2506}%
\special{pa 1216 2516}%
\special{pa 1245 2524}%
\special{pa 1275 2527}%
\special{pa 1304 2528}%
\special{pa 1334 2526}%
\special{pa 1364 2520}%
\special{pa 1394 2512}%
\special{pa 1423 2502}%
\special{pa 1453 2489}%
\special{pa 1482 2474}%
\special{pa 1511 2457}%
\special{pa 1539 2438}%
\special{pa 1567 2417}%
\special{pa 1595 2396}%
\special{pa 1622 2373}%
\special{pa 1649 2349}%
\special{pa 1675 2324}%
\special{pa 1700 2299}%
\special{pa 1724 2273}%
\special{pa 1748 2248}%
\special{pa 1771 2222}%
\special{pa 1793 2196}%
\special{pa 1813 2170}%
\special{pa 1834 2145}%
\special{pa 1837 2141}%
\special{sp}%
%
\special{pn 13}%
\special{pa 3833 2441}%
\special{pa 3854 2417}%
\special{pa 3876 2393}%
\special{pa 3898 2370}%
\special{pa 3921 2348}%
\special{pa 3945 2327}%
\special{pa 3970 2307}%
\special{pa 3998 2289}%
\special{pa 4027 2273}%
\special{pa 4057 2259}%
\special{pa 4088 2246}%
\special{pa 4120 2236}%
\special{pa 4153 2228}%
\special{pa 4186 2223}%
\special{pa 4219 2219}%
\special{pa 4252 2218}%
\special{pa 4285 2219}%
\special{pa 4317 2223}%
\special{pa 4349 2230}%
\special{pa 4379 2239}%
\special{pa 4408 2251}%
\special{pa 4436 2265}%
\special{pa 4463 2281}%
\special{pa 4489 2299}%
\special{pa 4514 2319}%
\special{pa 4539 2340}%
\special{pa 4563 2362}%
\special{pa 4587 2385}%
\special{pa 4611 2408}%
\special{pa 4634 2432}%
\special{pa 4643 2441}%
\special{sp}%
%
\special{pn 20}%
\special{sh 1}%
\special{ar 2713 1741 10 10 0  6.28318530717959E+0000}%
\special{sh 1}%
\special{ar 2713 1741 10 10 0  6.28318530717959E+0000}%
%
\special{pn 13}%
\special{pa 2713 1745}%
\special{pa 2691 1766}%
\special{pa 2669 1788}%
\special{pa 2647 1809}%
\special{pa 2625 1830}%
\special{pa 2603 1852}%
\special{pa 2581 1873}%
\special{pa 2559 1894}%
\special{pa 2536 1916}%
\special{pa 2514 1937}%
\special{pa 2491 1958}%
\special{pa 2469 1979}%
\special{pa 2446 2000}%
\special{pa 2423 2020}%
\special{pa 2400 2041}%
\special{pa 2377 2062}%
\special{pa 2353 2082}%
\special{pa 2330 2103}%
\special{pa 2306 2123}%
\special{pa 2282 2143}%
\special{pa 2258 2163}%
\special{pa 2233 2183}%
\special{pa 2209 2203}%
\special{pa 2184 2222}%
\special{pa 2159 2242}%
\special{pa 2133 2261}%
\special{pa 2107 2280}%
\special{pa 2081 2299}%
\special{pa 2055 2317}%
\special{pa 2028 2336}%
\special{pa 2001 2354}%
\special{pa 1973 2372}%
\special{pa 1946 2390}%
\special{pa 1917 2407}%
\special{pa 1889 2425}%
\special{pa 1860 2442}%
\special{pa 1830 2458}%
\special{pa 1801 2475}%
\special{pa 1770 2491}%
\special{pa 1740 2507}%
\special{pa 1709 2523}%
\special{pa 1677 2538}%
\special{pa 1645 2553}%
\special{pa 1612 2568}%
\special{pa 1579 2583}%
\special{pa 1546 2597}%
\special{pa 1512 2610}%
\special{pa 1477 2624}%
\special{pa 1443 2636}%
\special{pa 1408 2647}%
\special{pa 1374 2658}%
\special{pa 1339 2668}%
\special{pa 1304 2676}%
\special{pa 1270 2684}%
\special{pa 1236 2690}%
\special{pa 1202 2694}%
\special{pa 1168 2697}%
\special{pa 1135 2699}%
\special{pa 1103 2699}%
\special{pa 1071 2696}%
\special{pa 1040 2692}%
\special{pa 1010 2686}%
\special{pa 981 2678}%
\special{pa 953 2667}%
\special{pa 925 2654}%
\special{pa 899 2639}%
\special{pa 874 2621}%
\special{pa 851 2600}%
\special{pa 829 2577}%
\special{pa 808 2552}%
\special{pa 788 2524}%
\special{pa 771 2495}%
\special{pa 754 2463}%
\special{pa 740 2431}%
\special{pa 727 2397}%
\special{pa 715 2363}%
\special{pa 706 2327}%
\special{pa 698 2291}%
\special{pa 693 2255}%
\special{pa 689 2219}%
\special{pa 687 2183}%
\special{pa 687 2147}%
\special{pa 689 2112}%
\special{pa 694 2078}%
\special{pa 701 2045}%
\special{pa 709 2014}%
\special{pa 721 1984}%
\special{pa 734 1955}%
\special{pa 750 1929}%
\special{pa 769 1905}%
\special{pa 789 1883}%
\special{pa 812 1864}%
\special{pa 837 1846}%
\special{pa 864 1830}%
\special{pa 892 1816}%
\special{pa 922 1803}%
\special{pa 953 1792}%
\special{pa 985 1782}%
\special{pa 1018 1773}%
\special{pa 1051 1766}%
\special{pa 1086 1759}%
\special{pa 1120 1754}%
\special{pa 1155 1749}%
\special{pa 1190 1745}%
\special{pa 1225 1741}%
\special{pa 1260 1738}%
\special{pa 1294 1735}%
\special{pa 1328 1733}%
\special{pa 1362 1731}%
\special{pa 1395 1729}%
\special{pa 1428 1727}%
\special{pa 1461 1726}%
\special{pa 1494 1725}%
\special{pa 1526 1724}%
\special{pa 1559 1724}%
\special{pa 1591 1724}%
\special{pa 1623 1724}%
\special{pa 1655 1724}%
\special{pa 1686 1725}%
\special{pa 1718 1725}%
\special{pa 1749 1726}%
\special{pa 1781 1727}%
\special{pa 1812 1728}%
\special{pa 1843 1729}%
\special{pa 1874 1730}%
\special{pa 1905 1731}%
\special{pa 1936 1732}%
\special{pa 1968 1733}%
\special{pa 1999 1734}%
\special{pa 2030 1735}%
\special{pa 2061 1737}%
\special{pa 2092 1738}%
\special{pa 2123 1739}%
\special{pa 2154 1740}%
\special{pa 2186 1741}%
\special{pa 2217 1741}%
\special{pa 2249 1742}%
\special{pa 2280 1742}%
\special{pa 2312 1743}%
\special{pa 2344 1743}%
\special{pa 2376 1743}%
\special{pa 2408 1743}%
\special{pa 2440 1743}%
\special{pa 2472 1743}%
\special{pa 2504 1743}%
\special{pa 2536 1743}%
\special{pa 2568 1742}%
\special{pa 2600 1742}%
\special{pa 2633 1742}%
\special{pa 2665 1741}%
\special{pa 2697 1741}%
\special{pa 2707 1741}%
\special{sp}%
%
\special{pn 13}%
\special{pa 2723 1745}%
\special{pa 2704 1775}%
\special{pa 2685 1805}%
\special{pa 2667 1835}%
\special{pa 2648 1866}%
\special{pa 2629 1896}%
\special{pa 2610 1926}%
\special{pa 2591 1956}%
\special{pa 2572 1985}%
\special{pa 2553 2015}%
\special{pa 2534 2045}%
\special{pa 2515 2074}%
\special{pa 2495 2103}%
\special{pa 2476 2133}%
\special{pa 2457 2162}%
\special{pa 2437 2190}%
\special{pa 2417 2219}%
\special{pa 2398 2247}%
\special{pa 2378 2276}%
\special{pa 2358 2304}%
\special{pa 2338 2331}%
\special{pa 2318 2359}%
\special{pa 2297 2386}%
\special{pa 2277 2413}%
\special{pa 2256 2440}%
\special{pa 2235 2466}%
\special{pa 2214 2492}%
\special{pa 2193 2518}%
\special{pa 2172 2543}%
\special{pa 2150 2568}%
\special{pa 2128 2592}%
\special{pa 2106 2617}%
\special{pa 2084 2640}%
\special{pa 2062 2664}%
\special{pa 2039 2687}%
\special{pa 2016 2709}%
\special{pa 1993 2731}%
\special{pa 1970 2753}%
\special{pa 1946 2774}%
\special{pa 1922 2794}%
\special{pa 1898 2815}%
\special{pa 1874 2834}%
\special{pa 1849 2853}%
\special{pa 1824 2872}%
\special{pa 1799 2890}%
\special{pa 1774 2907}%
\special{pa 1748 2924}%
\special{pa 1722 2940}%
\special{pa 1695 2956}%
\special{pa 1668 2971}%
\special{pa 1641 2985}%
\special{pa 1614 2999}%
\special{pa 1586 3012}%
\special{pa 1558 3024}%
\special{pa 1529 3036}%
\special{pa 1500 3047}%
\special{pa 1471 3057}%
\special{pa 1441 3067}%
\special{pa 1411 3076}%
\special{pa 1381 3084}%
\special{pa 1350 3091}%
\special{pa 1319 3098}%
\special{pa 1303 3101}%
\special{sp}%
%
\special{pn 13}%
\special{pa 1293 3105}%
\special{pa 1266 3084}%
\special{pa 1240 3063}%
\special{pa 1216 3041}%
\special{pa 1192 3019}%
\special{pa 1171 2995}%
\special{pa 1153 2969}%
\special{pa 1138 2942}%
\special{pa 1127 2913}%
\special{pa 1119 2882}%
\special{pa 1115 2849}%
\special{pa 1115 2817}%
\special{pa 1118 2784}%
\special{pa 1124 2753}%
\special{pa 1132 2722}%
\special{pa 1143 2692}%
\special{pa 1155 2663}%
\special{pa 1169 2634}%
\special{pa 1185 2605}%
\special{pa 1201 2577}%
\special{pa 1218 2549}%
\special{pa 1233 2525}%
\special{sp -0.045}%
%
\special{pn 13}%
\special{pa 1233 2535}%
\special{pa 1266 2531}%
\special{pa 1298 2528}%
\special{pa 1331 2524}%
\special{pa 1363 2520}%
\special{pa 1395 2515}%
\special{pa 1428 2511}%
\special{pa 1460 2505}%
\special{pa 1491 2500}%
\special{pa 1523 2494}%
\special{pa 1554 2487}%
\special{pa 1585 2480}%
\special{pa 1616 2472}%
\special{pa 1646 2463}%
\special{pa 1676 2454}%
\special{pa 1706 2443}%
\special{pa 1735 2432}%
\special{pa 1763 2420}%
\special{pa 1791 2406}%
\special{pa 1819 2392}%
\special{pa 1846 2377}%
\special{pa 1873 2360}%
\special{pa 1900 2344}%
\special{pa 1926 2326}%
\special{pa 1952 2308}%
\special{pa 1978 2289}%
\special{pa 2003 2270}%
\special{pa 2029 2250}%
\special{pa 2054 2230}%
\special{pa 2079 2209}%
\special{pa 2103 2188}%
\special{pa 2128 2167}%
\special{pa 2153 2146}%
\special{pa 2177 2124}%
\special{pa 2202 2102}%
\special{pa 2226 2081}%
\special{pa 2251 2059}%
\special{pa 2275 2037}%
\special{pa 2300 2016}%
\special{pa 2325 1994}%
\special{pa 2350 1973}%
\special{pa 2375 1952}%
\special{pa 2400 1932}%
\special{pa 2425 1911}%
\special{pa 2451 1892}%
\special{pa 2477 1872}%
\special{pa 2503 1854}%
\special{pa 2530 1836}%
\special{pa 2557 1818}%
\special{pa 2584 1802}%
\special{pa 2612 1786}%
\special{pa 2640 1771}%
\special{pa 2669 1757}%
\special{pa 2698 1743}%
\special{pa 2727 1731}%
\special{sp}%
%
\special{pn 13}%
\special{pa 2767 1741}%
\special{pa 2790 1764}%
\special{pa 2812 1786}%
\special{pa 2835 1809}%
\special{pa 2858 1831}%
\special{pa 2881 1854}%
\special{pa 2903 1877}%
\special{pa 2926 1899}%
\special{pa 2949 1922}%
\special{pa 2972 1944}%
\special{pa 2995 1966}%
\special{pa 3018 1989}%
\special{pa 3041 2011}%
\special{pa 3064 2033}%
\special{pa 3087 2056}%
\special{pa 3110 2078}%
\special{pa 3133 2100}%
\special{pa 3157 2122}%
\special{pa 3180 2144}%
\special{pa 3203 2166}%
\special{pa 3227 2187}%
\special{pa 3250 2209}%
\special{pa 3274 2231}%
\special{pa 3298 2252}%
\special{pa 3321 2273}%
\special{pa 3345 2295}%
\special{pa 3369 2316}%
\special{pa 3394 2337}%
\special{pa 3418 2358}%
\special{pa 3442 2378}%
\special{pa 3467 2399}%
\special{pa 3491 2420}%
\special{pa 3516 2440}%
\special{pa 3541 2460}%
\special{pa 3566 2480}%
\special{pa 3591 2500}%
\special{pa 3616 2520}%
\special{pa 3641 2540}%
\special{pa 3667 2559}%
\special{pa 3692 2578}%
\special{pa 3718 2597}%
\special{pa 3744 2616}%
\special{pa 3770 2635}%
\special{pa 3797 2654}%
\special{pa 3823 2672}%
\special{pa 3850 2690}%
\special{pa 3877 2708}%
\special{pa 3904 2725}%
\special{pa 3932 2742}%
\special{pa 3959 2758}%
\special{pa 3987 2773}%
\special{pa 4016 2788}%
\special{pa 4045 2802}%
\special{pa 4074 2815}%
\special{pa 4104 2827}%
\special{pa 4134 2838}%
\special{pa 4164 2848}%
\special{pa 4196 2857}%
\special{pa 4227 2865}%
\special{pa 4259 2871}%
\special{pa 4292 2876}%
\special{pa 4324 2880}%
\special{pa 4357 2883}%
\special{pa 4390 2884}%
\special{pa 4423 2883}%
\special{pa 4456 2882}%
\special{pa 4490 2879}%
\special{pa 4522 2875}%
\special{pa 4555 2869}%
\special{pa 4588 2862}%
\special{pa 4620 2853}%
\special{pa 4652 2843}%
\special{pa 4683 2831}%
\special{pa 4714 2818}%
\special{pa 4744 2803}%
\special{pa 4774 2787}%
\special{pa 4803 2769}%
\special{pa 4831 2750}%
\special{pa 4858 2729}%
\special{pa 4884 2706}%
\special{pa 4910 2682}%
\special{pa 4934 2656}%
\special{pa 4956 2629}%
\special{pa 4978 2601}%
\special{pa 4998 2572}%
\special{pa 5016 2542}%
\special{pa 5032 2511}%
\special{pa 5046 2480}%
\special{pa 5058 2449}%
\special{pa 5067 2418}%
\special{pa 5074 2386}%
\special{pa 5079 2355}%
\special{pa 5081 2323}%
\special{pa 5080 2293}%
\special{pa 5075 2262}%
\special{pa 5068 2233}%
\special{pa 5058 2204}%
\special{pa 5044 2176}%
\special{pa 5028 2149}%
\special{pa 5010 2123}%
\special{pa 4989 2098}%
\special{pa 4966 2074}%
\special{pa 4942 2051}%
\special{pa 4915 2029}%
\special{pa 4888 2009}%
\special{pa 4859 1990}%
\special{pa 4828 1972}%
\special{pa 4797 1956}%
\special{pa 4766 1941}%
\special{pa 4734 1928}%
\special{pa 4702 1916}%
\special{pa 4669 1906}%
\special{pa 4637 1897}%
\special{pa 4605 1890}%
\special{pa 4573 1884}%
\special{pa 4541 1880}%
\special{pa 4509 1877}%
\special{pa 4477 1875}%
\special{pa 4445 1873}%
\special{pa 4413 1873}%
\special{pa 4381 1873}%
\special{pa 4349 1874}%
\special{pa 4317 1875}%
\special{pa 4285 1876}%
\special{pa 4254 1878}%
\special{pa 4222 1880}%
\special{pa 4190 1882}%
\special{pa 4158 1884}%
\special{pa 4126 1885}%
\special{pa 4094 1887}%
\special{pa 4062 1888}%
\special{pa 4030 1888}%
\special{pa 3998 1889}%
\special{pa 3966 1889}%
\special{pa 3934 1888}%
\special{pa 3902 1888}%
\special{pa 3870 1887}%
\special{pa 3838 1886}%
\special{pa 3806 1885}%
\special{pa 3774 1883}%
\special{pa 3742 1882}%
\special{pa 3710 1880}%
\special{pa 3678 1877}%
\special{pa 3646 1875}%
\special{pa 3614 1872}%
\special{pa 3582 1869}%
\special{pa 3550 1866}%
\special{pa 3519 1863}%
\special{pa 3487 1860}%
\special{pa 3455 1856}%
\special{pa 3423 1852}%
\special{pa 3391 1848}%
\special{pa 3359 1844}%
\special{pa 3327 1839}%
\special{pa 3295 1835}%
\special{pa 3264 1830}%
\special{pa 3232 1826}%
\special{pa 3200 1821}%
\special{pa 3169 1816}%
\special{pa 3137 1810}%
\special{pa 3105 1805}%
\special{pa 3074 1800}%
\special{pa 3042 1794}%
\special{pa 3011 1789}%
\special{pa 2979 1783}%
\special{pa 2948 1777}%
\special{pa 2916 1772}%
\special{pa 2885 1766}%
\special{pa 2853 1760}%
\special{pa 2822 1754}%
\special{pa 2791 1748}%
\special{pa 2773 1745}%
\special{sp}%
%
\special{pn 13}%
\special{pa 1693 1721}%
\special{pa 1553 1721}%
\special{fp}%
\special{sh 1}%
\special{pa 1553 1721}%
\special{pa 1620 1741}%
\special{pa 1606 1721}%
\special{pa 1620 1701}%
\special{pa 1553 1721}%
\special{fp}%
%
\special{pn 13}%
\special{pa 2345 1968}%
\special{pa 2256 2040}%
\special{fp}%
\special{sh 1}%
\special{pa 2256 2040}%
\special{pa 2320 2014}%
\special{pa 2297 2006}%
\special{pa 2295 1983}%
\special{pa 2256 2040}%
\special{fp}%
%
\special{pn 13}%
\special{pa 3247 2215}%
\special{pa 3353 2305}%
\special{fp}%
\special{sh 1}%
\special{pa 3353 2305}%
\special{pa 3315 2247}%
\special{pa 3312 2270}%
\special{pa 3289 2277}%
\special{pa 3353 2305}%
\special{fp}%
%
\special{pn 13}%
\special{pa 2753 1731}%
\special{pa 2785 1734}%
\special{pa 2817 1737}%
\special{pa 2849 1740}%
\special{pa 2881 1742}%
\special{pa 2914 1745}%
\special{pa 2946 1748}%
\special{pa 2978 1750}%
\special{pa 3010 1752}%
\special{pa 3042 1755}%
\special{pa 3074 1757}%
\special{pa 3106 1758}%
\special{pa 3138 1760}%
\special{pa 3170 1761}%
\special{pa 3202 1762}%
\special{pa 3234 1763}%
\special{pa 3266 1764}%
\special{pa 3297 1764}%
\special{pa 3329 1764}%
\special{pa 3361 1763}%
\special{pa 3393 1762}%
\special{pa 3424 1761}%
\special{pa 3456 1759}%
\special{pa 3488 1757}%
\special{pa 3519 1755}%
\special{pa 3551 1752}%
\special{pa 3582 1749}%
\special{pa 3614 1746}%
\special{pa 3645 1742}%
\special{pa 3676 1738}%
\special{pa 3708 1735}%
\special{pa 3739 1731}%
\special{pa 3771 1726}%
\special{pa 3802 1722}%
\special{pa 3834 1718}%
\special{pa 3865 1714}%
\special{pa 3897 1710}%
\special{pa 3928 1706}%
\special{pa 3960 1702}%
\special{pa 3991 1698}%
\special{pa 4023 1694}%
\special{pa 4055 1691}%
\special{pa 4086 1688}%
\special{pa 4118 1685}%
\special{pa 4150 1682}%
\special{pa 4182 1679}%
\special{pa 4214 1677}%
\special{pa 4246 1676}%
\special{pa 4278 1675}%
\special{pa 4310 1674}%
\special{pa 4342 1674}%
\special{pa 4375 1674}%
\special{pa 4407 1675}%
\special{pa 4440 1676}%
\special{pa 4472 1678}%
\special{pa 4505 1681}%
\special{pa 4538 1684}%
\special{pa 4571 1688}%
\special{pa 4604 1693}%
\special{pa 4638 1698}%
\special{pa 4671 1705}%
\special{pa 4705 1712}%
\special{pa 4738 1720}%
\special{pa 4772 1729}%
\special{pa 4805 1739}%
\special{pa 4838 1749}%
\special{pa 4871 1761}%
\special{pa 4903 1774}%
\special{pa 4935 1787}%
\special{pa 4966 1802}%
\special{pa 4996 1818}%
\special{pa 5025 1834}%
\special{pa 5054 1852}%
\special{pa 5081 1871}%
\special{pa 5107 1891}%
\special{pa 5132 1912}%
\special{pa 5155 1934}%
\special{pa 5177 1958}%
\special{pa 5197 1982}%
\special{pa 5216 2008}%
\special{pa 5232 2035}%
\special{pa 5247 2063}%
\special{pa 5260 2093}%
\special{pa 5271 2124}%
\special{pa 5279 2156}%
\special{pa 5286 2189}%
\special{pa 5290 2222}%
\special{pa 5292 2256}%
\special{pa 5292 2291}%
\special{pa 5290 2325}%
\special{pa 5286 2359}%
\special{pa 5279 2392}%
\special{pa 5271 2425}%
\special{pa 5260 2457}%
\special{pa 5247 2488}%
\special{pa 5233 2518}%
\special{pa 5216 2546}%
\special{pa 5197 2572}%
\special{pa 5176 2597}%
\special{pa 5157 2615}%
\special{sp}%
%
\special{pn 13}%
\special{pa 5143 2625}%
\special{pa 5111 2630}%
\special{pa 5080 2634}%
\special{pa 5048 2638}%
\special{pa 5016 2640}%
\special{pa 4984 2641}%
\special{pa 4952 2641}%
\special{pa 4920 2641}%
\special{pa 4888 2639}%
\special{pa 4856 2638}%
\special{pa 4824 2636}%
\special{pa 4792 2635}%
\special{pa 4760 2634}%
\special{pa 4728 2634}%
\special{pa 4696 2636}%
\special{pa 4664 2639}%
\special{pa 4633 2643}%
\special{pa 4601 2649}%
\special{pa 4570 2655}%
\special{pa 4539 2662}%
\special{pa 4507 2669}%
\special{pa 4476 2676}%
\special{pa 4445 2683}%
\special{pa 4413 2691}%
\special{pa 4382 2697}%
\special{pa 4350 2704}%
\special{pa 4319 2709}%
\special{pa 4287 2714}%
\special{pa 4256 2718}%
\special{pa 4237 2721}%
\special{sp -0.045}%
%
\special{pn 13}%
\special{pa 4283 2721}%
\special{pa 4249 2721}%
\special{pa 4215 2721}%
\special{pa 4182 2720}%
\special{pa 4149 2719}%
\special{pa 4117 2716}%
\special{pa 4085 2712}%
\special{pa 4054 2706}%
\special{pa 4024 2699}%
\special{pa 3996 2689}%
\special{pa 3968 2676}%
\special{pa 3942 2661}%
\special{pa 3917 2644}%
\special{pa 3894 2625}%
\special{pa 3871 2604}%
\special{pa 3848 2581}%
\special{pa 3827 2557}%
\special{pa 3806 2532}%
\special{pa 3785 2506}%
\special{pa 3765 2479}%
\special{pa 3745 2452}%
\special{pa 3725 2425}%
\special{pa 3705 2397}%
\special{pa 3685 2370}%
\special{pa 3665 2343}%
\special{pa 3644 2317}%
\special{pa 3623 2292}%
\special{pa 3601 2267}%
\special{pa 3579 2243}%
\special{pa 3557 2219}%
\special{pa 3534 2196}%
\special{pa 3511 2174}%
\special{pa 3488 2152}%
\special{pa 3464 2131}%
\special{pa 3439 2111}%
\special{pa 3415 2091}%
\special{pa 3390 2072}%
\special{pa 3364 2053}%
\special{pa 3338 2035}%
\special{pa 3312 2017}%
\special{pa 3286 2000}%
\special{pa 3259 1983}%
\special{pa 3231 1967}%
\special{pa 3204 1952}%
\special{pa 3176 1936}%
\special{pa 3148 1921}%
\special{pa 3119 1907}%
\special{pa 3091 1893}%
\special{pa 3062 1879}%
\special{pa 3033 1866}%
\special{pa 3003 1852}%
\special{pa 2974 1839}%
\special{pa 2944 1826}%
\special{pa 2915 1814}%
\special{pa 2885 1801}%
\special{pa 2855 1789}%
\special{pa 2825 1776}%
\special{pa 2795 1764}%
\special{pa 2773 1755}%
\special{sp}%
%
\special{pn 13}%
\special{pa 4205 1673}%
\special{pa 4055 1693}%
\special{fp}%
\special{sh 1}%
\special{pa 4055 1693}%
\special{pa 4124 1704}%
\special{pa 4108 1686}%
\special{pa 4118 1664}%
\special{pa 4055 1693}%
\special{fp}%
\put(12.8300,-16.7100){\makebox(0,0)[lb]{$\alpha_1$}}%
\put(20.0300,-21.2500){\makebox(0,0)[lb]{$\beta_1$}}%
%
\special{pn 13}%
\special{pa 2585 1978}%
\special{pa 2657 1860}%
\special{fp}%
\special{sh 1}%
\special{pa 2657 1860}%
\special{pa 2605 1906}%
\special{pa 2629 1906}%
\special{pa 2639 1927}%
\special{pa 2657 1860}%
\special{fp}%
%
\special{pn 13}%
\special{pa 5021 2510}%
\special{pa 5075 2370}%
\special{fp}%
\special{sh 1}%
\special{pa 5075 2370}%
\special{pa 5032 2425}%
\special{pa 5056 2420}%
\special{pa 5070 2439}%
\special{pa 5075 2370}%
\special{fp}%
\put(47.4300,-22.7500){\makebox(0,0)[lb]{$\alpha_2$}}%
\put(35.2300,-16.9100){\makebox(0,0)[lb]{$\beta_2$}}%
\put(27.9700,-19.4100){\makebox(0,0)[lb]{$b$}}%
%
\special{pn 13}%
\special{ar 1213 1151 223 223  0.3283294 6.2831853}%
\special{ar 1213 1151 223 223  0.0000000 0.2905992}%
%
\special{pn 20}%
\special{sh 1}%
\special{ar 1197 1145 10 10 0  6.28318530717959E+0000}%
\special{sh 1}%
\special{ar 1197 1145 10 10 0  6.28318530717959E+0000}%
%
\special{pn 13}%
\special{pa 1453 1121}%
\special{pa 2743 1741}%
\special{fp}%
\put(11.1700,-12.9500){\makebox(0,0)[lb]{$t_n$}}%
\put(14.4300,-10.6500){\makebox(0,0)[lb]{$t_n^{\ast}$}}%
\put(17.8700,-14.9500){\makebox(0,0)[lb]{$l_n$}}%
\put(10.3700,-8.4500){\makebox(0,0)[lb]{$\gamma_n$}}%
%
\special{pn 13}%
\special{ar 2635 655 224 224  0.3307889 6.2831853}%
\special{ar 2635 655 224 224  0.0000000 0.2948829}%
%
\special{pn 20}%
\special{sh 1}%
\special{ar 2619 650 10 10 0  6.28318530717959E+0000}%
\special{sh 1}%
\special{ar 2619 650 10 10 0  6.28318530717959E+0000}%
\put(25.3900,-8.0000){\makebox(0,0)[lb]{$t_j$}}%
\put(28.6500,-5.7000){\makebox(0,0)[lb]{$t_j^{\ast}$}}%
\put(21.2100,-4.4000){\makebox(0,0)[lb]{$\gamma_j$}}%
%
\special{pn 13}%
\special{pa 2857 655}%
\special{pa 2854 687}%
\special{pa 2851 719}%
\special{pa 2849 751}%
\special{pa 2846 783}%
\special{pa 2843 814}%
\special{pa 2840 846}%
\special{pa 2837 878}%
\special{pa 2834 910}%
\special{pa 2832 942}%
\special{pa 2829 974}%
\special{pa 2826 1006}%
\special{pa 2823 1038}%
\special{pa 2820 1070}%
\special{pa 2816 1101}%
\special{pa 2813 1133}%
\special{pa 2810 1165}%
\special{pa 2807 1197}%
\special{pa 2804 1229}%
\special{pa 2800 1261}%
\special{pa 2797 1292}%
\special{pa 2793 1324}%
\special{pa 2790 1356}%
\special{pa 2786 1388}%
\special{pa 2782 1420}%
\special{pa 2778 1451}%
\special{pa 2775 1483}%
\special{pa 2771 1515}%
\special{pa 2766 1547}%
\special{pa 2762 1578}%
\special{pa 2758 1610}%
\special{pa 2754 1642}%
\special{pa 2749 1673}%
\special{pa 2744 1705}%
\special{pa 2740 1737}%
\special{pa 2737 1755}%
\special{sp}%
%
\special{pn 13}%
\special{ar 4143 765 224 224  0.3307889 6.2831853}%
\special{ar 4143 765 224 224  0.0000000 0.2948829}%
%
\special{pn 20}%
\special{sh 1}%
\special{ar 4127 760 10 10 0  6.28318530717959E+0000}%
\special{sh 1}%
\special{ar 4127 760 10 10 0  6.28318530717959E+0000}%
\put(40.4700,-9.1000){\makebox(0,0)[lb]{$t_1$}}%
%
\special{pn 13}%
\special{pa 1177 1135}%
\special{pa 1423 1131}%
\special{dt 0.045}%
\special{pa 1423 1131}%
\special{pa 1422 1131}%
\special{dt 0.045}%
%
\special{pn 13}%
\special{pa 2627 651}%
\special{pa 2873 645}%
\special{dt 0.045}%
\special{pa 2873 645}%
\special{pa 2872 645}%
\special{dt 0.045}%
%
\special{pn 13}%
\special{pa 4123 765}%
\special{pa 4367 761}%
\special{dt 0.045}%
\special{pa 4367 761}%
\special{pa 4366 761}%
\special{dt 0.045}%
%
\special{pn 13}%
\special{pa 2767 1715}%
\special{pa 2799 1709}%
\special{pa 2830 1704}%
\special{pa 2862 1698}%
\special{pa 2893 1692}%
\special{pa 2925 1687}%
\special{pa 2956 1681}%
\special{pa 2988 1675}%
\special{pa 3019 1670}%
\special{pa 3051 1664}%
\special{pa 3082 1658}%
\special{pa 3114 1653}%
\special{pa 3145 1647}%
\special{pa 3177 1641}%
\special{pa 3208 1635}%
\special{pa 3240 1629}%
\special{pa 3271 1623}%
\special{pa 3303 1618}%
\special{pa 3334 1612}%
\special{pa 3365 1606}%
\special{pa 3397 1600}%
\special{pa 3428 1594}%
\special{pa 3460 1587}%
\special{pa 3491 1581}%
\special{pa 3522 1575}%
\special{pa 3553 1569}%
\special{pa 3585 1563}%
\special{pa 3616 1556}%
\special{pa 3647 1550}%
\special{pa 3678 1543}%
\special{pa 3710 1536}%
\special{pa 3741 1529}%
\special{pa 3772 1522}%
\special{pa 3803 1515}%
\special{pa 3834 1507}%
\special{pa 3865 1499}%
\special{pa 3896 1491}%
\special{pa 3927 1482}%
\special{pa 3957 1473}%
\special{pa 3988 1464}%
\special{pa 4019 1454}%
\special{pa 4049 1444}%
\special{pa 4080 1433}%
\special{pa 4111 1422}%
\special{pa 4141 1410}%
\special{pa 4171 1398}%
\special{pa 4201 1386}%
\special{pa 4231 1372}%
\special{pa 4261 1358}%
\special{pa 4290 1343}%
\special{pa 4318 1328}%
\special{pa 4347 1311}%
\special{pa 4374 1294}%
\special{pa 4401 1276}%
\special{pa 4427 1257}%
\special{pa 4452 1237}%
\special{pa 4477 1215}%
\special{pa 4501 1193}%
\special{pa 4523 1169}%
\special{pa 4545 1145}%
\special{pa 4565 1119}%
\special{pa 4584 1091}%
\special{pa 4602 1063}%
\special{pa 4619 1032}%
\special{pa 4634 1001}%
\special{pa 4648 968}%
\special{pa 4660 933}%
\special{pa 4671 897}%
\special{pa 4680 860}%
\special{pa 4685 824}%
\special{pa 4687 791}%
\special{pa 4684 761}%
\special{pa 4676 737}%
\special{pa 4663 719}%
\special{pa 4643 708}%
\special{pa 4618 702}%
\special{pa 4589 702}%
\special{pa 4556 705}%
\special{pa 4520 712}%
\special{pa 4482 722}%
\special{pa 4442 734}%
\special{pa 4401 747}%
\special{pa 4377 755}%
\special{sp}%
\put(28.6300,-13.1100){\makebox(0,0)[lb]{$l_j$}}%
\put(39.6700,-4.4100){\makebox(0,0)[lb]{$\gamma_1$}}%
\put(43.7300,-6.8500){\makebox(0,0)[lb]{$t_1^{\ast}$}}%
\put(44.6700,-13.9100){\makebox(0,0)[lb]{$l_1$}}%
%
\special{pn 8}%
\special{pa 1937 1355}%
\special{pa 1797 1291}%
\special{fp}%
\special{sh 1}%
\special{pa 1797 1291}%
\special{pa 1849 1337}%
\special{pa 1846 1313}%
\special{pa 1866 1301}%
\special{pa 1797 1291}%
\special{fp}%
%
\special{pn 13}%
\special{pa 2801 1270}%
\special{pa 2831 1090}%
\special{fp}%
\special{sh 1}%
\special{pa 2831 1090}%
\special{pa 2800 1152}%
\special{pa 2822 1143}%
\special{pa 2840 1159}%
\special{pa 2831 1090}%
\special{fp}%
%
\special{pn 13}%
\special{pa 2781 1770}%
\special{pa 2071 1090}%
\special{fp}%
%
\special{pn 13}%
\special{pa 2731 1720}%
\special{pa 3271 1330}%
\special{fp}%
%
\special{pn 13}%
\special{pa 2741 1750}%
\special{pa 2561 1180}%
\special{fp}%
\end{picture}%
\end{center}
\vspace{0.5cm}
\caption{Paths and Loops}
\label{figure:2}
\end{figure}
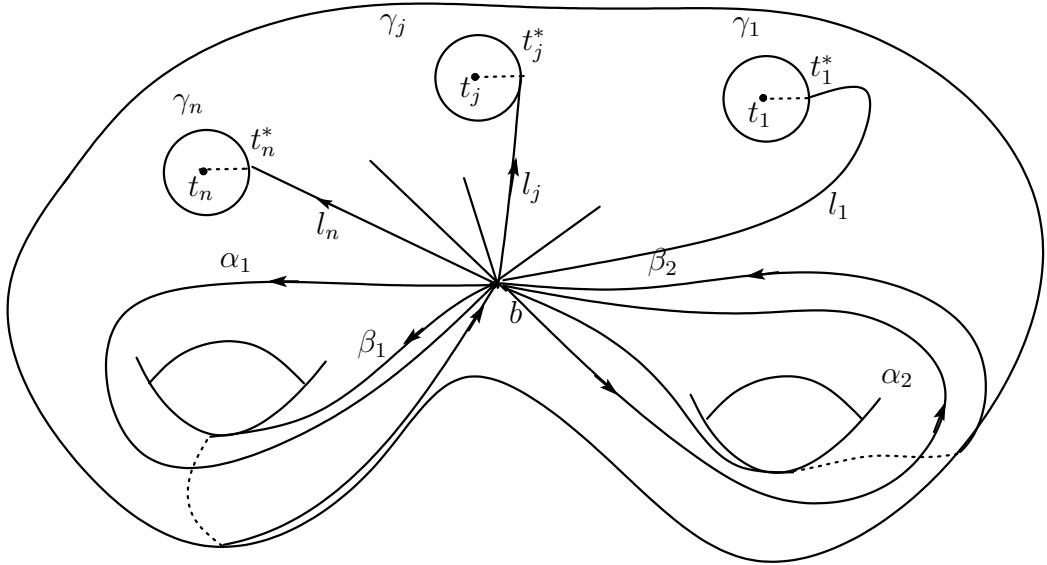

Fixing these data, we can associate the following 
{\em generalized monodromy data} to each $\bnu$-parabolic connection 
$(E, \nabla, \{l^{(i)}_j \} ) \in M^{\balpha}_{D/C}(r, d, (m_i))_{\bnu}$.

\vspace{0.5cm}
\noindent
{\bf Generalized monodromy data.}

\begin{itemize}

\item{\bf Formal monodromy $\{ \widehat{\gamma}_{i} \}$:}
For each $i, 1\leq i \leq n$, we can define 
the formal monodromy $ \widehat{\gamma}_{i} \in \Aut(V_{t_i}) $ 
coming from a monodromy on formal solutions.  The eigenvalues 
of $\widehat{\gamma}_{i}$ are determined by $\{a^{(i)}_{j, -1} \}_{0 \leq j \leq r-1}$ 
with some exponential maps. Since we fix the decomposition (\ref{eq:formal-decomp}), 
$\widehat{\gamma}_{i}$ are  fixed diagonal matrices. 

\item{\bf Stokes data $\{ St_{d^{(i)}_k} \}$:}
Let us consider a sector 
$S \subset \Delta_i \setminus \{ 0 \}$ and let $V(S)$ denote the  
space of analytic (or convergent) solutions of $\nabla = 0$ on the sector $S$.  
Let $\{  S^{(i)}_{k} \}_{1 \leq i \leq s_i}  $  be the set of sectors defined in (\ref{eq:sector}). 
For directions $d_{k, -} \in S^{(i)}_{k}$, $d_{k, +} \in S^{(i)}_{k+1}$,  
we have 
multi-summation maps 
\begin{eqnarray*}
{\rm mults}_{d_{k, -}} &: & V_{t_i} \lra V(S^{(i)}_{k}) \\
{\rm mults}_{d_{k, +}} &: & V_{t_i} \lra V(S^{(i)}_{k+1}) 
\end{eqnarray*}
which are  $\C$-linear isomorphisms between $V_{t_i}$ and $V(S^{(i)}_{k}) $
and $V(S^{(i)}_{k+1})$ respectively.  
The Stokes map $St_{d^{(i)}_k}$  comes from an isomorphism 
\begin{equation}\label{eq:sotkes}
St_{d^{(i)}_k} : V_{t_i} \lra V_{t_i} 
\end{equation}
which makes the following diagram commutative;
$$
\begin{array}{cccc}
{\rm mults}_{d_k, -} :& V_{t_i} &  \lra & V(S^{(i)}_{k}) \\
   & \quad \quad \downarrow St_{d^{(i)}_k} &   &  ||  \\
{\rm mults}_{d_k, +}: &V_{t_i} &  \lra &  V(S^{(i)}_{k+1} )
\end{array}
$$
The identification  $V(S^{(i)}_{k}) = 
V(S^{(i)}_{k+1} )$ comes from analytic continuations. 

According to the decomposition (\ref{eq:formal-decomp}) of 
$V_{t_i} = \oplus_{j=0}^{r-1} V^{(i)}_j$,  
each Stokes map $St_{d^{(i)}_k}$ has the form 
$$
St_{d^{(i)}_k}=Id +  \sum_{(j_1, j_2) 
\in \cJ(d_k^{(i)}, i)}  R_{j_1, j_2},
$$ 
with $R_{j_1, j_2} = i_{j_1} \circ M_{j_1, j_2} \circ p_{j_2} $
where $ 0 \leq j_1, j_2 \leq r-1, j_1 \not=j_2$ and 
 $p_{j_1}  :V_{t_i}  \ra V^{(i)}_{j_1}$ is the projection
and  $i_{j_2}:V^{(i)}_{j_2} \ra V_{t_i}$ is 
the canonical injection. Moreover 
$M_{j_1,j_2} : V^{(i)}_{j_1} \ra V^{(i)}_{j_2}$ is a 
linear map between one dimensional spaces.  
So $M_{j_1, j_2}$  
is given by a scalar $c_{j_2, j_1} \in \C$.  
In the matrix form,  one can write as 
$St_{d^{(i)}_k} =I_r +  \sum_{(j_1, j_2) \in \cJ(d^{(i)}_k, i)}  c_{j_2, j_1} I_{j_2, j_1}$ where $I_{j_2, j_1} $ is the $r \times r$ matrix whose $(i, k)$-entry is zero except for 
$(i, k)= (j_2, j_1)$ and the $(j_2, j_1)$-entry is $1$.  
(For this fact, see [Theorem 8.13, \cite{PS}] or [Lemma 6.5, \cite{sabbah-book}].)

\item{\bf The link $L_i \in \Hom_{\C}(V_b, V_{t_i})$:}
 Analytic continuation along $l_i$ gives a $\C$-linear 
isomorphism $V_b \lra V_{t_i^{\ast}}$. Composition of this 
isomorphism and  the inverse of 
multi-summation map $V_{t_i^{\ast}} \lra V_{t_i} $ gives the linear map  
which is called {\em a link (or a connection matrix)} 
\begin{equation}\label{eq:link}
L_i:V_{b} \stackrel{\simeq}{\lra}V_{t_i^{\ast}} \stackrel{\simeq}{\lra} V_{t_i}
\end{equation}

\item{\bf The topological monodromy $Top_i \in \Aut (V_{t_i})$:}

Identifying $V_{t^{\ast}_i}$ with $V_{t_i}$ by the multi-summation map, an   
analytic continuation along the loop $\gamma_i$ starting from $t^{\ast}_i$
gives a topological monodromy $Top_i \in \Aut (V_{t_i}) \simeq GL_r(\C)$.  
We have the following relation.  
\begin{equation}\label{eq:top}
Top_i = \widehat{\gamma}_i \circ St_{d^{(i)}_{s_i}} \circ  \cdots St_{d^{(i)}_2} \circ St_{d^{(i)}_1}. 
\end{equation}
 
\item{\bf The global monodromy representation:}

We can consider the monodromy representation $\rho: 
\pi_1(C \setminus \{t_1, \cdots, t_n \}, b ) \lra \Aut(V_b) \simeq GL_r(\C) $.
 Moreover $\rho (\gamma^{l}_i) = L^{-1}_i Top_i L_i$ and 
 we set $A_k = \rho(\alpha_k), B_k = \rho(\beta_k)$. 
 These data determine 
the monodromy representation 
$\rho: \pi_1(C \setminus \{t_1, \cdots, t_n \}, b ) \lra 
\Aut(V_b) \simeq GL_{r}(\C) $ 
associated to analytic continuations of the space of solutions of 
$\nabla \sigma = 0$. We have the relation 
\begin{equation}\label{eq:top-relation}
\prod_{i=n}^{1} L_i^{-1} Top_i L_i\prod_{k=g}^1 (B_k^{-1}A_k^{-1}B_kA_k)= I_r.
\end{equation}
(Note that in this notation, $\rho$ becomes an anti-homomorphism such that 
$\rho(\delta_1 \delta_2) = \rho(\delta_2) \rho(\delta_1)$.)
By the relation (\ref{eq:top}), we see that the formal monodromy $\widehat{\gamma_i}$,  
Stokes data $\{ St_{d^{(i)}_k} \}_{1 \leq k  \leq r(r-1)(m_i-1)}$ and the link $L_i$ determine 
$\rho (\gamma^{l}_i) $.  
\end{itemize}

For a generic  $\bnu \in N_r^{(n)}(d, D)$, we define the set $\tilde{{\mathcal R}}(\bnu)$ 
of all 
tuples 
$$
\{ \{\widehat{\gamma_i} \},  \{ St_{d_k^{(i)}} \}, \{L_i \},  \{ A_k,  B_k \}  \}
$$ 
satisfying: 
\begin{enumerate}
\item  For each  $1 \leq i \leq n $,  $\widehat{\gamma_i} \in GL(V_{t_i})$ preserving the 
decomposition (\ref{eq:formal-decomp}) 
whose eigenvalues are determined by $\{a^{(i)}_{j, -1} \}_{0 \leq j \leq r-1}$. 
(Hence, $\widehat{\gamma}_i$ is a diagonal matrix with prescribed eigenvalues.)

\item For each $1 \leq i \leq n$ and $1 \leq k \leq s_i$, $St_{d_k^{(i)}} \in GL(V_{t_i})$ 
of the form $St_{d^{(i)}_k} =Id +  \sum_{(j_1, j_2) \in \cJ(d^{(i)}_k, i)}  R_{j_1, j_2}$
where $R_{j_1, j_2}$ corresponds to a one dimensional homomorphism 
$c_{j_2, j_1}:V^{(i)}_{j_1} \lra V^{(i)}_{j_2}$.  
\item Linear bijections $L_i:V_b \lra V_{t_i}$ for  $1 \leq i \leq n$. 
\item Define $Top_{i} \in GL(V_{t_i}) $ by the formula (\ref{eq:top}). 
The set $\{\{ Top_{i} \}_{1\leq i \leq n}, \{A_k, B_k \in GL(V_{b})\}_{1\leq k 
\leq g} \}$ 
satisfying the relation (\ref{eq:top-relation}). 
\end{enumerate}

\begin{Definition}{\rm \label{def:equiv}
Two tuples 
$\{ \{\widehat{\gamma_i} \}, \{ St_{d_k^{(i)}} \}, \{L_i \},  \{ A_k,  B_k \}  \} $ 
and $\{ \{\widehat{\gamma_i}' \},  \{ St_{d_k^{(i)}}' \}, \{L'_i \},  \{ A'_k,  B'_k \}  \}$ are called equivalent, if there exist $\sigma^{(i)} \in GL(V_{t_i})$ preserving the decomposition 
$V_{t_i} = \oplus_{j=0}^{r-1} V^{(i)}_j$ in (\ref{eq:formal-decomp}) and $\sigma \in GL(V_b)$ satisfying 
\begin{eqnarray*}
\sigma^{(i)} L_i & = &  L'_i \sigma \quad \mbox{for each $ i, 1\leq i \leq n $ } \\
\sigma^{(i)} \widehat{\gamma_i} & = &  \widehat{\gamma_i}' \sigma^{(i)}, \quad    
\mbox{for each $ i, 1\leq i \leq n $ } \\
\sigma^{(i)} St_{d_k^{(i)}} & = &  St_{d_k^{(i)}}'  \sigma^{(i)}, \quad    
\mbox{for each $ i, 1\leq i \leq n, 1 \leq k \leq s_i $ } \\
A_k & = &  \sigma^{-1} A'_k \sigma, \quad    
B_k = \sigma^{-1}B'_k \sigma,  \quad  
\mbox{for each $ k, 1 \leq k \leq g $ }.  
\end{eqnarray*}
}
\end{Definition}

Note that under the assumption that $\bnu$ is generic we see that 
$\sigma^{(i)} \in GL(V_{t_i})$ above is a diagonal matrix 
in $\prod_{j=0}^{r-1} GL(V^{(i)}_j) \simeq (\C^{\times})^{r}$. 

Since the set $\tilde{\mathcal R}(\bnu)$ is an affine scheme 
with a natural action of the reductive group 
\begin{equation}\label{eq:group}
G := GL(V_b) \times \prod_{i=1}^{n} 
\prod_{j=0}^{r-1} GL(V^{(i)}_j) 
\end{equation}
in Definition \ref{def:equiv}, we 
can construct the categorical quotient
\begin{equation}\label{eq:quotient}
{\mathcal R}(\bnu) = \tilde{\mathcal R}(\bnu)//G
\end{equation}
which is considered as the set of equivalence classes of 
the generalized monodromy data associated to $\bnu$.  
By definition of the categorical quotient, 
$ {\mathcal R}(\bnu)$ is an affine scheme.

\begin{Proposition}\label{prop:smooth}
Assume that $\bnu \in N^{(n)}_r(d,D)$ is generic, non-resonant and irreducible 
(cf. Definition \ref{def:generic}). Then 
moduli space ${\mathcal R}(\bnu)$ is a nonsingular affine scheme and 
$$
\dim {\mathcal R}(\bnu)= 2r^2(g-1) + \sum_{i=1}^n m_i r(r-1) +2
$$
if ${\mathcal R}(\bnu)$ is non-empty.  
\end{Proposition}

\begin{proof} 
For a generic  $\bnu \in N^{(n)}_r(d,D)$, consider the affine variety 
of tuples
$$
{\mathcal S}(\bnu) =\{ \{  \{\widehat{\gamma_i} \}, \{ St_{d_k^{(i)}} \}, \{L_i \},  \{ A_k,  B_k \}  \} | \mbox{ without the relation
(\ref{eq:top-relation})} \}.
$$
Set $l_i = (m_i-1)r(r-1)$ and recall the equality
$\sum_{1 \leq k \leq s_i} \sharp \cJ(d^{(i)}_k, i) = l_i$ where $\sharp \cJ(d^{(i)}_k, i)$ 
is the multiplicity of the singular direction $d^{(i)}_k$.  
The set of Stokes matrices $St_{d_k^{(i)}}$ is isomorphic to the affine variety  
$\C^{\sharp \cJ(d^{(i)}_k, i)}$. 
Then we see that ${\mathcal S}(\bnu) \simeq \prod_{i=1}^n \C^{l_i} \times 
GL_{r}(\C)^{n} \times GL_r(\C)^{2g}$ (with 
$l_i = (m_i-1)r(r-1)$), hence ${\mathcal S}(\bnu)$ is 
a smooth affine variety of dimension $ \sum_{i=1}^n (m_i-1) r (r-1) + 
(n + 2g) r^2$.  
Define the morphism 
\begin{equation}
\mu:\cS(\bnu) \lra SL_r(\C) 
\end{equation} 
by
\begin{equation}
\mu(\{  \{\widehat{\gamma_i} \}, \{ St_{d_k^{(i)}} \}, \{L_i \},  \{ A_k,  B_k \} \})  = 
\prod_{i=n}^{1} L_i^{-1} Top_i L_i \prod_{k=g}^1 (B_k^{-1}A_k^{-1}B_kA_k)
\end{equation}
with $Top_i = \widehat{\gamma}_i \circ St_{d^{(i)}_{s_i}} \circ  \cdots St_{d^{(i)}_2} \circ St_{d^{(i)}_1}$. 
Then we see that $\tilde{\cR}(\bnu) = \mu^{-1}(I_r)$.  
As in [Theorem 2.2.5, \cite{HR}], in order to prove 
the smoothness of  $\tilde{\cR}(\bnu)$, we only have to prove 
that the derivative  $d\mu_s: T_{\cS, s} \lra T_{SL_{r}(\C), I_r} \simeq sl_{r}
(\C)$ is surjective at any point $s \in \cS$.  
If $\bnu$ is non-resonant and irreducible, this can be shown by  
direct calculations of $d\mu_s$  as in the proof of [Theorem 2.2.5, \cite{HR}]. 
Therefore $\tilde{\cR}(\bnu)$ is a smooth affine scheme with 
$$
\dim \tilde{\cR}(\bnu) = \dim \cS(\bnu) - (r^2-1) = \sum_{i=1}^n (m_i-1) r (r-1) + (n + 2g) r^2 -(r^2-1). 
$$
Recall that $G = GL(V_b) \times \prod_{i=1}^{n} 
\prod_{j=0}^{r-1} GL(V^{(i)}_j)  \simeq GL_{r}(\C) \times \prod_{i=1}^n (\C^{\times})^r$ acts on $\tilde{\cR}(\bnu)$ as in Definition \ref{def:equiv}.
Note that the subgroup $Z = \{ (c I_r, (c, \ldots, c)) \in G, c 
\in \C^{\times} \} $ acts on  $\tilde{\cR}(\bnu)$ trivially.   
Then under the assumption on $\bnu$, 
it is also easy to see that the action of 
$G/Z$ on $\tilde{\cR}(\bnu)$ is free. 
 Hence $\cR(\bnu) = \tilde{\cR}(\bnu)//G$ is a smooth affine scheme with 
\begin{eqnarray*}
\dim {\mathcal R}(\bnu) &=& 
\dim  \tilde{\mathcal R}(\bnu) - (\dim G -1) \\
  & = &  \sum_{i=1}^n (m_i-1) r (r-1) + (n + 2g) r^2 -(r^2-1)
- (r^2 + nr-1) \\
& =&  2 r^2(g-1) + \sum_{i=1}^n m_i r (r-1)+ 2.  
\end{eqnarray*}

\end{proof}

\subsection{The generalized Riemann-Hilbert correspondence}

Let us fix a data $(C, D= \sum_{i=1}^n m_i t_i)$ and  $z_i$ a generator of ${\frak m}_{t_i}$, 
and take a generic element $\bnu \in N_{r}^{(n)}(d, D)$.  For these data, we 
can also fix an analytic neighborhood $\Delta_{i} = \{ z_i \in \C |  |z_i| < \epsilon_i  \} $ of each $t_i$, singular directions $\{d^{(i)}_k \}$, sectors $\{S^{(i)}_k \}$  and $ t_i^{\ast} 
\in \partial \Delta_{i} $  as in the previous subsection. 

Moreover we fix a base point $b \in C \setminus \{ t_1, \cdots, t_n \}$ and 
a continuous path $l_i$ from $b$ to $t_{i}^{\ast}$ and loops $\{ \gamma^{l}_{i}, \alpha_k,  \beta_k \} $  with the condition (\ref{eq:fund}).   

Fixing these data, we can define the generalized Riemann-Hilbert correspondence as in the 
previous subsection. 
\begin{equation}\label{eq:RH-nu}
{\bf RH}_{(D/C),\bnu}: M^{\balpha}_{D/C}(r, d, (m_i))_{\bnu} \lra 
{\mathcal R}(\bnu).  
\end{equation}

\begin{Theorem}\label{thm:Riemann-Hilbert}
 Under the notation above,  assume further that $\bnu$ is  non-resonant and irreducible.  
Then the generalized Riemann-Hilbert correspondence ${\bf RH}_{(D/C), \bnu}$ $(\ref{eq:RH-nu})$ is an analytic isomorphism.  
\end{Theorem}

\begin{proof} 
Under the assumption that $\bnu$ is  generic, we can fix formal types of all
singularities of $\bnu$-parabolic connections $(E, \nabla, \{ l^{(i)}_j \}) 
\in  M^{\balpha}_{D/C}(r, d, (m_i))_{\bnu}$, and then
we can define the Riemann-Hilbert correspondence  ${\bf RH}_{(D/C), \bnu}$ as we explained above.    
The fact that   ${\bf RH}_{(D/C), \bnu}$ is a holomorphic map 
can be proved as follows.   All  generalized monodromy data can be defined 
by a system of local fundamental solutions of $\nabla=0$ defined in each 
open sets including the sectors near the singular points $t_i$ (cf. [Ch. VI, \cite{Sibuya}]).   
If one has a holomorphic family of $\bnu$-parabolic connections, Sibuya \cite{Sibuya-1} showed that at least locally in the parameter space
there exists  a  family of a system of fundamental solutions 
depending on the parameter holomorphically.  Hence, this shows that ${\bf RH}_{(D/C), \bnu}$ is holomorphic.  

Recall that  $M^{\balpha}_{D/C}(r, d, (m_i))_{\bnu}$ 
 is  a smooth quasi-projective scheme (Theorem \ref{smoothness-thm}) and  
 ${\mathcal R}(\bnu)$ is  a smooth affine algebraic scheme  (Proposition \ref{prop:smooth}). Since ${\bf RH}_{(D/C), \bnu}$ is an analytic morphism between 
 smooth analytic manifolds, we only have to prove that 
${\bf RH}_{(D/C), \bnu}$ is bijective. 

First, we prove the surjectivity of ${\bf RH}_{(D/C), \bnu}$. 
Let us take a tuple 
$$\{ \{\widehat{\gamma_i} \}, \{ St_{d_k^{(i)}} \}, \{L_i \},  \{ A_k,  B_k \}  \} 
\in \tilde{\mathcal R}(\bnu).
$$  
By Malgrange-Sibuya theorem formulated as in [Theorem 4.5.1, \cite{BV}] or 
original form [Theorem 6.11.1 \cite{Sibuya}]), we see that 
the local analytic isomorphism class of the singular connection 
on each small neighborhood $\Delta_{i}$ 
with the fixed formal type $\bnu^{(i)}= \{\nu^{(i)}_j(z_i) \}$   
has one to one correspondence with the set of the formal monodromy and 
Stokes data with the formal type determined by $\bnu^{(i)}$.  
So for each $i, 1 \leq i \leq n$, 
we can take local analytic connections $(E^{(i)}, \nabla^{(i)})$ on $\Delta_i$ 
whose formal types are given by $\bnu^{(i)}$ and whose local generalized monodromy data 
is isomorphic to $ \{ \{ \widehat{\gamma_i} \}, \{ St_{d_k^{(i)}} \}\}$. 

 Since we assume that $\bnu$ is generic and non-resonant, 
the parabolic structures $\{ l^{(i)}_j \} $  of $(E^{(i)}, \nabla^{(i)}) $ at $t_i$ 
can be uniquely determined, so we obtain local analytic $\bnu^{(i)}$-parabolic 
connections $(E^{(i)}, \nabla^{(i)}, \{ l^{(i)}_j \})$.  

The data $ \{Top_i, \{L_i\}, A_k, B_k \} $ determine the monodromy data of a 
flat bundle $ {\bf E}_1 $ on $C_0:=C \setminus \{t_1, \cdots, t_n \}$.  
Hence $E_1 = {\bf E}_1 \otimes \cO_{C_0}$ 
is a locally free sheaf with a flat connection $ \nabla:E_1 \lra E_1 \otimes 
\Omega^1_{C_0}$.  Since by (\ref{eq:top-relation}), the 
local monodromy data of $ (E_1, \nabla) $ and $(E^{(i)}, \nabla^{(i)}) $ 
is isomorphic over $ \Delta_i \setminus \{ 0 \}$, 
we can glue $ (E_1, \nabla_1) $ and $ (E^{(i)}, \nabla^{(i)}) $ to 
obtain a holomorphic vector bundle $E$ on $C$ and a flat connection 
$ \nabla:E \lra E \otimes \Omega^1_C(D) $. Then by GAGA,
we obtain a $\bnu$-parabolic connection  $ (E, \nabla, \{ l^{(i)}_j \} ) $ of degree $d$. 
(Note that by Fuchs relation, the residue part of  
$\bnu$  determines the degree of $E$). 
Since $\bnu$ is irreducible, this 
connection must be irreducible, hence it is $\balpha$-stable for any weight $\balpha$, so 
it is a member of  $M^{\balpha}_{D/C}(r, d, (m_i))_{\bnu}$. 
This shows that ${\bf RH}_{(D/C), \bnu}$ is surjective.
Now from this construction, the injectivity of ${\bf RH}_{(D/C), \bnu}$ is obvious.     Hence  ${\bf RH}_{(D/C), \bnu}$ is bijective.  
\end{proof}
 
\begin{Remark} 
\label{rem:family}
{\rm  In the next section, we will vary the data $(C, \bt) $, the local generators $\{ z_i \in {\frak m}_{t_i} \} $ in a suitable moduli space and $\bnu = \{ \bnu^{(i)}(z_i) \} \in N^{(n)}_r(d, D)$ and we will  construct   the continuous analytic family  of Riemann-Hilbert correspondences  ${\bf RH}_{(D/C), \bnu}$.  
In order to do this, we first fix a data  $(C, \bt), \{ z_1, \cdots, z_n \},  \bnu $ as a  base point in a connected  component of the moduli space of such data.     
(We will assume that $\bnu$ is generic and simple (see Definition \ref{def:multi1}) for a technical reason).
We can also fix a small neighborhood $\Delta_i$ near $t_i$ and the (simple) singular directions $\{ d_j^{(i)} \}_{1 \leq  j \leq s_i}   $ and the ordered sectors 
$\{ S_k^{(i)} \}_{1 \leq k \leq s_i}  $.   Moreover,  we  can  fix $t_i^{\ast}$ as before (see Figure \ref{figure:1}).  
Fixing a base point $b \in C \setminus \{ t_1, \cdots, t_n \}$, we can also take and fix paths and  loops  as in the previous subsection.  Once we fix these data,  we can define the moduli space $\cR(\bnu)$ of generalized monodromy data (\ref{eq:quotient}) and  the Riemann-Hilbert correspondence ${\bf RH}_{(D/C), \bnu}$ as in (\ref{eq:RH-nu}). 
Note that  in order to define a data of  $\cR(\bnu)$,  we need fix the paths $\{ l_i \}, 
\{ \gamma^{l}, \alpha_k, \beta_k \}$ and  the order of the sectors $\{ S^{(i)}_k \}_{1 \leq k \leq s_i}  $  near each $t_i$ which is determined by the singular directions 
determined by $\bnu^{(i)}$ as in \ref{ss:mon}.    The closure of the first sector $S^{(i)}_1$ contains the end point $t_i^{\ast}$  of the path $l_i$.  If we vary  $\bnu$ continuously from the original data in the connected component  under the condition that $\bnu$ is generic and simple and   
fixing the data $(C, \bt), \{ z_i \} $,  the singular directions and sectors are changing continuously.   In this procedure, we need to keep the order of sectors, hence we need to change  the point $t^{\ast}_i$  and the path $l_i$ continuously. 
It is easy to see that when we vary the data $(C, \bt)$, $\{ z_i \} $,  $\bnu$ continuously  in the connected component of the  moduli spaces (see \ref{subsec:isom} )  starting from  the base data,  we can vary continuously singular directions, sectors  and paths and loops  starting from the original data .   By this procedure, we can define the continuous analytic family of Riemann-Hilbert correspondences in each connected component of the moduli space. }
\end{Remark}

\begin{Remark}{\rm  By using the result in \cite{Sibuya-1}, Jimbo, Miwa and Ueno \cite{JMU1} discussed about the analycity of 
the Riemann-Hilbert correspondence when one varies $\{(\BP^1, D), \bnu \} 
\in M_{0, n} \times N_{r}^{(n)}(d, D)$ and discuss about the isomonodromic deformations of linear connections. 
When $\bnu$ varies in the open set of $N_{r}^{(n)}(d, D)$
corresponding to generic exponents, one can define an analytic family of  Riemann-Hilbert correspondences.}   
\end{Remark}

\begin{Remark}\label{rem:jumping}
 {\rm The surjectivity part of Theorem \ref{thm:Riemann-Hilbert} 
 is related to the generalized Riemann-Hilbert problem with 
 irregular singularities over $C = \BP^1$ 
 which has been investigated,  for example,  in 
 \cite{JMU1}, \cite{BMM}, \cite{vdP-S}.  
Usually, they would like to obtain  singular connections  
$(E, \nabla)$ with trivial bundle $E=  \cO_{\BP^1}^{\oplus r}$. 
However from our view point of global moduli spaces of the connections, 
even in the case of $C = \BP^1$ and $d = \deg E =0$, 
it is not natural to assume that vector bundle $E$ is always trivial, that is, $E 
=  \cO_{\BP^1}^{\oplus r}$, for the set of such connections may correspond to  
a Zariski dense open subset of the moduli space  
$M^{\balpha}_{D/C}(r, d, (m_i))_{\bnu}$ but they may not 
cover all of the moduli space.  
The type of bundle $E$ may jump, for example, as 
$E \simeq \cO_{\BP^1}(1) \oplus \cO_{\BP^1}(-1) 
\oplus \cO_{\BP^1}^{\oplus (r-2)}$.  
The jumping phenomena of the bundle types in the moduli space of 
semistable bundles, which both of authors learned from professor Maruyama, 
is one of keys of many moduli problems and 
make the moduli theory interesting. 
In the case of the connections, divisors for jumping 
phenomena are corresponding to the $\tau$-divisors. }
\end{Remark}

\begin{Remark}{\rm 
In \cite{B}, Boalch constructed the space of isomorphism classes 
of meromorphic connections on a degree zero bundles on $\BP^1$ 
with compatible framing of fixed generic irregular type by an analytic 
method and  showed that taking monodromy data induces the 
bijection between the space of meromorphic connections on degree zero 
bundles and the corresponding spaces of monodromy data (cf. [Corollary 
4.9, \cite{B}]).  In  \cite{BB},  Biquard and Boalch generalized 
the analytic construction of the moduli spaces of the connections and 
showed that under a slight weaker generic condition 
the space of  meromorphic connections 
with fixed equivalence classes of polar part over a curve  is 
a hyper-K\"{a}hler manifold.  Despite these interesting analytic constructions, 
we believe that our algebro-geometric constructions of the moduli 
space of stable parabolic connections with fixed irregular singular types 
have some advantages such as natural algebraic structures on 
the moduli spaces which are crucial to write down the isomonodromic 
differential equations in some rational algebraic equations on the 
algebraic coordinates on the phase spaces.   
}
\end{Remark}

\begin{Remark}{\rm 
In \cite{Inaba-1},  Inaba showed more stronger statement for Riemann-Hilbert correspondence when all of the singularities are at most regular (that is, $m_i=1$ for all $i$). See also \cite{IIS-1}, \cite{IIS-2} and \cite{IISA} for former results on the  Riemann-Hilbert correspondences.  }
\end{Remark}

\section{Geometric Painlev\'e property for generalized isomonodromy differential systems}

\subsection{Generalized isomonodromic differential systems and their geometric Painlev\'e property} \label{subsec:isom}
\par\noindent
Let us fix integers $g, n, d, r, (m_i)_{1 \leq i \leq n}$ as in 
the previous section and let $M_{g,n}$ be an algebraic scheme which is a smooth covering
of the moduli stack of $n$-(distinct) pointed curves such that  
$M_{g,n}$ is smooth and has the universal family $({\mathcal C}, \tilde{t}_1, \cdots, \tilde{t}_n)  \lra M_{g, n}$.
We put $D = \sum_{i=1}^{n} m_i \tilde{t}_i$.   
For each $(C, t_1, \cdots, t_n) \in M_{g,n}$ and $i, 1 \leq i \leq n$, let 
$
\Psi_i:\cO_{C, t_i}/{\mathfrak m}_
{t_i}^{m_i} 
\stackrel{\simeq}{\lra} \C[z_i]/(z_i^{m_i})
$
be  ring isomorphisms.  The 
moduli space $M_{g, n, (m_i)}$ of tuples 
$
(C, t_1, \cdots, t_n, \{\Psi_i \}_{1 \leq i \leq n})
$ 
is a smooth quasi-projective scheme over $M_{g,n}$.  
Let 
$
M_{g,n, (m_i)} \lra M_{g, n}
$
be the natural morphism and consider the scheme
$
\cN^{(n)}_r(d, D)
$
over $M_{g,n}$  of generalized exponents defined in (\ref{eq:scheme-exp}) in  \S 2.  
Then by using the local coordinates $z_i$ at $\tilde{t_i}$, 
we have a natural isomorphism 
$$
M_{g, n, (m_i)} \times_{M_{g, n}}  \cN^{(n)}_r(d, D) 
\simeq M_{g, n, (m_i)} \times  N^{(n)}_r(d, D)
$$
where $N^{(n)}_r(d, D)$ is defined in (\ref{eq:exponent}). 
This space is the parameter space of our moduli spaces, and 
for simplicity, from now on,  we set 
\begin{equation}\label{eq:time}
T = M_{g, n, (m_i)} \times_{M_{g, n}}  \cN^{(n)}_r(d, D)  \simeq M_{g, n, (m_i)} \times  N^{(n)}_r(d, D). 
\end{equation}

Let us take  $\bnu=(\nu^{(i)}_j)_{1\leq i\leq n}^{0\leq j\leq r-1} \in 
 N^{(n)}_r(d,D)$ and write $ \nu^{(i)}_j(z_i) $  as in  (\ref{eq:local-exponent})

\begin{equation}\label{eq:local-exponent2}
\nu^{(i)}_j (z_i) =(a^{(i)}_{j, -m_i} z_i^{-m_i} + \cdots + a^{(i)}_{j, -1} z_i^{-1 } )dz_i 
=\sum_{k=-m_i}^{-1} (a^{(i)}_{j, k} z_i^{k}) dz_i  
\quad \mbox{for $ 1 \leq  i  \leq n $}. 
\end{equation}
Let us consider the following decomposition according to the order of 
expansions in (\ref{eq:local-exponent2})
$$ 
N^{(n)}_r(d, D) =N_{top} \times  N_{mid} \times N_{res}
$$
where we set
$N_{top}= \{ (a^{(i)}_{j, -m_i} ), m_i \geq 2 \  \}, 
N_{mid} = \{  (a^{(i)}_{j, k}), -m_i < k < -1,  \ m_i \geq 3 \}, 
N_{res} = \{ (a^{(i)}_{j, -1} ) \}$. 
Using this decomposition, 
for $\bnu \in N^{(n)}_r(d, D)$, we can write as $\bnu = (\bnu_{top}, \bnu_{mid}, \bnu_{res} )$.  
Let us define 
$$
N^{\circ}_{top}=\{ (a^{(i)}_{j, -m_i} ) \  | \ a^{(i)}_{j_1, -m_i} \not= a^{(i)}_{j_2, -m_i}, \mbox{if \ }  j_1 \not=j_2  \}.  
$$ 
Since the genericity condition on $\bnu \in N^{(n)}_r(d, D)$ depends on the part $\bnu_{top}$ (cf. Definition \ref{def:generic}), 
\begin{equation}\label{eq:generic}
N^{\circ} = N^{\circ}_{top} \times N_{mid} \times N_{res} \subset N^{(n)}_r(d, D)
\end{equation}
 is the space of generic generalized exponents. Note that $N^{\circ}$ 
 is an affine open subvariety of  $N^{(n)}_r(d, D)$.    Moreover 
the conditions of resonance and reducibility on $\bnu$ depend just on $\bnu_{res} $(cf. Definition \ref{def:generic}).
Let us denote by $\cP$ the set of formal monodromies $\{ \widehat{\gamma}_i\}$ 
associated to $\bnu_{res}$, which admits a surjective map  by a exponential map
$$
\be:N_{res} \lra \cP,  \quad \{ a^{(i)}_{j, -1} \} \mapsto \{ \widehat{\gamma}_i\}.  
$$ 
(Note that we have the Fuchs relation of $\bnu_{res}$.)

Recall that for $\bnu \in N^{\circ}$, in the previous section, we define the moduli space of generalized monodromy data $\cR(\bnu)$ as in  (\ref{eq:quotient}).

Now we will  see the dependence of isomorphism classes of $\cR(\bnu)$ on $\bnu \in N^{\circ}$.  In order to avoid technical difficulties coming from the multiplicity of the Stokes lines, 
we give the following definition.  
\begin{Definition}\label{def:multi1}{\rm  A generic local exponent  $\bnu =( \nu^{(i)}_j (z_i) )  \in N^{\circ}$ is called simple, if all of the multiplicities of the singular directions of $\bnu$ are one.  We denote by $N^{\circ, s} $  the set of all simple generic  local exponents $\bnu$. 
}
\end{Definition}

Since the singular directions for generic local exponents can be 
determined by $\bnu_{top}$ as in subsection \ref{ss:mon}, we have the following
\begin{Lemma}We can write 
\begin{equation}
N^{\circ, s} 
= N^{\circ, s}_{top}  \times N_{mid} \times N_{res}
\end{equation}
where $N^{\circ, s}_{top}$ consists of  $\bnu_{top}=(a^{(i)}_{j, -m_i} ) \in N^{\circ}_{top}$
with the conditions that for any $i$ and $(j_1, j_2) \not=(k_1, k_2) $
$$ 
 \arg (a^{(i)}_{j_1, -m_i} 
- a^{(i)}_{j_2, -m_i})  \not\equiv  \arg (a^{(i)}_{k_1, -m_i} 
- a^{(i)}_{k_2, -m_i} )  \mbox{ mod $2 \pi \Z$}.  
$$
Note that $N^{\circ, s}_{top}$ is not a Zariski open subset of  $N^{\circ}_{top}$ and 
$N^{\circ, s}_{top}$ may  not be  connected. 
\end{Lemma}

We constructed the moduli 
space $\cR(\bnu)$ of the generalized monodromy data associated to 
the formal type $\bnu$ as in (\ref{eq:quotient}) .   
Since for  $\bnu \in N^{\circ, s}$ every Stokes matrix 
 associated to each singular direction is  one dimensional,   
we can easily see that  the algebraic isomorphism class
of  the affine scheme $\cR(\bnu) $  only depends on $\bnu_{res}$ or 
on $\bp=\be(\bnu_{res})$ for  $\bnu \in N^{\circ, s}$.  
So we may write as $\cR(\bnu) = \cR(\bnu_{res}) = \cR(\bp)$ for $\bnu \in N^{\circ, s}$.
Fix a base element  in each connected component of  $M_{g, n, (m_i)} \times N^{\circ, s}_{top} \times N_{mid} \times \cP$  and fixing the data of singular directions, sectors,  paths and loops 
for it as in the previous section. Varying the data  continuously in each connected component of $M_{g, n, (m_i)} \times N^{\circ, s}_{top} \times N_{mid} \times \cP$,  we can  
construct the family of moduli spaces of generalized monodromy data 
\begin{equation}\label{eq:fam-monod}
\pi_1:{\mathcal R}  \lra  M_{g, n, (m_i)} \times N^{\circ, s}_{top} \times N_{mid} \times \cP 
\end{equation}
such that $\pi_1^{-1}((C, \bt,\{\Psi_i\}),  (\bnu_{top},\bnu_{mid}, \be(\bnu_{res})) \simeq {\mathcal R}(\bnu) = \cR(\bnu_{res})$.   (Note that in order to construct the family (\ref{eq:fam-monod}), we need to consider  the actions of the fundamental groups of the base spaces to 
singular directions and the homotopy classes of  paths and loops in \ref{ss:mon}.)   
Let us fix $\bnu_{res} \in N_{res}$ and set $\bp = \be(\bnu_{res})= 
\{\widehat{\gamma}_i \} \in \cP$.  For simplicity, we set
\begin{equation}\label{eq:space}
  T^{\circ, s}_{\bnu_{res}} = 
M_{g, n, (m_i)} \times N^{\circ, s}_{top} \times 
N_{mid} \times \{ \bnu_{res} \} \subset T^{\circ}=M_{g, n, (m_i)} \times N^{\circ}_{top} \times N_{mid} \times N_{res} 
\end{equation}
Since $   T^{\circ, s}_{\bnu_{res}} \simeq M_{g, n, (m_i)} \times N^{\circ, s}_{top} \times 
N_{mid} \times \{ \bp \}$, restricting the family $\pi_1$ (\ref{eq:fam-monod}) 
to this space,  we obtain the family of moduli spaces 
\begin{equation}\label{eq:fam-mon}
\pi_{1, \bp}: {\mathcal R}_{\bp} \lra   T^{\circ, s}_{\bnu_{res}}
\end{equation}
which is analytically locally constant with the typical fiber $\cR(\bnu_{res}) =
\cR(\bp)$.   
Considering the universal covering map  \small
\begin{equation}\label{eq:univ}
\widetilde{T}^{\circ, s}_{\bnu_{res}}= \tilde{M}_{g, n, (m_i)} \times \tilde{N}^{\circ, s}_{top}\times \tilde{N}_{mid} \times \{ \bnu_{res} \} 
 \lra   T^{\circ, s}_{\bnu_{res}} = M_{g, n, (m_i)} \times N^{\circ, s}_{top} \times N_{mid}\times \{\bnu_{res}\},  
\end{equation}\normalsize
we can pull back the family $\pi_{1, \bp}$ (\ref{eq:fam-mon}) to the family over the universal covering which is isomorphic to the product fibration: 
$$
\tilde{\pi}_{1, \bp}:
\widetilde{\cR}_{\bp} \simeq 
\cR(\bp) \times 
\widetilde{T}^{\circ,s}_{\bnu_{res}} \lra \widetilde{T}^{\circ,s}_{\bnu_{res}}
$$
with the fixed fiber $\cR(\bnu_{res})= \cR(\bp)$. 
On the other hand, by applying Theorems \ref{thm-moduli-exists}
and \ref{smoothness-thm} to the family of $n$-pointed curves
over $M_{g,n,(m_i)}$, there exists the quasi-projective smooth family
of relative moduli spaces
\[
 \pi_2:M^{\balpha}_{D/\cC/M_{g,n,(m_i)}}(r,d,(m_i))
 \longrightarrow T =
 M_{g,n,(m_i)}\times_{M_{g,n}}\cN^{(n)}_r(d,D)
 \cong M_{g,n,(m_i)}\times N^{(n)}_r(d,D).
\]
We denote by $M^{\balpha}_{D/{\mathcal C}/  T^{\circ, s}_{\bnu_{res}}}$
the pull back of
$  T^{\circ, s}_{\bnu_{res}}=M_{g,n,(m_i)}\times N^{\circ, s}_{top}\times N_{mid}\times\{\bnu_{res}\}
\subset T =M_{g,n,(m_i)}\times N^{(n)}_r(d,D)$
by the morphism $\pi_2$.
Then there exists the quasi-projective smooth  family of relative moduli spaces 
\begin{equation}\label{eq:univ-family}
\pi_{2,\bnu_{res}}:M^{\balpha}_{D/{\mathcal C}/  T^{\circ, s}_{\bnu_{res}}}
\lra   T^{\circ, s}_{\bnu_{res}}. 
\end{equation}

Pulling back this family by the universal covering map (\ref{eq:univ}),  
we obtain the  family
$\tilde{\pi}_{2,\bnu_{res}}:
M^{\balpha}_{\tilde{D}/\tilde{\mathcal C}/\widetilde{T}^{\circ, s}_{\bnu_{res}}}
\lra \widetilde{T}^{\circ, s}_{\bnu_{res}} $ of moduli spaces. 

Now take a base point $(C, t_1, \cdots, t_n , \{ z_i \}_{1 \leq i \leq n } , \bnu)$ in each connected component of  $ T^{\circ, s}_{\bnu_{res}}$ and fix a  small neighborhood $\Delta_i$ near each 
$t_i$ and the (simple) singular directions $\{ d_j^{(i)} \}_{1 \leq  j \leq s_i} $ 
and the ordered sectors $\{ S_k^{(i)} \}_{1 \leq k \leq s_i}  $ for each $i$ as in 
\ref{ss:mon}.    Moreover,  fixing  $t_i^{\ast} \in S^{(i)}_1  \cap \partial \Delta_j$ and a base point $b \in C \setminus \{ t_1, \cdots, t_n \}$, we can fix paths  $\{ l_i \}, 
\{ \gamma^{l}, \alpha_k, \beta_k \}$   as in \ref{ss:mon}. 

As explained in Remark \ref{rem:family},  when we vary the data  in $  T^{\circ, s}_{\bnu_{res}}$ or in $ \widetilde{T}^{\circ, s}_{\bnu_{res}}$ starting from each base point, we can vary the choice of sectors, paths and loops continuously.  Hence we can 
define  an analytic morphism 
$$
{\bf RH}_{\bnu_{res}}: M^{\balpha}_{\tilde{D}/\tilde{\mathcal C}/
 \widetilde{T}^{\circ, s}_{\bnu_{res}}} \lra {\mathcal R}(\bnu_{res}) \times 
 \widetilde{T}^{\circ, s}_{\bnu_{res}}
$$  
which makes the following diagram commutative and induces the continuous analytic family of 
Riemann-Hilbert correspondences of fibers of $\tilde{\pi}_{2, \bnu_{res}}$ and 
$ \tilde{\pi}_{1, \bp}$
\begin{equation}\label{eq:RH}
 \begin{array}{ccc}
 M^{\balpha}_{\tilde{D}/\tilde{\mathcal C}/
 \widetilde{T}^{\circ,s}_{\bnu_{res}}}
 & \stackrel{\RH_{\bnu_{res}}}{\lra} & {\mathcal R}(\bnu_{res}) \times 
 \widetilde{T}^{\circ,s}_{\bnu_{res}} \\
&& \\
\quad \quad \downarrow \tilde{\pi}_{2, \bnu_{res}} & & \quad \downarrow \tilde{\pi}_{1, \bp} \\
&& \\
\widetilde{T}^{\circ, s}_{\bnu_{res}} &  =  & \widetilde{T}^{\circ,s}_{\bnu_{res}}. 
 \end{array}
\end{equation}
The analycity of  ${\bf RH}_{\bnu_{res}}$ also follows from the 
result  in \cite{Sibuya-1}.  
Since $ \tilde{\pi}_{2, \bnu_{res}}$ is smooth, we can consider
the natural surjection of tangent sheaves 
\begin{equation}\label{eq:tangent}
\varphi: \Theta_{M^{\balpha}_{\tilde{D}/\tilde{\mathcal C}/
\widetilde{T}^{\circ, s}_{\bnu_{res}}}} \lra \tilde{\pi}_{2, \bnu_{res}}^{\ast} 
(\Theta_{\widetilde{T}^{\circ,s}_{\bnu_{res}}} ) \lra 0. 
\end{equation}
Now one can introduce the (generalized) isomonodromic flows and 
isomonodromic differential systems as follows. 

\begin{Definition} 
{\rm Assume that 
$\bnu_{res}$ is non-resonant and irreducible so that 
${\bf RH}_{\bnu_{res}}$ induces an 
analytic isomorphism between the closed fibers of   $\tilde{\pi}_{1, \bp}$
and $\tilde{\pi}_{2,\bnu_{res}} $ over every closed point of  $\widetilde{T}^{\circ}_{\bnu_{res}}$ by Theorem \ref{thm:Riemann-Hilbert}.
The pull back of  the set of all  
constant sections of $\tilde{\pi}_{1, \bp}$  over  $\widetilde{T}^{\circ,s}_{\bnu_{res}}
(\subset \widetilde{T}^{\circ}_{\bnu_{res}})$ 
 via the Riemann-Hilbert 
correspondence ${\bf RH}_{\bnu_{res}}$ 
gives the set of  horizontal analytic sections 
of  $ \tilde{\pi}_{2, \bnu_{res}} $ in (\ref{eq:RH}) 
which we call the {\em $($generalized$)$ isomonodromic flows}. 
Then the isomonodromic flows define a splitting 
$\tilde{\Psi}: \tilde{\pi}_{2, \bnu_{res}}^{\ast} 
(\Theta_{\widetilde{T}^{\circ,s}_{\bnu_{res}}} ) \hookrightarrow
\Theta_{M^{\balpha}_{\tilde{D}/\tilde{\mathcal C}/
\widetilde{T}^{\circ,s}_{\bnu_{res}}}}$ 
of the surjection (\ref{eq:tangent}) and define the subsheaf 
\begin{equation}\label{eq:split}
\tilde{\btheta}_{\bnu_{res}} = \tilde{\btheta}_{\bp} :=\tilde{ \Psi}( \tilde{\pi}_{2, \bnu_{res}}^{\ast} 
(\Theta_{\widetilde{T}^{\circ, s}_{\bnu_{res}}} )) \subset 
\Theta_{M^{\balpha}_{\tilde{D}/\tilde{\mathcal C}/
\widetilde{T}^{\circ,s}_{\bnu_{res}}}},
\end{equation}
which we call the isomonodromic foliation or the isomonodromic differential system.  It is obvious that the isomonodromic flows become
solution manifolds, or integral manifolds of the differential system
$\tilde{\btheta}_{\bnu_{res}}$.  
The differential system $\tilde{\btheta}_{\bnu_{res}}$ in 
(\ref{eq:split}) is called  {\em the isomonodromic differential system}  associated to the moduli space of $\bnu$-parabolic connections. The parameter space
$\widetilde{T}^{\circ, s}_{\bnu_{res}} =  \tilde{M}_{g, n, (m_i)} \times \tilde{N}^{\circ, s}_{top}\times \tilde{N}_{mid} \times \{ \bnu_{res} \} $ can be  considered as the space of time variables, though some of parameters may be redundant. }
\end{Definition}

Now from the diagram (\ref{eq:RH}), we can descend $\RH_{\bnu_{res}}$ to obtain the following commutative diagram:
\begin{equation}\label{eq:RH-descent}
\begin{array}{ccc}
 M^{\balpha}_{{\mathcal D}/{\mathcal C}/{T}^{\circ, s}_{\bnu_{res}}} 
 & \stackrel{\RH_{\bnu_{res}}}{\lra} & {\mathcal R}_{\bp}  \\
&& \\
\quad \quad \downarrow {\pi}_{2, \bnu_{res}} & &  \quad  \downarrow {\pi}_{1, \bp} \\
&& \\
{T}^{\circ, s}_{\bnu_{res}} &  =  & {T}^{\circ, s}_{\bnu_{res}}
 \end{array}
\end{equation}
By the same reason, we can pull back the locally constant sections of $\pi_{1, \bp}$ by $\RH_{\bnu_{res}}$, and define an isomonodromic flows on 
$
\pi_{2,\bnu_{res}}: M^{\balpha}_{{\mathcal D}/{\mathcal C}/{T}^{\circ, s}_{\bnu_{res}}} 
\lra {T}^{\circ, s}_{\bnu_{res}}$.

Then we can also define the splitting 
\begin{equation}\label{eq:split2}
\Psi: {\pi}_{2, \bnu_{res}}^{\ast} 
(\Theta_{{T}^{\circ, s}_{\bnu_{res}}} ) \hookrightarrow
\Theta_{ M^{\balpha}_{{\mathcal D}/{\mathcal C}/{T}^{\circ, s}_{\bnu_{res}}}},
\end{equation} 
and we can define an analytic foliation 
\begin{equation}\label{eq:isom-flow-alg}
\btheta_{\bnu_{res}} = \btheta_{\bp} := 
\Psi( \pi_{2,\bnu_{res}}^{\ast} (\Theta_{{T}^{\circ, s}_{\bnu_{res}}} ))
\subset \Theta_{M^{\balpha}_{{\mathcal D}/{\mathcal C}/ {T}^{\circ, s}_{\bnu_{res}} } }. 
\end{equation} 

It is natural to consider both isomonodromic differential systems $\btheta_{\bnu_{res}}$ and $\tilde{\btheta}_{\bnu_{res}}$.   Since their integral  manifolds are the isomonodromic flows 
on the corresponding phase spaces, now it is almost trivial to 
see the following theorem as is explained in \cite{Inaba-1}, \cite{IIS-1}.  

\begin{Theorem}\label{thm:GP} Assume that $\bnu_{res}$ is   
non-resonant and irreducible. Then the isomonodromic differential system $\tilde{\btheta}_{\bnu_{res}}$ in (\ref{eq:split}) on the phase space  $M^{\balpha}_{\tilde{D}/\tilde{\mathcal C}/
 \widetilde{T}^{\circ,s}_{\bnu_{res}}}$ satisfies the geometric Painlev\'e property.  Moreover 
the differential system $\btheta_{\bnu_{res}}$ in  (\ref{eq:isom-flow-alg}) on the phase space 
$M^{\balpha}_{{\mathcal D}/{\mathcal C}/{T}^{\circ, s}_{\bnu_{res}}}$  also satisfies the  geometric Painlev\'e property.  
\end{Theorem}

Let us consider the affine variety $T^{\circ}_{\bnu_{res}}$ which contains $T^{\circ, s}_{\bnu_{res}}$ as an analytic dense open set.  Then we have  the following diagram: 
\begin{equation}\label{eq:RH-descent2}
\begin{array}{ccc}
 M^{\balpha}_{{\mathcal D}/{\mathcal C}/{T}^{\circ, s}_{\bnu_{res}}} 
 & \hookrightarrow  &  M^{\balpha}_{{\mathcal D}/{\mathcal C}/{T}^{\circ}_{\bnu_{res}}}  \\
&& \\
\quad \quad \downarrow {\pi}_{2, \bnu_{res}} & &  \quad  \downarrow {\pi'}_{2, \bnu_{res}} \\
&& \\
{T}^{\circ, s}_{\bnu_{res}} & \subset   & {T}^{\circ}_{\bnu_{res}}.
 \end{array}
\end{equation}
Since ${\pi'}_{2, \bnu_{res}}$ is  smooth and algebraic, 
we have a natural surjective homomorphism 
\begin{equation}\label{eq:tangent2}
\varphi: \Theta_{M^{\balpha}_{D/{\mathcal C}/
T^{\circ}_{\bnu_{res}}}} \lra (\pi'_{2, \bnu_{res}})^{\ast} 
(\Theta_{T^{\circ}_{\bnu_{res}}} ) \lra 0.
\end{equation}
Over  the phase space 
$M^{\balpha}_{{\mathcal D}/{\mathcal C}/{T}^{\circ, s}_{\bnu_{res}}} $, 
this is nothing but  the surjection in (\ref{eq:tangent}).  The following 
theorem says that the splitting $\Psi$ in (\ref{eq:split2}) can be extended to
the algebraic splitting 
$\Psi: (\pi'_{2,\bnu_{res}})^*(\Theta_{T^{\circ}_{\bnu_{res}}})
 \longrightarrow \Theta_{M^{\balpha}_{D/\cC/T^{\circ}_{\bnu_{res}}}}$.

\begin{Theorem}\label{thm:algebraic-splitting}
 We can extend the splitting  $\Psi$ in (\ref{eq:split2}) to 
the algebraic splitting 
\begin{equation} \label{eq split-ex}
\Psi:(\pi'_{2,\bnu_{res}})^*(\Theta_{T^{\circ}_{\bnu_{res}}})
 \hookrightarrow
 \Theta_{M^{\balpha}_{D/\cC/T^{\circ}_{\bnu_{res}}}}. 
 \end{equation}
\end{Theorem}

\begin{proof}
Take an affine open subset $U\subset T^{\circ}_{\bnu_{res}}$
and an algebraic vector field $v\in H^0(U,\Theta_{T^{\circ}_{\bnu_{res}}})$.
$v$ corresponds to a morphism
$\iota^v:\Spec{\mathcal O}_U[\epsilon]\rightarrow T^{\circ}_{\bnu_{res}}$,
where $\epsilon^2=0$.
We denote the pullback to $\cC\times\Spec{\mathcal O}_U[\epsilon]$
of the local defining equation of $\tilde{t}_i$ by $\tilde{g}_i$.
We may assume that $\tilde{g}_i|_{m_i\tilde{t}_i}$ is the element
given by $v$.
Consider the composite
\[
 d_{\epsilon}:{\mathcal O}_{\cC\times\Spec{\mathcal O}_U[\epsilon]}
 \stackrel{d}\rightarrow\Omega^1_{\cC\times\Spec{\mathcal O}_U[\epsilon]/U}
 ={\mathcal O}_{\cC\times\Spec{\mathcal O}_U[\epsilon]}d\tilde{g_i}\oplus
 {\mathcal O}_{\cC\times\Spec{\mathcal O}_U[\epsilon]}d\epsilon
 \rightarrow{\mathcal O}_{\cC\times\Spec{\mathcal O}_U[\epsilon]}d\epsilon.
\]
Note that $\epsilon d\epsilon=0$ and so
${\mathcal O}_{\cC\times\Spec{\mathcal O}_U[\epsilon]}d\epsilon
\cong{\mathcal O}_{\cC_U}d\epsilon$.
Let $(\nu^{(i)}_j)+\epsilon(\mu^{(i)}_j)$ be the pullback of the
universal family on $T^{\circ}_{\bnu_{res}}$ by $\iota^v$,
where $d_{\epsilon}(\nu^{(i)}_j)=0$.
There is an \'etale surjective morphism
$V=\coprod_kV_k\rightarrow (\pi'_{2,\bnu_{res}})^{-1}(U)$
such that $V$ is an affine scheme and
there is a universal family $(\tilde{E},\tilde{\nabla},\{\tilde{l}^{(i)}_j\})$
on $\cC_V$.

Take an affine open covering
$\cC_{V_k}=\bigcup_{\alpha}W_{\alpha}$.
After shrinking $V_k$ we may assume that
$\sharp\{\alpha|(\tilde{t}_i)_{V_k}\subset W_{\alpha}\}=1$ for any $i$
and $\sharp\{i|(\tilde{t}_i)_{V_k}\cap W_{\alpha}\neq\emptyset\}\leq 1$
for any $\alpha$.
Take a free module ${\mathcal O}_{W_{\alpha}}[\epsilon]$-module
$E_{\alpha}$ with an isomorphism
$E_{\alpha}\otimes{\mathcal O}_{W_{\alpha}}[\epsilon]/(\epsilon)
\stackrel{\phi_{\alpha}}\rightarrow \tilde{E}|_{W_{\alpha}}$.
Assume that $(\tilde{t}_i)_{V_k}\subset W_{\alpha}$.
We can take a basis $e_0,\ldots,e_{r-1}$ of $E_{\alpha}$ and
$A_{\alpha}\in\End(E_{\alpha})$ such that
$\tilde{\nabla}|_{W_{\alpha}}(e_j)=
\tilde{g}_i^{-m_i}d\tilde{g}_i(A_{\alpha}\otimes{\mathcal O}_U[\epsilon]/(\epsilon))(e_j)$
and
$A_{\alpha}|_{(2m_i-1)\tilde{t}_i}(e_j|_{(2m_i-1)\tilde{t}_i})
=(\tilde{g}_i^{m_i}\nu^{(i)}_j)e_j|_{(2m_i-1)\tilde{t}_i}$
for each $0\leq j\leq r-1$.
We may assume that $d_{\epsilon}(A_{\alpha})=0$.
We can take a matrix
$B_{\alpha}\in\End(E_{\alpha})\tilde{g}_i^{1-m_i}$
such that
$B_{\alpha}|_{(2m_i-1)\tilde{t}_i}(e_j|_{(2m_i-1)\tilde{t}_i})
=(\int\mu^{(i)}_j)e_j|_{(2m_i-1)\tilde{t}_i}$
for each $0\leq j\leq r-1$.
Here note that $\mu^{(i)}_j$ has no residue part and so
$\int\mu^{(i)}_j$ is single valued.
We have
\begin{align*}
 &(A_{\alpha}|_{(2m_i-1)\tilde{t}_i}B_{\alpha}|_{(2m_i-1)\tilde{t}_i}
 -B_{\alpha}|_{(2m_i-1)\tilde{t}_i}A_{\alpha}|_{(2m_i-1)\tilde{t}_i})(e_j|_{(2m_i-1)\tilde{t}_i}) \\
 &=A_{\alpha}|_{(2m_i-1)\tilde{t}_i}((\int\mu^{(i)}_j)e_j|_{(2m_i-1)\tilde{t}_i})
 -B_{\alpha}|_{(2m_i-1)\tilde{t}_i}((\tilde{g}_i^{m_i}\nu^{(i)}_j)e_j|_{(2m_i-1)\tilde{t}_i}) \\
 &=(\int\mu^{(i)}_j)A_{\alpha}|_{(2m_i-1)\tilde{t}_i}(e_j|_{(2m_i-1)\tilde{t}_i})
 -(\tilde{g}_i^{m_i}\nu^{(i)}_j)B_{\alpha}|_{(2m_i-1)\tilde{t}_i}(e_j|_{(2m_i-1)\tilde{t}_i}) \\
 &=(\int\mu^{(i)}_j)(\tilde{g}_i^{m_i}\nu^{(i)}_j)(e_j|_{(2m_i-1)\tilde{t}_i})
 -(\tilde{g}_i^{m_i}\nu^{(i)}_j)(\int\mu^{(i)}_j)(e_j|_{(2m_i-1)\tilde{t}_i})=0.
\end{align*}
This means that
$A_{\alpha}B_{\alpha}-B_{\alpha}A_{\alpha}\in\tilde{g}_i^{m_i}\End(E_{\alpha})$.
We define
\[
 C_{\alpha}:=\tilde{g}_i^{m_i}\genfrac{}{}{}{}{\partial B_{\alpha}}{\partial\tilde{g}_i}
 +A_{\alpha}B_{\alpha}-B_{\alpha}A_{\alpha}\in\End(E_{\alpha}).
\]
Then we have
$(A_{\alpha}+\epsilon C_{\alpha})\tilde{g}_i^{-m_i}d\tilde{g}_i|_{m_i\tilde{t}_i}
(e_j|_{m_i\tilde{t}_i})
=(\nu^{(i)}_j+\epsilon\mu^{(i)}_j)(e_j|_{m_i\tilde{t}_i})$.
We put
\[
 \tilde{A}_{\alpha}:=(A_{\alpha}+\epsilon C_{\alpha})\tilde{g}_i^{-m_i}d\tilde{g}_i
 +B_{\alpha}d\epsilon
\]
and define a connection
$\nabla_{\alpha}:E_{\alpha}\rightarrow E_{\alpha}\otimes\tilde{\Omega}^1$
by
\[
 \nabla_{\alpha}(\sum_{j=0}^{r-1}a_je_j):=\sum_{j=0}^{r-1}da_j\otimes e_j
 +\sum_{j=0}^{r-1}a_j\tilde{A}_{\alpha}(e_j)
\]
for $a_j\in{\mathcal O}_{W_{\alpha}}$,
where $\tilde{\Omega}^1$ is the subsheaf of
$\Omega^1_{\cC_{V_k}\times_U\Spec{\mathcal O}_U[\epsilon]/V_k}(D)$
locally generated by $\tilde{g}_i^{-m_i}d\tilde{g}_i$
and $\tilde{g}_i^{1-m_i}d\epsilon$.
Then $\nabla_{\alpha}$ is a flat connection, that is
$\nabla_{\alpha}\circ\nabla_{\alpha}=0$.
We define a local parabolic structure
$\{(l_{\alpha})^{(i)}_j\}$ by
$(l_{\alpha})^{(i)}_j=\langle e_{r-1}|_{m_i\tilde{t}_i},\ldots,e_j|_{m_i\tilde{t}_i}\rangle$.
So we obtain a triple $(E_{\alpha},\nabla_{\alpha},\{(l_{\alpha})^{(i)}_j\})$
which satisfies
$\nabla_{\alpha}|_{m_i\tilde{t}_i}((l_{\alpha})^{(i)}_j)
\subset(l_{\alpha})^{(i)}_j\otimes\tilde{\Omega}^1$ for any $i,j$ and
$(\tilde{\nabla}_{\alpha}|_{m_i\tilde{t}_i}-(\nu^{(i)}_j+\epsilon\mu^{(i)}_j))((l_{\alpha})^{(i)}_j)
\subset(l_{\alpha})^{(i)}_{j+1}\otimes\Omega^1_{\cC_{V_k}[\epsilon]/V_k[\epsilon]}(D_{V_k}[\epsilon])$
for any $i,j$, where $\cC_{V_k}[\epsilon]=\cC_{V_k}\times_U\Spec{\mathcal O}_U[\epsilon]$,
$D_{V_k}[\epsilon]=D_{V_k}\times_U\Spec{\mathcal O}_U[\epsilon]$ and
$\tilde{\nabla}_{\alpha}$ is the relative connection
induced by $\nabla_{\alpha}$.

We call $({\mathcal E},\nabla_{\mathcal E},\{(l_{\mathcal E})^{(i)}_j\})$
a horizontal lift of $(\tilde{E},\tilde{\nabla},\{\tilde{l}^{(i)}_j\})$ with respect to $v$ if
\begin{itemize}
\item[(1)] ${\mathcal E}$ is a vector bundle on
$\cC_{V_k}\times_U\Spec{\mathcal O}_U[\epsilon]$,
\item[(2)] ${\mathcal E}|_{m_i\tilde{t}_i}=(l_{\mathcal E})^{(i)}_0\supset\cdots\supset
(l_{\mathcal E})^{(i)}_r=0$ is a filtration by subbundle for $i=1,\ldots,n$ and
\item[(3)] $\nabla_{\mathcal E}:{\mathcal E}\rightarrow{\mathcal E}\otimes\tilde{\Omega}^1$
is a connection satisfying
\begin{enumerate}
\item[(a)] $\nabla_{\mathcal E}|_{m_i\tilde{t}_i}((l_{\mathcal E})^{(i)}_j)\subset
(l_{\mathcal E})^{(i)}_j\otimes\tilde{\Omega}^1$ for any $i,j$,
\item[(b)] the curvature $\nabla_{\mathcal E}\circ\nabla_{\mathcal E}:
{\mathcal E}\rightarrow{\mathcal E}\otimes\tilde{\Omega}^2$ is zero,
\item[(c)] $(\tilde{\nabla}_{\mathcal E}|_{m_i\tilde{t}_i}
-(\nu^{(i)}_j+\epsilon\mu^{(i)}_j)\mathrm{id})((l_{\mathcal E})^{(i)}_j)
\subset (l_{\mathcal E})^{(i)}_{j+1}\otimes
\Omega^1_{\cC_{V_k}[\epsilon]/V_k[\epsilon]}
(D_{\cC_{V_k}[\epsilon]})$ for any $i,j$,
where $\tilde{\nabla}_{\mathcal E}$ is the relative connection induced by $\nabla_{\mathcal E}$ and
\item[(d)] $({\mathcal E},\tilde{\nabla}_{\mathcal E},\{(l_{\mathcal E})^{(i)}_j\})
\otimes{\mathcal O}_U[\epsilon]/(\epsilon)\cong(\tilde{E},\tilde{\nabla},\{\tilde{l}^{(i)}_j\})$.
\end{enumerate}
\end{itemize}
Note that $(E_{\alpha},\nabla_{\alpha},\{(l_{\alpha})^{(i)}_j\})$ is a local horizontal lift
and the obstruction class for the existence of a global horizontal lift lies in
$\mathbf{H}^2({\mathcal F}^{\bullet})$, where
\begin{align*}
 {\mathcal F}^0&:=\left\{u\in{\mathcal End}(\tilde{E})\left|
 \text{$u|_{m_i\tilde{t}_i}(\tilde{l}^{(i)}_j)\subset\tilde{l}^{(i)}_j$ for any $i,j$}\right.\right\} \\
 {\mathcal F}^1&:=\left\{u\in{\mathcal End}(\tilde{E})\otimes\overline{\Omega}^1\left|
 \begin{array}{l}
 \text{$u|_{m_i\tilde{t}_i}(\tilde{l}^{(i)}_j)\subset\tilde{l}^{(i)}_j\otimes\overline{\Omega}^1$
 for any $i,j$ and the image of} \\
 \text{$\tilde{l}^{(i)}_j\hookrightarrow \tilde{E}|_{m_i\tilde{t}_i}
 \stackrel{u|_{m_i\tilde{t}_i}}\longrightarrow\tilde{E}|_{m_i\tilde{t}_i}\otimes\overline{\Omega}^1
 \rightarrow\tilde{E}|_{m_i\tilde{t}_i}\otimes\Omega^1_{\cC_{V_k}/V_k}(D_{V_k})$} \\
 \text{lies in $l^{(i)}_{j+1}\otimes\Omega^1_{\cC_{V_k}/V_k}(D_{V_k})$ for any $i,j$}
 \end{array}\right.\right\} \\
 {\mathcal F}^2&:=\left\{u\in{\mathcal End}(\tilde{E})\otimes\tilde{\Omega}^2\left|
 \text{$u|_{m_i\tilde{t}_i}(l^{(i)}_j)\subset\tilde{l}^{(i)}_{j+1}\otimes\tilde{\Omega}^2$
 for any $i,j$}\right.\right\} \\
 d^0&:{\mathcal F}^0\ni u\mapsto
 \tilde{\nabla}\circ u-u\circ\tilde{\nabla}+ud\epsilon\in{\mathcal F}^1 \\
 d^1&:{\mathcal F}^1\ni \omega+ad\epsilon\mapsto
 d\epsilon\wedge\omega+(\tilde{\nabla}\circ a-a\circ\tilde{\nabla})\wedge d\epsilon
 \in{\mathcal F}^2.
\end{align*}
Here $\overline{\Omega}^1=
\Omega^1_{\cC_{V_k}/V_k}(D_{V_k})\oplus{\mathcal O}_{\cC_{V_k}}d\epsilon$.
We can easily check that the complex ${\mathcal F}^{\bullet}$ is exact
and so $\mathbf{H}^2({\mathcal F}^{\bullet})=0$.
So there is a horizontal lift $({\mathcal E},\nabla_{\mathcal E},\{(l_{\mathcal E})^{(i)}_j\})$
of $(\tilde{E},\tilde{\nabla},\{\tilde{l}^{(i)}_j\})$.
(In fact we can see that a horizontal lift is unique because of
$\mathbf{H}^1({\mathcal F}^{\bullet})=0$.)
$({\mathcal E},\nabla_{\mathcal E},\{(l_{\mathcal E})^{(i)}_j\})$
determines an algebraic vector field
$\Psi'(v)\in H^0(V_k,(\Theta_{M^{\balpha}_{D/\cC/T^{\circ}_{\bnu_{res}}}})_{V_k})$.
We can see that $\Psi'(v)$ descends to an algebraic vector field
$\overline{\Psi'}(v)\in
H^0((\pi'_{2,\bnu_{res}})^{-1}(U),\Theta_{M^{\balpha}_{D/\cC/T^{\circ}_{\bnu_{res}}}})$.
By construction we have $\overline{\Psi'}(v)=\Psi(v)$, that is, $\Psi$ is algebraic.
\end{proof}

\begin{Remark}{\rm  The algebraic splitting in (\ref{eq split-ex})  also defines 
an algebraic differential system on the phase space $M^{\balpha}_{D/\cC/T^{\circ}_{\bnu_{res}}}$
\begin{equation}\label{eq:isom-flow-alg-ex}
\btheta'_{\bnu_{res}} = \btheta'_{\bp} := 
\Psi( (\pi'_{2,\bnu_{res}})^{\ast} (\Theta_{{T}^{\circ}_{\bnu_{res}}} ))
\subset \Theta_{M^{\balpha}_{{\mathcal D}/{\mathcal C}/ {T}^{\circ}_{\bnu_{res}} } }. 
\end{equation}
which coincides with $\btheta_{\bnu_{res}} = \btheta_{\bp}$ on $M^{\balpha}_{{\mathcal D}/{\mathcal C}/ {T}^{\circ, s}_{\bnu_{res}} }$.  It seems natural  to expect that  $\btheta'_{\bnu_{res}}$ also satisfies the geometric Painlev\'e property when $\bnu_{res}$ is non-resonant and irreducible, that is, the condition for simpleness for $\bnu$ (or $\bnu_{top}$) may not be necessary.  
If we will fix a nonsimple $\bnu_{top}$ and vary the other data in $T^{\circ}_{\bnu_{res}}$,  
 we can show the geometric Painlev\'e property for the vector fields $\btheta'_{\bnu_{res}}$ 
 from Theorem \ref{thm:Riemann-Hilbert}.   
}
\end{Remark}

Now we show that geometric Painlev\'e property of a  differential system  $\btheta_{\bnu_{res}}$ on  $M^{\balpha}_{{\mathcal D}/{\mathcal C}/ {T}^{\circ, s}_{\bnu_{res}} }$  
implies that the analytic or  classical Painlev\'e property of differential system holds as follows
(cf. \cite{IISA}, \cite{Inaba-1}).  
Assume that on an affine Zariski open subset $U$ of ${T}^{\circ}_{\bnu_{res}}$ and then 
we have algebraic coordinates $T_1, \cdots, T_l$ of $U$  where $l =  l(g, n, (m_i), \bp) = \dim {T}^{\circ}_{\bnu_{res}}$.  Then we may also consider them  as a coordinate system on $U \cap  {T}^{\circ, s}_{\bnu_{res}} $.
Then we can see that the differential system $\btheta'_{\bnu_{res}}$ 
on the phase space $M^{\balpha}_{{\mathcal D}/{\mathcal C}/U}$ over $U$ are generated by 
the following algebraic vector fields 
$$
\btheta'_{\bnu_{res}}= \{ \theta'_1, \cdots, \theta'_l \} \ \mbox{where} \  \theta'_{i}= \Psi( \frac{\partial}{\partial T_i}).  
$$
These vector fields naturally commute to each other and 
by using affine algebraic coordinate charts of $M^{\balpha}_{{\mathcal D}/{\mathcal C}/U}$ we may write these vector fields explicitly and define algebraic partial differential equations on $M^{\balpha}_{{\mathcal D}/{\mathcal C}/U}$. 
Restricting  these vector fields on the phase space
$M^{\balpha}_{{\mathcal D}/{\mathcal C}/U  \cap T^{\circ, s}_{\bnu_{res}}}$ over $U \cap T^{\circ, s}_{\bnu_{res}}$,   we obtain the vector fields $\btheta_{\bnu_{res}}= \{ \theta_1, \cdots, \theta_l \} $ which are equivalent to the isomonodromic flows defined in Theorem \ref{thm:GP}.  Hence  $\btheta_{\bnu_{res}}$ can be written in  partial algebraic differential equations with the independent variables $T_1, \cdots, T_n $.  Since all the solutions of $\btheta_{\bnu_{res}}$ are in the isomonodromic flows,   the solutions stay in the phase space over $U \cap T^{\circ, s}_{\bnu_{res}}$.  This means that all solutions can be arranged in a coordinate chart after the rational transformations of algebraic coordinates of the fibers.  So the movable singularities of the associated differential equations are only poles, which implies the analytic Painlev\'e property.

\begin{Remark}\label{rem:JMU1}{\rm Jimbo, Miwa and Ueno \cite{JMU1} 
gave explicit isomonodromic differential systems in  the 
case of $C= \BP^1$.   }  
\end{Remark}

\begin{Remark}\label{rem:GPP} {\rm Even if  $\bnu_{res}$ is resonant or 
reducible, we can define the Riemann-Hilbert 
correspondence  ${\bf RH}_{\bnu_{res}}$ under the condition 
that $\bnu$ is generic.  We expect that the Riemann-Hilbert 
correspondence  ${\bf RH}_{\bnu_{res}}$ is a {\em proper} surjective bimeromorphic analytic map on each fiber of every closed 
point of $T^{\circ}_{\bnu_{res}}$.
  If we can show this fact, we can 
define an isomonodromic differential system and show  its 
geometric Painlev\'e property. }
\end{Remark}

\subsection{Relations to the classical Painlev\'e equations}
\par
\noindent
\vspace{0.5cm}

Painlev\'e (\cite{P1}, \cite{P2}) and Gambier (\cite{Gambier}) classified the second order rational algebraic ordinary differential equations which may have analytic Painlev\'e property into 6 types, $P_J, J=I, \cdots, VI$.  We call these equations classical 
Painlev\'e equations.  However they did not  give 
the proof of Painlev\'e property for classical Painlev\'e equations. 

Okamoto introduced a one parameter family of  algebraic surfaces associated to 
each type of classical Painlev\'e equation (\cite{O1}) on which 
the Painlev\'e equation has  horizontal separated solutions at least locally.  
A surface appeared as a fiber in the Okamoto's family is called Okamoto's space of initial conditions.  It has a nice compactification $S$, which is a smooth 
rational projective surface, whose anti-canonical divisor $-K_S = Y = 
\sum_{i=1}^{s} n_i Y_i$ is  an effective normal crossing divisor and 
the space of initial conditions can be given as $S \setminus Y_{red}$. 
It satisfies the condition $-K_S\cdot Y_i = Y \cdot Y_i = 0 $ for all 
$i, 1 \leq i \leq s$.  We call such a pair $(S, Y)$ where $S$ is a smooth 
projective rational surface and $Y \in |-K_S|$ with the above condition 
an Okamoto-Painlev\'e pair (\cite{Sakai},\cite{ST},\cite{STT}).  
In \cite{Sakai}, \cite{ST}, \cite{STT}, Okamoto-Painlev\'e pairs $(S, Y)$ are
classified into 8 types corresponding to the affine Dynkin diagrams of types 
$D^{(1)}_k, 4 \leq k \leq 8, E^{(1)}_6, E^{(1)}_7, E^{(1)}_8$. 
Moreover one can show that  such pairs $(S, Y)$  
have a special one parameter deformation and one can derive the 
classical Painlev\'e equations from the special deformations of 
Okamoto-Painlev\'e pairs.  

In \cite{IIS-1} and \cite{IIS-2}, we proved that the Okamoto-Painlev\'e 
pair of type $D^{(1)}_4$ which corresponds to 
Painlev\'e VI equation $P_{VI}$ can be obtained by the 
moduli space of stable $\bnu$-parabolic connections of rank $2$ and 
degree $-1$ over $\BP^1$ with $4$-regular singular points.  
Since it was known that Painlev\'e VI equations can be obtained as 
isomonodromic differential equations, so we can prove
the Painlev\'e property for $P_{VI}$ in \cite{IIS-1}, \cite{IIS-2}.

One can classify types of regular or irregular singularities 
of parabolic connections of rank $2$ on $\BP^1$ whose isomonodromic 
differential equations give the classical Painlev\'e equations of 
8 types (\cite{JMU1}, \cite{vdP-S}). In Table \ref{tab:sing}, 
we list up the type of singularities of linear connections of rank $2$ 
by specifying the orders $m_i$ of singularities at 4 points of $\BP^1$
$i=0, 1, \infty, t \not=0, 1, \infty$. When $m_i= -$, it indicates 
that there are no singularities at the point, and when $m_i$ is 
a half integer, it indicates that the connection has a ramified 
irregular singularity with Katz invariant $m_i -1$.  
Moreover $\cP$ is the space of formal monodromies as in the previous 
subsection.

\begin{table}[h]
\begin{center} 
\begin{tabular}{|c|c||c|c|c|c|c|c|} \hline 
Dynkin & Painlev\'e equations & $m_0$ & $m_1$ & $m_{\infty}$ & $m_{t}$ & 
$\dim \cP$ \\ \hline 
 $D^{(1)}_4$ & $P_{VI}$ &  $1$ & $1$ & $1$ & $1$ & $4$ \\ \hline 
$D^{(1)}_5$ & $P_{V}$ & $1$ & $1$ & $2$ & - & $3$ \\ \hline 
$D^{(1)}_6$  &  deg $P_V$=$P_{III}(D^{(1)}_6)$ & 1 & 1 & 1+1/2 & - & 2  \\ \hline 
$D^{(1)}_6$ & $P_{III}(D^{(1)}_6)$ & 2 & - &  2 & - & 2  \\ \hline  
$D^{(1)}_7$  & $P_{III}(D^{(1)}_7)$ &1+ 1/2& - & 2 & - & 1 \\ \hline 
$D^{(1)}_8$  & $P_{III}(D^{(1)}_8)$ & 1+1/2 & - &1+ 1/2 & - & 0  \\ \hline  
$E^{(1)}_6$ &$ P_{IV}$     & 1& - & 3 & - & 2 \\ \hline
$E^{(1)}_7$ & $P_{II}(FN)=P_{II}$ & 1 & - & 1+3/2 & - & 1\\ \hline  
$E^{(1)}_7$ & $P_{II}$ & - & - & 4 & - & 1 \\ \hline 
$E^{(1)}_8$ & $P_{I}$  &- & - & 1+5/2 & -& 0  \\ \hline 
\end{tabular}
\vspace{0.2cm}
\end{center}
\caption{}
\label{tab:sing}
\end{table}
{F}rom Table \ref{tab:sing}, 
we can see that the following 5 types depending on the parameter 
$\bp \in \cP$ 
are corresponding to the rank $2$ connections with  
regular or unramified  irregular singularities.
  
\begin{equation}\label{eq:URP}
P_{VI}(D^{(1)}_4)_{\bp}, P_{V}(D^{(1)}_5)_{\bp}, P_{III}(D^{(1)}_6)_{\bp}, 
P_{IV}(E^{(1)}_6)_{\bp},  P_{II}(E^{(1)}_7)_{\bp}
\end{equation}

As a corollary of Theorem \ref{thm:GP}, we have the following 

\begin{Theorem}\label{thm:classicalP}
Classical Painlev\'e equations of above 5 types  
in (\ref{eq:URP})  have the geometric Painlev\'e 
property as well as the analytic Painlev\'e property 
if the parameter $\bp \in \cP$ is non-resonant and
irreducible. 
\end{Theorem}

\begin{proof} It is easy to check 
that each classical Painlev\'e equation listed above 
coincides with our isomonodromic flows $\btheta_{\bp}$ 
on a Zariski open set of our family of the moduli space 
of the parabolic connections of the type above
(cf. \cite{JMU1} or \cite{vdP-S}). Then by Theorem \ref{thm:GP} 
classical Painlev\'e equations satisfy Geometric Painlev\'e 
property. 
\end{proof}

\begin{Remark} {\rm In the case of $P_{VI}(D^{(1)}_4)_{\bp}$, 
the geometric Painlev\'e property holds even for 
resonant and reducible parameter $\bp \in \cP$ (cf. \cite{IIS-1}, 
\cite{IIS-2}). Actually, if  all singularities are 
regular, the result of Inaba \cite{Inaba-1} implies that 
the corresponding isomonodromic differential systems $\btheta_{\bp}$ have
geometric Painlev\'e property even for resonant or 
reducible parameter $\bp \in \cP$.} 
\end{Remark}

\begin{Remark}{\rm In \cite{vdP-S}, explicit 
families of connections corresponding to each type in Table \ref{tab:sing} 
are given as well as isomonodromic differential equations for these families.  
However these connections only cover a Zariski dense open set of 
our moduli spaces.  So it is not enough to show the Painlev\'e property 
for classical Painlev\'e equations.  
  Moreover even when $C = \BP^1, d=0$,  
constructions of moduli spaces by using only the trivial bundle 
does not give a whole moduli space of ours because of the existence of 
jumping locus of the bundle type.}  
\end{Remark}

\begin{Remark}{\rm In \cite{vdP-S},  
one can see the all of  explicit equations  
corresponding to the moduli spaces $\cR(\bnu_{res})$  of 
generalized monodromy data for ten types in 
Table \ref{tab:sing}.   These equations are all cubic equations 
in three variables $x_1, x_2, x_3$ with the coefficients in parameters
in $\cP$. In the case of $P_{VI}(D^{(1)}_4)_{\bp}$, the equation is 
given classically  by Fricke-Klein (cf. \cite{IISA}, \cite{vdP-S}).  
} 
\end{Remark}

\end{document}